\documentclass[a4paper,12pt]{amsart}
\usepackage{amsmath}
\usepackage{array}
\usepackage{amsthm}
\usepackage{amssymb}
\usepackage[all]{xy}

\usepackage{hyperref}

\theoremstyle{plain}
\newtheorem*{thm*}{Theorem}
\newtheorem*{prop*}{Proposition}
\newtheorem*{rem*}{Remark}

\newtheorem{thm}{Theorem}[section]
\newtheorem{cor}[thm]{Corollary}
\newtheorem{defi}[thm]{Definition}
\newtheorem{prop}[thm]{Proposition}
\newtheorem{lm}[thm]{Lemma}

\newtheorem{claim*}{Claim}
\newtheorem{exple}[thm]{Example}
\newtheorem{nota}[thm]{Notation}

\numberwithin{equation}{thm}

\theoremstyle{remark}
\newtheorem{rem}[thm]{Remark}

\newcommand{\Y}{\mathcal Y}
\newcommand{\R}{\mathcal R}
\newcommand{\PP}{\mathcal P}

\newcommand{\Z}{\mathbb Z}

\newcommand{\ol}{\overline}

\newcommand{\C}{{\mathcal{C}}}
\newcommand{\D}{{\mathcal{D}}}

\newcommand{\T}{{\mathcal{T}}}

\newcommand{\ti}{{\text{-}}}

\def\co{\colon \thinspace}

\begin{document}

	 \title{Quadratic functors on pointed categories}

		       \author{Manfred Hartl \& Christine Vespa}

\address{Universit\'e de Valenciennes, Laboratoire de Math\'ematiques et de leurs Applications, Valenciennes and FR CNRS 2956, France.}
\email{manfred.hartl@univ-valenciennes.fr}

\address{ Universit\'e Louis Pasteur, Institut de Recherche
  Math\'ematique Avanc\'ee, Strasbourg, France. }                        
\email{vespa@math.u-strasbg.fr}    

\date{\today}

\begin{abstract}
We study polynomial functors of degree $2$, called quadratic, with values in the category of abelian groups $Ab$, and whose source category is  an arbitrary category $\C$ with null object such that all objects are colimits of copies of a generating object $E$ which is small and regular projective; this includes all pointed algebraic varieties. More specifically, we are interested in such
quadratic functors $F$ from $\C$ to $Ab$ which preserve filtered colimits and suitable coequalizers; one may take reflexive ones if $\C$ is Mal'cev and Barr exact. 

A functorial equivalence is established between such functors $F:\C\to Ab$ and
certain minimal algebraic data which we call quadratic $\C$-modules:  these
involve the values on $E$ of the cross-effects of $F$ and  certain structure maps generalizing the second Hopf invariant and 
the Whitehead product.


Applying this general result to the case where $E$ is a cogroup these data take a particularly simple form. This application extends results of Baues and Pirashvili obtained for $\C$ being the category of groups or of modules over some ring; here quadratic $\C$-modules are equivalent with abelian square groups or quadratic $R$-modules, respectively.
\vspace{.3cm}

\textit{Mathematics Subject Classification:} 18D; 18A25; 55U
\vspace{.3cm}

\textit{Keywords}: polynomial functors; quadratic functors; algebraic theory
\end{abstract}
\begin{footnote}{ 
The first author is grateful to the Institut de Recherche Math\'ematique Avanc\'ee for its hospitality during the preparation of this paper, and also to the Institut de G\'eom\'etrie, Alg\`ebre et Topologie de l'\'Ecole Polytechnique F\'ed\'erale de Lausanne for an invitation where this project was started.}
\end{footnote}

\maketitle

In their fundamental work on homology of spaces thereafter linked to their names \cite{EML} Eilenberg and MacLane introduced cross-effects and polynomial functors (see section 1 for definitions). Since then, these functors proved to play a crucial role in unstable homotopy theory; and during the last decade, homological algebra of polynomial functors turned out to be a powerful tool at the crossroad of various fields, such as algebraic K-theory, generic representation theory or cohomology of general linear groups.

In this paper  we determine polynomial functors of degree 2, called quadratic, $F:\T \to Ab$ where $Ab$ is the category of abelian groups and $\T$ is a pointed algebraic theory, i.e.\ a category $\T$ with null object such that all objects are finite coproducts of a generating object $E$. We then extend our results to determine ``good'' quadratic functors on an arbitrary pointed category $\C$ with sums and with a small regular projective generating object $E$, see section 6.5 for definitions; here for brevity we say that a functor is good if it preserves filtered colimits and suitable coequalizers. We point out that one may take $\C$ to be any algebraic variety, with $E$ being the free object of rank $1$.
%



To give a topological example,  the homotopy category of  finite one-point unions $\bigvee_{i=1}^n X$ of copies of  a given space $X$ is a pointed algebraic theory, and metastable homotopy groups on such a category are examples of  quadratic functors if $X$ is a suspension. This was one of the motivations for
Baues and Pirashvili \cite{Baues}, \cite{Baues-Pira}
to study quadratic functors on 
several particular types of algebraic theories. 

In the cited papers, and also in work of the same authors with Dreckmann and Franjou on polynomial functors of higher degree \cite{BDFP}, a functorial equivalence is established   between polynomial  functors on $\T$ and certain minimal algebraic data: these consist of the values of the cross-effects of the corresponding functor on the generating object $E$, and of certain maps relating them. For example,  quadratic  functors from the category of free groups of finite rank to the category of groups correspond to diagrams
$$ (M_e \xrightarrow{H} M_{ee} \xrightarrow{P} M_{e}) $$ 
called \textit{square groups}, where $M_{e}$ is a  group, $M_{ee}$   an abelian group, $H$ a quadratic map and $P$  a homomorphism (satisfying certain relations), see \cite{Baues-Pira}. Those functors taking values in $Ab$ hereby correspond to square groups for which $H$ is linear, i.e. a homomorphism. Similarly, given a ring $R$,  
quadratic functors from the category of finitely generated free $R$-modules to $Ab$ correspond to diagrams of the same type called \textit{quadratic $R$-modules}, but where $M_{e}$ is an abelian  group endowed with a quadratic action of $R$, $M_{ee}$ is an $R\otimes R$-module, and $H$ and $P$  are homomorphisms compatible with the actions of $R$, see \cite{Baues}. As a last example, quadratic functors from the category of finite pointed sets to $Ab$ correspond to diagrams 
$$ (M_{ee}\xrightarrow{T} M_{ee}\xrightarrow{P} M_{e}) $$ 
containing no operator $H$ but an involution $T$ of $M_{ee}$ instead, see \cite{Pira-Dold}.

In this paper (Theorem \ref{thm}) we show that quadratic functors from an \textit{arbitrary} pointed theory $\T$ to $Ab$ (and good quadratic functors on a category $\C$ as above) are functorially equivalent with diagrams
$$M=(T_{11}(cr_2U)(E,E)\otimes_{\Lambda}M_e  \hspace{1mm}\xrightarrow{\hat{H}}  \hspace{1mm}M_{ee}  \hspace{1mm}\xrightarrow{T} \hspace{1mm} M_{ee}   \hspace{1mm}\xrightarrow{P}  \hspace{1mm}M_e)$$
which we call \textit{quadratic $\T$-modules} (resp.\ $\C$-modules), where $T_{11}(cr_2U)$ is the bilinearization of the second cross-effect $cr_2$ of the reduced standard projective functor $U:\T\to Ab$ associated with $E$,
$\Lambda$ is the reduced monoid ring of $\T(E,E)$, $M_e$ and $M_{ee}$ are modules over $\Lambda$ and $\Lambda\otimes \Lambda$, resp., the map $T$ is an involution of $M_{ee}$, and $\hat{H}$ and $P$ are homomorphisms compatible with these structures,  see Definition \ref{quad-C} for details. 
Just as the maps $H$ and $P$ above, the maps $\hat{H}$ and $P$ can be viewed as algebraic generalizations of the second Hopf invariant and the Whitehead product, cf.\ \cite{Baues}. In the quadratic $\C$-module associated with a quadratic functor $F:\T\to Ab$   we have 
 $M_e=F(E)$ and $M_{ee}=(cr_2F)(E,E)$.

Quadratic $\T$-modules can be described as modules over a certain ringoid $\mathcal R$ with two objects; this follows from an alternative approach to our original one which was suggested to us by T.\ Pirashvili: we determine a pair of projective generators of the abelian category $Quad(\T,Ab)$ of quadratic functors from $\T$ to $Ab$, compute the maps between them by using a Yoneda lemma for polynomial functors, thus providing the ringoid $\mathcal R$, and deduce an equivalence of $Quad(\T,Ab)$ with the category of $\mathcal R$-modules from the Gabriel-Popescu theorem. 

We then extend the correspondance between quadratic functors and quadratic $\T$-modules to categories $\C$ with sums and a small regular projective generator, in particular to the algebraic variety $\C=Model(\T)$ of models (or algebras) of the theory $\T$; here $\T$ identifies with the full sub-category of free objects of finite rank in $\C$. For example, the category of groups and the category of algebras over a reduced operad are  algebraic varieties.
 
Such an extension of quadratic functor theory 
 from an algebraic theory to its category of models was established in \cite{Baues-Pira}, in the case where $Model(\T)$ is the category $Gr$ of groups. This is achieved by introducing a \textit{quadratic tensor product} $G\otimes M\in Gr$ for a group $G$ and a square group $M$. We generalize this device to arbitrary categories $\C$ as above, by constructing a quadratic tensor product $-\otimes M \co \C\to Ab$ for any quadratic $\C$-module $M$,
%
%
 and by studying its properties: we compute its effect on $E$ (in fact, $E\otimes M \cong M_e$) and its cross-effect (Theorem \ref{cr-of-otM-Cgen}), and show that it preserves filtered colimits and suitable coequalizers (Theorem \ref{otMpreserves}).


A particularly interesting case arises when
 $E$ has a cogroup structure in $\T$; note that this holds when  $\T$ is the category of finitely generated free algebras over an operad, or the homotopy category of finite one-point unions of a suspension.
In this  cogroup case our above data simplify considerably; notably, the map
$\hat{H}$ splits into two maps
$$ M_e  \hspace{1mm}\xrightarrow{H_1}  \hspace{1mm}M_{ee}  \hspace{1mm}\xleftarrow{H_2}  \hspace{1mm}T_{11}(cr_2\T(E,-))(E,E)\otimes_{\Lambda} coker(P) $$
the first of which generalizes the map $H$ in the cited examples above, while the second one was not visible in these special cases where it is either trivial or determined by the remaining structure. For example, when $\T$ is the category of finitely generated free groups, $H_2$ is equivalent to the map $\Delta=H_1PH_1-2H_1$ in \cite{Baues-Pira}. Moreover,
the involution $T$ here is determined by the remaining structure as $T=H_1P-1$.

Summarizing we may say that in the general cogroup case  a quadratic $\T$-module is a square group $(H_1,P)$ enriched by suitable actions of $\Lambda$ and $\Lambda\otimes \Lambda$ and by an additional structure map $H_2$.

Our result shows that in order to model polynomial functors $F:\T\to Ab$, it is \textit{not} sufficient to just add structure maps of the type $H$ and $P$ between the various cross-effects of $F$ as is suggested by the special cases treated in the literature: the more complicated domain of the map $\hat{H}$ (and its decomposition into  {two} maps in the cogroup case) destroys this ideal picture. On  the other hand, the  map $\hat{H}$ has the interesting structure of a morphism of symmetric $\Lambda\otimes \Lambda$-modules which had not become apparent so far.

We finally note that along our way, we need to provide plenty of auxiliary material which might be of independent interest: we consider and use bilinearization of bifunctors, give more explicit descriptions of the linearization and quadratization of a functor, compute them in a number of cases, in particular for diagonalizable functors, and introduce a notion of quadratic map from a morphism set of a pointed category with finite sums to an abelian group which generalizes the notion of quadratic map from a group to an abelian group in the sense of Passi \cite{Pa} or the first author \cite{QMaps}, and allows to characterize quadratic functors in terms of their effect on morphism sets.

In subsequent work we plan to extend our results to quadratic functors with values in the category of \textit{all} groups where the situation is much more intricate for several reasons. We also expect that our approach generalizes to polynomial functors of higher degree.


\tableofcontents

\section{Polynomial functors}
\subsection{Generalities on polynomial functors and bifunctors}
Throughout this paper, $\C$ denotes a pointed category (i.e. having a null object denoted by $0$) with finite   coproducts denoted by $\vee$. Let $Gr$ and $Ab$ denote the categories of groups and abelian groups, resp.
We begin by giving a definition and basic properties of the cross-effect and of polynomial functors from $\C$ to $Gr$, generalizing those given in \cite{Baues-Pira} for linear and quadratic functors and those given by Eilenberg and Mac Lane in the case of functors from an  abelian category to $Ab$ \cite{EML}. \vskip 1cm

In the sequel, ${\mathcal{D}}$ denotes one of the categories $Gr$ or $Ab$. We consider functors from $\C$ to $\D$. In particular, for $E$ a fixed object of $\C$ we define the universal functor $U_E: \C \to Ab$ as follows. For a set $S$,  let $\mathbb{Z}[S]$ denote the free abelian group with basis $S$. 
Since for all $X \in\C$, $\C(E,X)$ is pointed with basepoint the zero map, we can define a subfunctor $\mathbb{Z}[0]$ of $\mathbb{Z}[\C(E,-)]$ by $\mathbb{Z}[0](X)=\mathbb{Z}[\{E \xrightarrow{0} X \}]$ for $X \in \C$. 
\begin{defi}
The universal functor $U_E: \C \to Ab$ relative to $E$ is the quotient of $\mathbb{Z}[\C(E,-)]$ by the subfunctor $\mathbb{Z}[0]$.
\end{defi}
Note that $U_E$ is the reduced standard projective functor associated with $E$.

To keep notation simple we write $f$ also for the equivalence class in $U_E(X)$ of an element $f$ of $\C(E,X)$, and we often omit the subscript $E$ in $U_E$.

Let $F: \C \to \D$ be a functor. We often note $f_* =F(f)$ for a morphism $f$ in $\C$. For objects $X_1,\ldots,X_n$ of $\C$ and $1\le k\le n$ let
$$ X_k \xrightarrow{i^n_k} X_1 \vee \ldots \vee X_n \xrightarrow{r^n_k} X_k$$
be the canonical injection and retraction, resp., the latter being defined by $r^n_k i^n_k = 1_{X_k}$ and $r^n_k i^n_p = 0$ if $p\neq k$.

\begin{defi}

The $n$-th cross-effect of $F$ is a functor $cr_nF: \C^{\times n} \to \D$ (or a multi-functor) defined inductively by
$$cr_1F(X) = ker(F(0) : F(X) \to F(0))$$
$$cr_2F(X,Y) = ker((F(r_1^2),F(r_2^2))^t : F(X\vee Y) \to F(X)\times F(Y))$$
and, for $n \geq 3$, by
$$ cr_nF(X_1, \ldots, X_n) = cr_2(cr_{n-1}(-,X_3,\ldots,X_n))(X_1,X_2) .$$

\end{defi}

In other words, to define the $n$-th cross-effect of $F$ we consider the $(n-1)$-st cross-effect, we fix the $n-2$ last variables and we consider the second cross-effect of this functor. One often writes $F(X_1 \mid \ldots\mid X_n) =cr_nF(X_1, \ldots, X_n)$.

Note that $F(X) \simeq cr_1F(X) \rtimes F(0)$ as $F(0) : F(0) \to F(X)$ is a natural section of $F(0) : F(X) \to F(0)$.
Moreover, one easily checks by induction that $cr_nF(X_1, \ldots, X_n)$ is a subgroup of $F(X_1 \vee \ldots\vee X_n)$.

In this paper we are mainly interested in \textit{reduced} functors $F: \C \to \D$, that is satisfying $F(0)=0$. We denote by $Func_*(\C, {\mathcal{D}})$ the category of reduced functors $F: \C \to {\mathcal{D}}$. 

There is an alternative description of cross-effects for reduced functors. To state this let  $r^n_{12\dots (k-1)(k+1) \ldots n}  \colon X_1 \vee \ldots \vee X_n \to 
X_1 \vee \ldots \vee \hat{X_k} \vee \ldots \vee X_n$ be the map whose restriction to $X_i$ is its canonical injection for $i\neq k$ and is the zero map if $i=k$.


\begin{prop}
Let $F:\C \to \D$ be a functor. Then the  $n$-th cross-effect \linebreak $cr_nF(X_1, \ldots, X_n)$ is equal to the kernel of the natural homomorphism
$$ \prod_{k=1}^n F(r^n_{12\dots (k-1)(k+1) \ldots n})\hspace{1mm} : \hspace{1mm} F(X_1 \vee \ldots \vee X_n) \hspace{1mm} \xrightarrow{\phantom{aaaa}} \hspace{1mm} \prod_{k=1}
^n F(
X_1 \vee \ldots \vee \hat{X_k} \vee \ldots \vee X_n). $$
\end{prop}

As a consequence, we see that $cr_nF(X_1, \ldots, X_n)$ actually is a \textit{normal} subgroup of $F(X_1 \vee \ldots\vee X_n)$.
Moreover, it follows that $cr_nF(X_1, \ldots, X_n)$ is symmetric in $X_1, \ldots, X_n$. Finally, we see that the functor $cr_n$ is \textit{multi-reduced}, i.e., $cr_nF(X_1, \ldots, X_n)$ vanishes if one of the $X_k$ is the zero object since then $F(r^n_{12\dots (k-1)(k+1) \ldots n})$ is an isomorphism.   

The importance of  cross-effects comes from the following property of functors with values in $Ab$.

\begin{prop} \label{ce-prop} Let $F: \C \to Ab$ be a reduced functor.
Then there is a natural decomposition 
$$F(X_1 \vee \ldots \vee X_n) \simeq \bigoplus_{k=1}^{n} \bigoplus_{1 \leq i_1 < \ldots < i_k \leq n} cr_kF(X_{i_1}, \ldots, X_{i_k}) .$$
\end{prop}

The cross-effects have the following crucial property.
\begin{prop} \label{exact}
The functor $cr_n: Func(\mathcal{C}, \D) \to Func(\mathcal{C}^{\times n} , \D)$ is exact for all $n\ge 1$.
\end{prop}

\begin{proof} For $n=1$ it is a consequence of the natural decomposition $F(X) \simeq cr_1F(X) \rtimes F(0)$. For $n=2$ this follows from the snake-lemma. For higher $n$ use induction.
\end{proof}

\begin{defi} \label{poly} 
A functor $F: \C \to \D$  is said to be \textit{polynomial of degree lower or equal to $n$} if $cr_{n+1}F=0$. Such a functor is called \textit{linear} if $n=1$ and is called \textit{quadratic} if $n=2$. We denote by $Func(\mathcal{C}, \mathcal{D})_{\leq n}$ the full subcategory of $Func(\mathcal{C}, \mathcal{D})$ consisting of polynomial functors of degree lower or equal to $n$.
\end{defi}

The category $Func(\mathcal{C}, \mathcal{D})_{\leq n}$ has the following fundamental property which is an immediate consequence of Proposition \ref{exact}.

\begin{prop} \label{stable}
The category $Func(\mathcal{C}, \D)_{\leq n}$ is  thick i.e. closed under quotients, subobjects and extensions.
\end{prop} 

This is an immediate consequence of Proposition \ref{exact}.

Throughout this paper we denote by $\Delta^n_{\C}: \C \to \C^{\times n}$ the diagonal functor. For $n=2$ we write $\Delta_{\C}$ instead of $\Delta^2_{\C}$.

\begin{defi} \label{+}
For $F \in Func(\mathcal{C},\D)$ and $X \in \mathcal{C}$, we denote by $S_{n}^F$ the natural transformation $S_{n}^F: (cr_{n}F) \Delta^{n}_{\C} \to F$ given by the composition
$$cr_{n}F (X , \ldots , X) \xrightarrow{inc} F( \vee_{i=1}^{n} X) \xrightarrow{F(\nabla^n)} F(X)$$
where $\nabla^n: \vee_{i=1}^{n} X \to X$ is the folding map.

\end{defi}

Note that the image of $S_{n}^F$ is normal in $F(X)$; in fact, $cr_{n}F (X , \ldots , X)$ is normal in $F( \vee_{i=1}^{n} X)$, and $F(\nabla^n)$ is surjective admitting $F(i_1^n)$ as a section. This fact is used in the following:

\begin{defi} \label{Tn-coker}
For $F \in Func(\mathcal{C}, \D)$ the $n-$Taylorisation functor $T_n :  Func(\mathcal{C}, \D) \to  Func(\mathcal{C}, \D)_{\leq n}$ is defined by: $T_nF=coker((cr_{n+1}F) \Delta^{n+1}_{\C}  \xrightarrow{S_{n+1}^F} F).$   We call $T_1$ the linearization functor and   $T_2$   the quadratization functor.
\end{defi}

Let $U_n: Func(\mathcal{C}, \D)_{\leq n} \to Func(\mathcal{C}, \D)$ denote the forgetful (i.e.\ inclusion)      functor.

\begin{prop} \label{Tn-prop}
The $n-$Taylorisation functor $T_n :  Func(\mathcal{C}, \D) \to  (Func(\mathcal{C}, \D))_{\leq n}$ is a left adjoint to $U_n$. The unit of the adjunction is the natural epimorphism $t_n: F \to T_nF$ which is an isomorphism if $F$ is polynomial of degree $\le n$.

\end{prop}

Thus, we obtain the diagram:

$$\xymatrix{
     &&F  \ar[dl]_{t_{n+1}}   \ar[d]_{t_n}  \ar[dr]^{t_{n-1}} &&& \\
    \ldots   \ar[r] & T_{n+1}F \ar[r]^{q_{n+1}}&  T_{n}F \ar[r]^{q_{n}}  &T_{n-1}F \ar[r]^{q_{n-1}}& \ldots \ar[r]&T_1F \ar[r] & T_0F=0.
      }$$

Since the cross-effect $cr_2F$ of a functor  is a bifunctor we need some  general definitions and facts about bifunctors.

\begin{defi}
A bifunctor $B: \C \times \C \to \D$ is said to be bireduced if for all $X \in \C$, $B(X,0)=B(0,X)=0$. We denote by $BiFunc_{*,*}(\mathcal{C} \times \mathcal{C}, \D)$ the category of bireduced bifunctors from $ \C \times \C$ to $\D$.

A bireduced bifunctor $B: \C \times \C \to \D$  is said to be \textit{bipolynomial of bidegree $\leq (n,m)$} if  for all $X \in \C$  the functors $B(-,X), B(X,-): \C \to \D$ are polynomial of degree $\leq n$ and $\leq m$ respectively. We denote by  $BiFunc_{*,*}(\mathcal{C} \times \mathcal{C}, \D)_{\leq (n,m)}$ the category of bipolynomial bifunctors of bidegree $\leq (n,m)$.
\end{defi}

\begin{prop} \label{bi-stable}
The category $BiFunc_{*,*}(\mathcal{C} \times \mathcal{C}, \D)_{\leq (n,m)}$ is thick.
\end{prop}

\begin{defi} Let $B: \C \times \C \to \D$ be a bireduced bifunctor and $n,m\ge 1$. Then the bifunctor $T_{n,m}B: \C \times \C \to \D$ is defined by
$ T_{n,m}B(X,Y) = B(X,Y)/N_1N_2$
where
$$ N_1= im(S_{n+1}^{B(-,Y)}: cr_{n+1}B(-,Y)(X,\ldots,X) \to B(X,Y)),$$
$$ N_2= im(S_{m+1}^{B(X,-)}: cr_{m+1}B(X,-)(Y,\ldots,Y) \to B(X,Y)).$$
For $(n,m)=(1,1)$ we call $T_{1,1}B$
 the bilinearization of $B$.
\end{defi}

Let $U_{n,m}: (BiFunc_{*,*}(\mathcal{C} \times \C, \D))_{\leq (n,m)} \to BiFunc_{*,*}(\mathcal{C} \times \C, \D)$ be the forgetful functor.

\begin{prop}
 The $(n,m)-$Taylorisation functor 
 $$T_{n,m} :  BiFunc_{*,*}(\mathcal{C} \times \C, Ab) \to   BiFunc_{*,*}(\mathcal{C}\times \C, Ab )_{\leq (n,m)}$$
  is a left adjoint to $U_{n,m}$. The unit of this adjunction is given by the natural epimorphism $t_{n,m}\colon B\to T_{n,m}B$.\\

\end{prop}

\noindent {\bf Notation.} For brevity we will often write $T_{11}$ and  $t_{11}$ instead of $T_{1,1}$ and  $t_{1,1}$, resp., and  $\bar{x}$ instead of $t_1(x)$ or $t_{11}(x)$ for $x\in F(X)$ or 
$x\in B(X,Y)$, resp.

\begin{exple}\label{TFoG} For reduced functors $F,G: \C\to Ab$ define the bifunctor $F\boxtimes G : \C\times \C\to Ab$ by $(F\boxtimes G)(X,Y) = F(X) \otimes F(Y)$. Then there is a natural isomorphism
$$ T_{n,m}(F\boxtimes G) \simeq T_nF \boxtimes T_mG .$$
\end{exple}

This is immediate from right-exactness of the tensor product. In the following proposition we give another characterization of the quadratization functor which is useful in the sequel. This requires some notations: in the diagram

\begin{equation} \label{r123}
\xymatrix{
F(X_1 \vee  \ldots \vee X_n ) \ar[rrrr]_-{((r_{12 \ldots (k-1)(k+1) \ldots n}^n)_*, (r^n_k)_*)^t}&&&&  F(X_1 \vee \ldots \vee \hat{X_k} \vee \ldots \vee X_n ) \oplus F(X_k)  \ar@(ul,ur)[llll]_-{((i_{12 \ldots (k-1)(k+1) \ldots n}^n)_*, (i^n_k)_*)}
}
\end{equation}
the map $i_{12 \ldots (k-1)(k+1) \ldots n}^n$ is the obvious injection.

Considering the kernel of $((r_{12 \ldots (k-1)(k+1) \ldots n}^n)_*, (r^n_k)_*)^t$ we obtain the maps:
\begin{equation} \label{r1234}
\xymatrix{
F(X_1 \vee \ldots \vee \hat{X_k} \vee \ldots \vee X_n |X_k) \ar@{^(->}[rrr]_-{\iota_{(12 \ldots (k-1)(k+1) \ldots n,k)}^n}&&& F(X_1 \vee  \ldots \vee X_n )  \ar@(ul,ur)[lll]_-{\rho_{(12 \ldots (k-1)(k+1) \ldots n,k)}^n}.
}
\end{equation}
where $\rho_{(12 \ldots (k-1)(k+1) \ldots n,k)}^n$ is the retraction induced by the section $((i_{12 \ldots (k-1)(k+1) \ldots n}^n)_*, (i^n_k)_*)$ of $((r_{12 \ldots (k-1)(k+1) \ldots n}^n)_*, (r^n_k)_*)^t$.

>From now on, we only consider functors on $\C$ with values in $Ab$.

\subsection{Algebraic theories and polynomial functors}
When $\C$ is an algebraic theory, the polynomial functors from $\C$ to $Ab$ of degree $n$ have the crucial property that they are determined by their values on $n$ objects of $\C$.\medskip

\paragraph{\textbf{Recollections on algebraic theories}}
We here recall and discuss the definition of a pointed  algebraic theory used in this paper and many others by Baues, Jibladze and Pirashvili. 

\begin{defi}

A pointed  (algebraic) theory $\mathcal{T}$ is a pointed category $\mathcal{T}$ with an object $E$ such that any object of $\mathcal{T}$ is isomorphic to a finite sum of copies of $E$. In particular, for any object $E$ of $\mathcal{C}$ we denote by $\langle E \rangle_{\mathcal{C}}$ the theory generated by $E$, i.e. the full subcategory of $\mathcal{C}$ consisting of the objects $E^{\vee n}=E\vee \ldots \vee E$ ($n$ times), $n\ge 0$, with $E^{\vee 0}= 0$. 

\end{defi}

Note that this definition of an algebraic theory is dual to the classical one as being a category encoding algebraic operations, see \cite{Bor2}. Thus a model of a theory $\mathcal{T}$ in our sense is a \textit{contravariant} functor from $\mathcal{T}$ to the category of sets transforming coproducts into products. The advantage of this definition is that here $\mathcal{T}$ identifies with a full subcategory of its category of models, namely the category of free models of $\mathcal{T}$ of finite rank \cite{Bor2}. This allows the quadratic functors we construct in section 5, from data depending only on $\mathcal{T}$, to be naturally defined on the whole category of models of $\mathcal{T}$; indeed, on all of $\C$ in the more general case where $\mathcal{T} = \langle E \rangle_{\mathcal{C}}$.\medskip

\paragraph{\textbf{Polynomial functors on algebraic theories}}
The following property of polynomial functors is crucial in the sequel.
\begin{prop} \label{car-poly}
Let $F, G: \C \to Ab$ be two reduced polynomial functors of degree lower or equal to $n$ and $\phi: F \to G$ be a natural transformation of functors. If $\C=\langle E \rangle_{\C}$, the following statements are equivalent:
\begin{enumerate}
\item
$\phi$ is a natural isomorphism;
\item
$\forall k \leq n$, $\phi_{E^{\vee k}}$ is an isomorphism;
\item
$\phi_{E^{\vee n}}$ is an isomorphism;
\item
$\forall k$ such that  $ 1 \leq k \leq n$, $cr_k(\phi)_{E, \ldots, E}$ is an isomorphism.
\end{enumerate}
\end{prop}

\begin{proof}
Clearly $(1) \Rightarrow (2) \Rightarrow (3)$. The implication $(3) \Rightarrow (4)$ is a consequence of the natural decomposition given in Proposition \ref{ce-prop}. To prove $(4) \Rightarrow (1)$, let  $p \in \mathbb{N}$ and $m=min(p,n)$.  Proposition \ref{ce-prop} provides a natural decomposition:
\begin{eqnarray*}
F(X_1 \vee \ldots \vee X_p) & \simeq& \bigoplus_{k=1}^{p} \bigoplus_{1 \leq i_1 < \ldots < i_k \leq p} cr_kF(X_{i_1}, \ldots, X_{i_k}) \\
&\simeq& \bigoplus_{k=1}^{m} \bigoplus_{1 \leq i_1 < \ldots < i_k \leq p} cr_kF(X_{i_1}, \ldots, X_{i_k})
\end{eqnarray*}

since $F$ is supposed to be polynomial  of degree $n$. Using the analogous decomposition for $G(X_1 \vee \ldots \vee X_p)$ we have:

$$\phi_{X_1 \vee \ldots \vee X_p}  \simeq \bigoplus_{k=1}^{m} \bigoplus_{1 \leq i_1 < \ldots < i_k \leq p} cr_k (\phi)_{X_{i_1}, \ldots, X_{i_k}}.
$$
For $X_1=\ldots= X_p=E$ we deduce that $\phi_{E^{\vee p}}$ is an isomorphism.
\end{proof}

Proposition \ref{car-poly} implies the following analogue for bipolynomial bifunctors.
\begin{cor} \label{car-bipoly}
If $\C=\langle E \rangle_{\C}$, for $B, D: \C \times \C \to Ab$ two bipolynomial bifunctors of bidegree lower or equal to $(n,m)$ and $\phi: B \to D$ a natural transformation of functors.
Then $\phi$ is an natural equivalence if and only if $\phi_{(E^{\vee k}, E^{\vee l})}$ is an isomorphism for all $k \leq n$ and $l \leq m$.
\end{cor}

\subsection{Symmetric bifunctors}
In this section we emphasize a supplementary structure on the cross-effect.

\begin{defi} \label{symbifdef} A symmetric bifunctor from $\C$ to $Ab$ is a pair $(B,T)$ where  $B:\C \times \C \to Ab$ is a bifunctor  and  $T:B \to B\circ V$ is a natural isomorphism.
\end{defi}

Cross-effects are natural examples of symmetric bifunctors, as follows. For $X,Y\in \C$ denote by $\tau_{X,Y}:X\vee Y \to Y \vee X$ the canonical switch.

\begin{prop} \label{crissymbif} Let $F:\C\to Ab$ be a functor. Then 
there are symmetric bifunctors $(cr_2F,T^F)$ and $(T_{11}(cr_2F),\bar{T}^F)$ where
$$T_{X,Y}^F : (cr_2F)(X,Y) \to (cr_2F)(Y,X) \quad \mbox{and} \quad
\bar{T}^F_{X,Y} : T_{11}(cr_2F)(X,Y) \to T_{11}(cr_2F)(Y,X) $$
are given by
$$ T_{X,Y}^F = (\iota^2_{(1,2)})^{-1} F(\tau_{X,Y})(\iota^2_{(1,2)}) \quad \mbox{and} \quad
\bar{T}^F_{X,Y} = T_{11}(T^F),$$
noting that $T_{11}((cr_2F)\circ V)=T_{11}(cr_2F)\circ V$.\hfill$\Box$
\end{prop}


\section{Study of the (bi-)linearization and quadratization functors}

\subsection{(Bi)-linearization and the identity functor of the category of groups}
The following calculations are needed in section 7. Recall the following fact:

\begin{lm}[\cite{Baues-Pira} Lemma 1.6] \label{lin-ab}
For $F: \C \to Gr$ a reduced linear functor, $F(X)$ is an abelian group for $X \in \C$.
\end{lm}

For a group $G$ and $a,b\in G$ let $[a,b]=aba^{-1}
b^{-1}$ and let
$G^{ab}=G/[G,G]$ denote the abelianization of $G$.

\begin{prop} \label{T1Id} 
Let $Id_{Gr} :Gr \to Gr$ be the identity functor. There is 
a natural isomorphism of functors $Gr \to Gr$
$$\Gamma_1 : T_{1}(Id_{Gr})(G) \xleftarrow{\simeq} G^{ab}$$
such that for $g\in G$, one has $\Gamma_1(\bar{g}) = t_{1}(g)$.
\end{prop}

\begin{proof}
The map $\Gamma_1$ is welldefined by Lemma \ref{lin-ab}.
The natural homomorphism $ab: G \to G^{ab}$ factors through $t_1$ followed by a map $\overline{ab}: T_1(Id_{Gr})(G) \to G^{ab}$ since the abelianization functor $G\mapsto G^{ab}$ is linear. It is straightforward to check that $\overline{ab}$ is the inverse of $\Gamma_1$.
\end{proof}

We note that this result generalizes to a natural isomorphism $T_nId_{Gr}(G) \cong G/\gamma_{n+1}(G)$ where $\gamma_{n}(G)$ is the $n$-th term of the lower central series of $G$; this observation is the starting point of  forthcoming work on nilpotent categories.

\begin{prop}\label{T11crId}
Let $Id_{Gr} :Gr \to Gr$ be the identity functor. There is 
a natural isomorphism of bifunctors $Gr\times Gr \to Ab$
$$\Gamma_{11} : T_{11}cr_2(Id_{Gr})(G,H) \xleftarrow{\simeq} G^{ab}\otimes H^{ab} $$
such that for $g\in G$, $h\in H$ one has $\Gamma_{11}(\bar{g}\otimes\bar{h}) = t_{11}([i_1^2g,i_2^2h])$.
\end{prop}

\begin{proof}
Write $G^*=G\backslash\{1\}$, and for a set $E$ let $L(E)$ denote the free group with basis $E$. It is wellknown that there is an isomorphism
$$ \sigma : L(G^*\times H^*) \xrightarrow{\simeq} cr_2(Id_{Gr})(G,H) $$
such that $\sigma(g,h) = [i_1^2g,i_2^2h]$, see \cite{MKS}.
Let $B:Gr\times Gr \to Gr$ denote the bifunctor given by $B(G,H)=L(G\times H)/N\simeq L(G^*\times H^*)$ where $N$ is the normal subgroup generated by  $G\times \{1\} \cup \{1\}\times H$. Let $\pi: L(G \times H) \to B(G , H)$ be the canonical projection. The natural homomorphism $\Gamma'_{11}: 
B(G , H) \to G^{ab}\otimes H^{ab} $ sending $(g,h)$ to $\bar{g}\otimes\bar{h}$ factors through $t_{11}$ followed by a map 
$\bar{\Gamma}'_{11}: 
T_{11}B(G,H) \to G^{ab}\otimes H^{ab} $ since the bifunctor sending $(G,H)$ to $G^{ab}\otimes H^{ab} $ is bilinear.
So it remains to show that $\bar{\Gamma}'_{11}$ is an isomorphism. For this it suffices to check that the map $b:G\times H\to T_{11}B(G,H)$ sending $(g,h)$ to $t_{11}\pi (g,h)$ is bilinear, thus providing an inverse of $\bar{\Gamma}'_{11}$.  To show that $b$ is linear in $h$ consider the map $B(Id_G,\nabla^2): B(G,H\vee H) \to B(G,H)$. One has 
$$ x=\pi((g,h_1h_2) (g,h_2)^{-1} (g,h_1)^{-1})=  
B(Id_G,\nabla^2)(y)$$ with $y=
\pi \left( (g,i_1^2(h_1)i^2_2(h_2)) (g,i^2_2(h_2))^{-1} (g,i^2_1(h_1))^{-1} \right)$. But 
$$B(Id_G,r_1^2)(y) =  \pi((g,h_1)(g,1)^{-1} (g,h_1)^{-1})=1$$
and 
$$B(Id_G,r_2^2)(y) = \pi( (g,h_2)(g,h_2)^{-1}(g,1)^{-1} )=1,$$
 whence $y\in B(G,-)(H\mid H)$. Thus $x \in im(S_2^{B(G,-)})$, whence
$t_{11}(x)=1$. Thus $b$ is linear in $h$. Similarly one shows that $b$ is linear in $g$, as desired.

\end{proof}

\subsection{Linearization and quadratization of diagonalizable functors}
Recall that a reduced functor $F:\C\to \D$ is called \textit{diagonalizable} if there exists a bireduced bifunctor 
$B:\C\times \C\to \D$ such that $F=B\Delta_{\C}$.

\begin{prop} \label{T1BD=0}
The linearization of a diagonalizable functor $F$ is trivial.
\end{prop} 

\begin{proof}
The section $B(i_1^2,i_2^2)$ of the map $B\Delta_{\C}(\nabla^2):
 B(X\vee X,X\vee X) \to B(X,X)$ takes values in $cr_2(B\Delta_{\C})(X,X)$ since $B(r_k^2,r_k^2)B(i_1^2,i_2^2)=0$ for $k=1,2$. Hence $S_2^F$ is surjective and $T_1F=0$.
\end{proof}

We need to related the quadratization of a functor to the bilinearization of its cross-effect.
 Let $F: \C \to Ab$ be a reduced functor. Consider the natural map of bifunctors:
 $$cr_2(t_2): cr_2(F) \to cr_2(T_2F).$$
 Since $T_{11}$ is the left adjoint of the forgetful functor $U: (BiFunc_*(\C,Ab))_{\leq (1,1)} \to BiFunc_*(\C,Ab)$ we obtain that $cr_2(t_2)$ factors through the unit map $t_{11}: cr_2(F) \to T_{11}(cr_2(F))$, thus providing a canonical morphism of bifunctors:
 \begin{equation} \label{cr2t2-1}
 \overline{cr_2(t_2)}: T_{11}(cr_2(F)) \to cr_2(T_2F).
 \end{equation}
 
The following theorem is special case of a more general result in \cite{Manfred-multicalculus}.
 \begin{thm} \label{cr2t2}
The morphism $\overline{cr_2(t_2)}: T_{11}(cr_2(F)) \to cr_2(T_2F)$ is an isomorphism of bifunctors.
\end{thm}

\begin{lm} \label{diago}
If $B: \C \times \C \to Ab$ is a bilinear bireduced bifunctor  then $B \Delta_{\C}: \C \to Ab$ is a quadratic functor.
\end{lm}
\begin{proof}
For $X, Y \in \C$ we have:
$$B  \Delta_{\C}(X \vee Y)=B(X \vee Y, X \vee Y)=B(X,X) \oplus B(Y,Y) \oplus B(X,Y) \oplus B(Y,X)$$
$$=B  \Delta_{\C}(X) \oplus B  \Delta_{\C}(Y) \oplus B(X,Y) \oplus B(Y,X)$$
where the second equality follows from the bilinearity of $B$.
We deduce that $cr_2(B  \Delta_{\C})$ $(X,Y)=B(X,Y) \oplus B(Y,X)$ which is a bilinear functor. So $B  \Delta_{\C}$ is quadratic.
\end{proof}

\begin{prop} \label{lm-bifunc}
For a bireduced bifunctor $B: \C \times \C \to Ab$ we have:
$$T_2(B \Delta)=(T_{11}B) \Delta.$$
\end{prop}

\begin{proof}
Consider the natural map of functors
$$\Delta^*t_{11}: B \Delta \to (T_{11}B) \Delta$$
where $(T_{11}B) \Delta$ is a quadratic functor by Lemma \ref{diago}. Hence $\Delta^* t_{11}$ factors through the quotient map: $\xymatrix{t_2: B \Delta  \ar@{->>}[r] & T_2B \Delta}$  thus providing a canonical morphism $f: T_2(B \Delta) \to (T_{11}B) \Delta$ making the following diagram commutative:
$$\xymatrix{
B \Delta \ar@{->>}[r]_{t_2}   \ar@{->>}@(ur,ul)[rr]^{\Delta^* t_{11}} & T_2(B \Delta) \ar[r]_-{f} & (T_{11}B) \Delta.
}$$
To prove that $f$ is an isomorphism first note that $f$ is an epimorphism since $\Delta^* t_{11}$ is.
To prove that $f$ is a monomorphism it is sufficient to construct a map: $\alpha: (T_{11} B) \Delta \to T_2B \Delta$ such that $\alpha f t_2=t_2$ since $t_2$ is epimorphic.
For $X \in \C$, the map: $B(i_1, i_2): B(X,X) \to B(X \vee X, X \vee X)=B \Delta (X \vee X)$ induces a map $\beta: B(X,X) \to cr_2(B \Delta)(X,X)$ 
such that $\iota_{(1,2)}^2 \beta = B(i_1,i_2)$ 
since $B(r_1,r_1) B(i_1, i_2)= B(r_2, r_2) B(i_1, i_2)=0$ as $B$ is bireduced.. We consider the map $\alpha_X: T_{11}B(X,X)  \to T_2(B \Delta)(X)$ given by the following composition:
$$\xymatrix{
T_{11}B(X,X) \ar[r]^-{T_{11}(\beta)}& (T_{11} cr_2(B \Delta))(X, X) \ar[r]^-{\overline{cr_2(t_2)}_{(X,X)}}& cr_2(T_2B \Delta)(X,X) \ar[r]^-{\iota^2_{(1,2)}} & T_2(B \Delta)(X \vee X) \ar[d]^{T_2(B \Delta)(\nabla)} \\
 & & &T_2(B \Delta)(X)}$$
 where $\overline{cr_2(t_2)}$ is the canonical morphism of bifunctors given in \ref{cr2t2-1}.
 
 By naturality of $t_{11}, t_2$ and by definition of $\overline{cr_2(t_2)}$ we get the following commutative diagram:
 $$\xymatrix{
 B(X,X) \ar[rr]^{(t_{11})_{(X,X)}} \ar[d]_{\beta}&& T_{11}B(X,X) \ar[d]^{T_{11}(\beta)}\\
 cr_2(B \Delta)(X,X) \ar[rr]^{(t_{11})_{(X,X)}} \ar[d]_{\iota^2_{(1,2)}} \ar[rrd]_-{cr_2(t_2)_{(X,X)}}&& (T_{11}cr_2(B \Delta))(X,X) \ar[d]^{\overline{cr_2(t_2)}_{(X,X)}}\\
 B\Delta(X \vee X) \ar[d]_{B \Delta (\nabla)} \ar[rrd]_-{(t_2)_{X\vee X}}&& cr_2(T_2(B \Delta))(X,X) \ar[d]^{\iota^2_{(1,2)}} \\
 B \Delta(X)=B(X, X) \ar[drr]_-{(t_2)_{X}}&&T_2(B \Delta)(X \vee X) \ar[d]^{T_2(B \Delta)(\nabla)} \\
  & &T_2(B \Delta)(X).
 }$$
 Consequently, we obtain:
 
 \begin{eqnarray*}
(\alpha)_X \circ f_X \circ (t_2)_X &=& (\alpha)_X \circ (\Delta^* t_{11})_X\\
&=&(\alpha)_X \circ (t_{11})_{(X,X)}\\
&=&(t_2)_X \circ B \Delta(\nabla) \iota^2_{(1,2)} \beta \mathrm{\quad by\ the\ previous\ commutative\ diagram}\\
&=& (t_2)_XB(\nabla, \nabla) B(i_1, i_2)\\
&=& (t_2)_X \mathrm{\quad since\ } \nabla i_1=\nabla i_2=Id_X.
\end{eqnarray*}

\end{proof}

\subsection{Generalized quadratic maps} \label{sec-quad-maps}
We introduce the notions of cross-effect and quadraticity of maps from morphisms sets of $\C$ to abelian groups, generalizing Passi's notion of quadratic (more generally polynomial) maps between abelian groups. The quadratization functor $T_2$ turns out to provide universal quadratic maps in this sense, and allows to characterize quadratic functors from $\C$ to $Ab$ as being functors whose restriction to each morphism set is a quadratic map.

We start by recalling some elementary facts on polynomial maps between groups.
Let $f:G\to A$ be a function from a group $G$ to an abelian goup $A$ such that $f(1)=0$. Let 

Then $f$ is said to be polynomial of degree $\le n$ if its $\mathbb{Z}$-linear extension $\bar{f}:\mathbb{Z}[G]\to A$ to the group ring $\mathbb{Z}[G]$ of $G$ annihilates the $n+1$-st power $I^{n+1}(G)$ of the augmentation ideal $I(G)$ of $\mathbb{Z}[G]$, or equivalently, if its restriction to $I(G)$ factors through the natural projection $\xymatrix{I(G) \ar@{->>}[r] & P_n(G) := I(G)/
I^{n+1}(G)}$,
see \cite{Pa}. This property can be characterized by using cross-effects of maps \cite{Q3}; in particular, $f$ is quadratic, i.e.\ polynomial of degree $\le 2$, iff its cross effect $d_f:G\times G \to A$, $d_f(a,b)=f(ab)-f(a)-f(b)$, is a bilinear map, see also \cite{QMaps}.

%

Now we generalize this situation, thereby exchanging the definition of a quadratic map via $P_2(G)$ with its  characterization in terms of the cross-effect, as follows. Let $X,Y \in \C$, $A \in Ab$ and $\varphi: \C(X,Y) \to A$ an arbitrary normalized function,which means that $f(0)=0$.

\begin{defi} \label{crmapdef}
The second cross-effect of $\varphi$ is the homomorphism of groups 
$$cr_2(\varphi): U_X(Y \mid Y) \to A$$
defined as follows. Let $\tilde{cr_2}(\varphi): U_X(Y \vee Y) \to A$ given by $\tilde{cr_2}(\varphi)(\xi)=\varphi(\nabla \xi)-\varphi(r_1 \xi)- \varphi(r_2 \xi)$ for $\xi \in \C(X, Y \vee Y)$. Now let 
$cr_2(\varphi) = \tilde{cr_2}(\varphi)\iota^2_{(1,2)}$.

\end{defi}

\begin{defi}\label{qumapdef}
For a bifunctor $B: \C \times \C \to Ab$ and a homomorphism of groups $\psi: B(Y,Y) \to A$ we say that  $\psi$ is bilinear if $\psi$ factors through $t_{11}: B(Y,Y) \to (T_{11}B)(Y,Y)$. Now $\varphi$ is said to be quadratic if its cross-effect $cr_2(\varphi)$ is bilinear.
\end{defi}
A universal quadratic map is provided by the quadratization functor $T_2$; this fact is needed in Lemma \ref{Lambdabarbarmod} below. To state this, let $\overline{\varphi}: U_X(Y) \to A$ be the $\mathbb{Z} \ti$linear extension of $\varphi$.

\begin{prop} \label{univqumap}
The map $\varphi: \C(X,Y) \to A$ is quadratic if and only if $\overline{\varphi}$ factors through $\xymatrix{t_2: U_X(Y) \ar@{->>}[r] & T_2U_X(Y)}$.
\end{prop}

\begin{rem} The above definitions and Proposition indeed formally generalizes the situation for quadratic maps between groups recalled at the beginning of the paragraph. To see this, take $\C=Gr$, $X=\mathbb Z$ and $Y=G$. Using the isomorphism of functors $U_{\mathbb Z} \cong I$ in section 8.1 below, and Propositions \ref{lm-bifunc}, \ref{ideal-lin} and
and \ref{T11crId} one gets an isomorphism of functors ${\Xi}': T_2U_{\mathbb Z} \cong P_2=I/I^3$
such that ${\Xi}'(t_2a) = a-1+I^3(G)$ for $a\in G$. This shows that for a map $\xymatrix{G=Gr(\mathbb Z,G) \ar[r]^-{\varphi} &A}$, the map
$\overline{\varphi}$ factors through $t_2$ iff it is polynomial of degree $\le 2$ in Passi's sense. 
On the other hand, it can be deduced  from Proposition \ref{Ypsilonprop} that $cr_2(\varphi)$  is bilinear in the sense of Definition \ref{qumapdef} iff its group theoretic cross-effect $d_{\varphi}$ is a bilinear map in the usual sense. Thus Proposition \ref{univqumap} here is equivalent with the characterization of polynomial maps of degree $\le 2$ by the bilinearity of their cross effect.
\end{rem}

\begin{rem} It is well known that a functor $F\co \mathcal A \to \mathcal B$ between additive categories $\mathcal A$ and $\mathcal B$ is polynomial of degree $\le n$ iff for all objects $A,B\in \mathcal A$ the map $F\co \mathcal A(A,B) \to \mathcal B(FA,FB)$ is polynomial of degree $\le n$ in the sense of Passi \cite{Pa}; we point out that the latter notion can be described in terms of cross-effects of maps \cite{Q3}. In degree $n=2$, our notion of quadratic map above allows to generalize this fact to functors $F\co \C \to \mathcal B$, as follows: let $\overline{\Z}[\C]$ be the ringoid with the same objects as $\C$ and morphisms $\overline{\Z}[\C](X,Y) = \overline{\Z}[\C(X,Y)] = U_X(Y)$; similarly, let $P_n\C$ be the quotient ringoid of $\overline{\Z}[\C]$ where 
$P_n\C(X,Y) = T_nU_X(Y)$; this category is introduced, in a more general context, by Johnson and McCarthy in \cite{JMCahiers1}. They prove that $F$ is polynomial of degree $\le n$ iff the natural extension of $F$ to an additive functor $\bar{F}\co \overline{\Z}[\C] \to \mathcal B$ factors through the natural quotient functor $\xymatrix{t_n\co \overline{\Z}[\C] \ar@{->>}[r] &P_n\C}$. Now for $n=2$ the latter property is equivalent with all maps $F\co \mathcal \C(X,Y) \to \mathcal B(FX,FY)$, $X,Y\in \C$, being quadratic, by Proposition \ref{univqumap}.

\end{rem}

In order to prove Proposition \ref{univqumap} and also for later use, we need the following description of quadratization functor from the bilinearization functor.

\begin{prop} \label{T2byT11}
For  $F \in Func_{*}(\mathcal{C}, \mathcal{D})$  and $X \in \mathcal{C}$ we have: 
$$T_2(F)=coker\Big(ker\Big(cr_2F(X,X) \xrightarrow{t_{11}} (T_{11}cr_2F)(X ,X)\Big) \xrightarrow{S^F_2} F(X) \Big).$$
\end{prop}
\begin{proof}
We have the following commutative diagram 
$$\xymatrix{
cr_3F(X,X,X) =cr_2(cr_2F(-,X))(X,X) \ar@{^{(}->}[r]^-{\iota^2_{(1,2)}} \ar[rd]_-{S_2^{F(- \mid X)}}& cr_2F(-,X)(X \vee X) \ar@{^{(}->}[r]^-{\iota^3_{(12,3)}} \ar[d]^{F(\nabla^2 \mid 1)}& F(X \vee X \vee X) \ar[d]^{F(\nabla^3)}\\
  & cr_2F(-,X)(X) \ar[r]_-{S_2^F} \ar@{->>}[rd]_{t_1}& F(X)\\
  & & T_1(cr_2F(-,X))(X)
}$$

where the right hand square commutes by the following commutative diagram
$$\xymatrix{
cr_2F(X \vee X,X) \ar[dd]_{F(\nabla^2 \mid 1)} \ar@{^{(}->}[rr] & & F(X \vee X \vee X) \ar[dl]_{F(\nabla^2 \vee 1)} \ar[dd]^{F(\nabla^3)}\\
 & F(X \vee X) \ar[rd]^{F(\nabla^2)}& \\
  cr_2F(X,X) \ar@{^{(}->}[ru]_{\iota^2_{(1,2)}} \ar[rr]_-{S_2^F}& &F(X).\\
  }$$
By Definition \ref{Tn-coker} 
$$T_2F(X)=coker(cr_3F(X,X,X) \xrightarrow{S^F_{3}} F(X))=coker(cr_3F(X,X,X) \xrightarrow{F(\nabla^3) \iota^3_{(12,3)} \iota^2_{(1,2)} } F(X)).$$
We deduce from the first diagram that
\begin{eqnarray} \label{eqn-T2-1}
{}\hspace{15mm}T_2F(X)&=&coker(cr_3F(X,X,X) \xrightarrow{S^F_{2} S_2^{F(- \mid X)}} F(X))\\
&=&coker\Big(ker\left(cr_2F(- ,X)(X) \xrightarrow{t_{1}} T_{1}cr_2F(- ,X)(X)\right) \xrightarrow{S^F_2} F(X)\Big).\label{eqn-T2-12}
\end{eqnarray}
Considering $cr_2F(X,-)$ instead of $cr_2F(-,X)$, we can write down similar commutative diagrams which imply  
\begin{eqnarray} \label{eqn-T2-2}
{}\hspace{15mm} T_2F(X)&=&coker\Big(ker\left(cr_2F(X ,-)(X) \xrightarrow{t_{1}} T_{1}cr_2F(X,-)(X)\right) \xrightarrow{S^F_2} F(X)\Big).
\end{eqnarray}
Combining \ref{eqn-T2-12} and \ref{eqn-T2-2} we deduce the result.
\end{proof}

\begin{proof}[Proof of Proposition \ref{univqumap}]
 For $\xi \in \C(E, E \vee E)$ we have:
\begin{eqnarray*}
cr_2(\varphi) \rho^2_{(1,2)}(\xi)&=& \varphi(\nabla \xi)-\varphi(r_1 \xi) - \varphi( r_2 \xi) -(\varphi(\nabla i_1r_1\xi) - \varphi(r_1i_1r_1\xi)  - \varphi(r_2i_1r_1\xi))
\\
& & {}-(\varphi(\nabla i_2r_2\xi) - \varphi(r_1i_2r_2\xi)  - \varphi(r_2i_2r_2\xi))\\
&=& \varphi(\nabla \xi)-\varphi(r_1 \xi) - \varphi( r_2 \xi)\\
&=& \overline{\varphi} \nabla (\xi- i_1 r_1 \xi-i_2 r_2 \xi)\\
&=& \overline{\varphi} \nabla  \iota^2_{(1,2)} \rho^2_{(1,2)}(\xi)\\
&=& \overline{\varphi} S_2^{U_X}\rho^2_{(1,2)}(\xi).
\end{eqnarray*}
Hence $cr_2(\varphi)=\overline{\varphi} S_2^{U_X}.$ Now, $\varphi$ is quadratic iff $cr_2(\varphi) Ker(t_{11})=0$. But $cr_2(\varphi) Ker(t_{11})=\overline{\varphi} S_2^{U_X} Ker(t_{11})=\overline{\varphi} Ker(t_{2})$ by Proposition \ref{T2byT11}, whence the assertion.
\end{proof}

\subsection{Explicit description of the (bi-)linearization and quadratization functor}

The general principle of this section is to express the values $T_nF(X)$ of the functor $T_nF$  as a cokernel of a map $F(X^{\vee(n+1)}) \to F(X)$ instead of a map $cr_{n+1}F(X, \ldots, X)  \to F(X)$ since the elements of $cr_{n+1}F(X, \ldots ,X) $ are more difficult to describe than the elements of $F(X^{\vee (n+1)})$. 

As a particular case of diagrams \ref{r123} and \ref{r1234} we have the following split short exact sequence:
\begin{equation} \label{ses2}
\xymatrix{ 
0 \ar[r] & F(X|Y) \ar@{^(->}[r]_-{\iota_{(1,2)}^2}& F(X \vee Y)  \ar@(ul,ur)[l]_-{\rho_{(1,2)}^2}
 \ar[r]_-{(r_{1*}^2, r^2_{2*})}& F(X) \oplus F(Y ) \ar@(ul,ur)[l]_-{(i_{1*}^2, i^2_{2*})} \ar[r]& 0
}
\end{equation}
which implies that:
\begin{equation} \label{ses1}
Id_{F(X \vee Y)}={\iota^2_{(1,2)}} \circ \rho_{(1,2)}^2+ i^2_{1*} \circ r^2_{1*}+ i^2_{2*} \circ r^2_{2*}
\end{equation}
so
\begin{equation} \label{Imiota}
Im(\iota^2_{(1,2)})= Im(Id_{F(X \vee Y)}-i^2_{1*} \circ r^2_{1*}- i^2_{2*} \circ r^2_{2*}).
\end{equation} 
Furthermore, we obtain a natural isomorphism of bifunctors
\begin{equation} \label{cr2asquot}
F(X|Y) \simeq F(X \vee Y)/i_{1*}^2F(X) + i_{2*}^2F(Y).
\end{equation}

\subsubsection{Linearization and bilinearization functors}

In the following Proposition we give an explicit description of the linearization functor $T_1$.
\begin{prop} \label{explicit-lin}
For  $F \in Func_{*}(\mathcal{C}, Ab)$  and $X \in \mathcal{C}$ we have: 
$$T_1(F)(X)=coker(F(X \vee X) \xrightarrow{S^F_2 \circ \rho^2_{(1,2)}} F(X) )$$
so
$$T_1F(X)=F(X)/ Im(S^F_2  \circ \rho^2_{(1,2)})=F(X)/\{ \nabla^2
_*(x)-r_{1*}^2(x)-r_{2*}^2(x)|\  x\in F(X \vee X)\}$$
$$=F(X)/\{((1,1)_*-(1,0)_*-(0,1)_*)(x)|\  x\in F(X \vee X)\}$$
\end{prop}
\begin{proof}
The map $\rho^2_{(1,2)}$ is surjective by the short exact sequence \ref{ses2}, hence
$$T_1(F)(X)=coker(F(X \mid X) \xrightarrow{S^F_2} F(X) )=coker(F(X \vee X) \xrightarrow{S^F_2 \circ \rho^2_{(1,2)}} F(X) )$$
and  we have by \ref{ses1}:
$$S^F_2  \circ \rho^2_{(1,2)} =F(\nabla^2) \iota^2_{(1,2)} \rho^2_{(1,2)} =F(\nabla^2)(Id- i^2_{1*} \circ r^2_{1*}-i^2_{2*} \circ r^2_{2*}) $$
$$=\nabla^2_* -r^2_{1*} -r^2_{2*} .$$
\end{proof}

Similarly we obtain:

\begin{prop} \label{bilin-cross1}
For $B \in BiFunc_{*,*}(\mathcal{C} \times \mathcal{C}, Ab)$ and $X, Y \in \C$ we have:
$$T_{11}B(X,Y)=coker\Big(B(X \vee X,Y) \oplus B(X, Y \vee Y) \xrightarrow{(S^{cr_2B(-,Y)}_2 \circ (\rho^2_{(1,2)})^1, S^{cr_2B(X,-)}_2 \circ (\rho^2_{(1,2)})^2)} B(X,Y)\Big)$$
where $(\rho^2_{(1,2)})^1: B(-,Y) (X \vee X) \rightarrow cr_2B(-,Y)(X,X)$ and $(\rho^2_{(1,2)})^2: B(X,-) (Y \vee Y) \rightarrow cr_2B(X,-)(Y,Y)$, so
$$T_{11}B(X,Y)=B(X,Y)/Im((S^F_2 \circ (\rho^2_{(1,2)})^X, S^F_2 \circ (\rho^2_{(1,2)})^Y)$$
$$=B(X,Y)/\{K(x,y)\ | \ x \in B(X \vee X,Y), y \in B(X, Y \vee Y) \}$$
where $K(x,y)= B(\nabla^2, Id)(x)-B(r^2_1,Id)(x)-B(r^2_2,Id)(x)+ B(Id,\nabla^2)(y)-B(Id, r^2_1)(y)-B(Id, r^2_2)(y)$.
\end{prop}

Recall that for $F \in Func_*(\C, Ab)$ we have $cr_2F \in BiFunc_{*,*}(\mathcal{C} \times \mathcal{C}, Ab)$, so we can consider the bilinearization of the bifunctor $cr_2F$.

Applying the previous proposition to $cr_2F$ gives us a description of $T_{11}cr_2F(X,Y)$ as a quotient of $cr_2F(X,Y)$ where the relations are obtained from elements in $cr_2F(X \vee X,Y)$ and $cr_2F(X, Y \vee Y)$. A more manageable description of $T_{11}cr_2F(X,Y)$, as a quotient of $F(X\vee Y)$, is  given as follows: 

\begin{prop} \label{bilin-cross}
For $F \in Func_{*}( \mathcal{C}, Ab)$ and $X, Y \in \C$ we have:
$$T_{11}cr_2F(X|Y)=F(X \vee Y)/ \{ i^2_{1*}(x) + i^2_{2*}(y) + A(z_1) + B(z_2),  |\ x \in F(X), y \in F(Y),$$
$$ z_1 \in F(X \vee X \vee Y), z_2 \in F(X \vee Y \vee Y) \}$$
where 
$$A=F(\nabla^2 \vee Id) -F(r_1^2 \vee Id) - F(r_2^2 \vee Id)$$
and
$$B=F(Id \vee \nabla^2)-F( Id \vee r^2_1)- F(Id \vee r^2_2 ).$$
\end{prop}
\begin{proof}
By Proposition \ref{explicit-lin} we have 
$$T_1(cr_2F(-,Y))(X)=coker(F(\nabla^2 \mid Id) -F(r_1^2 \mid Id)- F(r_2^2 \mid Id)).$$
We obtain the term $A$ from the following commutative diagram where the vertical arrows are isomorphisms of bifunctors by (\ref{cr2asquot}):
$$\xymatrix{
cr_2F(-,Y)(X \vee X) \ar[d]_-{\simeq}   \ar[rrrr]^{F(\nabla^2 \mid Id)-F(r_1^2 \mid Id)- F(r_2^2 \mid Id)} &&&&cr_2F(-,Y)(X) \ar[d]^{\simeq} \\
F((X \vee X) \vee Y ) /(i^3_{12*}F(X \vee X)+i^3_{3*}F(Y))  \ar[rrrr]_-{\overline{F(\nabla^2 \vee Id)}-\overline{F(r_1^2 \vee Id)}-\overline{ F(r_2^2 \vee Id)}} && & &F(X \vee Y)/(i^2_{1*}F(X)+ i^2_{2*}F(Y)) .
}
$$
Similarly, considering $T_1(cr_2F(X,-))(Y)$ we obtain the term $B$.
\end{proof}

\subsubsection{Quadratization functor}
\begin{prop} \label{quadratization}
For  $F \in Func_{*}(\mathcal{C}, Ab)$  and $X \in \mathcal{C}$ we have: 
$$T_2F(X)=F(X)/\{ (\nabla^3_*-(\nabla^2 r^3_{12})_*- (\nabla^2 r^3_{13})_*-(\nabla^2 r^3_{23})_*+r_{1*}^3+r_{2*}^3+r_{3*}^3)(x)   $$
$$   |\  x\in F(X \vee X \vee X) \}$$
$$=F(X)/\{ ((1,1,1)_*-(1,1,0)_*- (1,0,1)_*-(0,1,1)_*+(1,0,0)_*+(0,1,0)_*+(0,0,1)_*)(x) $$
$$      | \ x\in F(X \vee X \vee X) \}.$$
\end{prop}

\begin{proof}
By the proof of Proposition \ref{T2byT11} we have
$$T_2F(X)=coker(cr_3F(X,X,X) \xrightarrow{S^F_2 S^{cr_2F(-,X)}_2} F(X)).$$
By  the following commutative diagram:
$$\xymatrix{
cr_2F(-,X)(X \vee X) \ar[drrrr]^{cr_2F(\nabla^2,Id)} &&&&cr_2(cr_2F(-,X))(X) \ar@{_{(}->}[llll]_{\iota^2_{(1,2)}} \ar[d]^{ S^{cr_2F(-, X)}_2}\\
cr_2F(-,X)(X \vee X) \ar[u]^{Id-cr_2F(i_1^2r_1^2,Id)-cr_2F(i_2^2r^2_2, Id)} \ar@{^{(}->}[d]_-{\iota^3_{(12,3)}}   \ar[rrrr]^{F(\nabla^2 \mid Id)-F(r_1^2 \mid Id)- F(r_2^2 \mid Id)} &&&&cr_2F(-,X)(X) \ar@{^{(}->}[d]^{\iota^2_{(1,2)}} \ar[dr]^{S^F_2}\\
F((X \vee X) \vee X ) \ar@{->>}@(ur,u)[u]_-{\rho^3_{(12,3)}}  \ar[rrrr]_{F(\nabla^2 \vee Id)-F(r_1^2 \vee Id)- F(r_2^2 \vee Id)} && & &F(X \vee X) \ar@{->>}[r]_-{F(\nabla^2)}& F(X)
}
$$
we have
$$\begin{array}{lll}
Im(S^F_2 S^{cr_2F(-, X)}_2)\\
=Im(S^F_2 cr_2F(\nabla^2,Id) \iota^2_{(1,2)})\\
=Im(S^F_2 cr_2F(\nabla^2,Id) \iota^2_{(1,2)} \rho^2_{(1,2)}) \mathrm{\ since \ } \rho^2_{(1,2)} \mathrm{is \ surjective}\\
=Im(S^F_2 cr_2F(\nabla^2,Id)(Id-cr_2F(i_1^2r_1^2,Id)-cr_2F(i_2^2r^2_2, Id))) \mathrm{\ by\ } \ref{ses1}\\
=Im(S^F_2 (F(\nabla^2 \mid Id)-F(r_1^2 \mid Id)- F(r_2^2 \mid Id)))\\
=Im(S^F_2 (F(\nabla^2 \mid Id)-F(r_1^2 \mid Id)- F(r_2^2 \mid Id)) \rho^3_{(12,3)}) \mathrm{\ since \ } \rho^3_{(12,3)} \mathrm{is \ surjective}\\
=Im(F(\nabla^2) (F(\nabla^2 \vee Id)-F(r_1^2 \vee Id)- F(r_2^2 \vee Id)) \iota^2_{(1,2)} \rho^3_{(12,3)}).
\end{array}$$
By the following short exact sequence 
$$\xymatrix{ 
0 \ar[r] & F(X \vee X |Y) \ar@{^(->}[r]_-{\iota_{(12,3)}^3}& F(X \vee X \vee Y)  \ar@(ul,ur)[l]_-{\rho_{(12,3)}^3}
 \ar[r]_-{(r_{12*}^3, r^3_{3*})}& F(X \vee X) \oplus F(Y ) \ar@(ul,ur)[l]_-{(i_{12*}^3, i^3_{3*})} \ar[r]& 0
}$$
we obtain:
$$\begin{array}{ll}
(F(\nabla^2 \vee Id) -F(r_1^2 \vee Id)- F(r_2^2 \vee Id)) \iota^3_{(12,3)} \rho^3_{(12,3)}\\
=(F(\nabla^2 \vee Id) -F(r_1^2 \vee Id)- F(r_2^2 \vee Id)) (Id-i^3_{12*}r^3_{12*}-i^3_{3*}r^3_{3*})\\
=F(\nabla^2 \vee Id) -F(r_1^2 \vee Id)- F(r_2^2 \vee Id)-F(\iota^2_1 \nabla^2r^3_{12})+ F(i^2_1r^3_1)+F(i^2_1r^3_2)\\
\quad -F(i^2_2r^3_3)+F(i^2_2r^3_3)+F(i^2_2 r^3_3)\\
=F(\nabla^2 \vee Id) -F(r_1^2 \vee Id)- F(r_2^2 \vee Id)-F(\iota^2_1 \nabla^2r^3_{12})+ F(i^2_1r^3_1)+F(i^2_1r^3_2)+F(i^2_2 r^3_3).
\end{array}$$
Hence
$$\begin{array}{lll}
Im(S^F_2 S^{cr_2F(-, X)}_2)\\
=Im(F(\nabla^2)(F(\nabla^2 \vee Id) -F(r_1^2 \vee Id)- F(r_2^2 \vee Id)-F(\iota^2_1 \nabla^2r^3_{12})\\
\quad + F(i^2_1r^3_1)+F(i^2_1r^3_2)+F(i^2_2 r^3_3))))\\
=Im(F(\nabla^3)-F(\nabla^2 r^3_{13})-F(\nabla^2 r^3_{23})-F(\nabla^2r^3_{12})+ F(r^3_1)+F(r^3_2)+F(r^3_3)))
\end{array}$$
\end{proof}

\section{Equivalence between polynomial functors and modules over suitable rings}
In this section we give a classification of polynomial functors by modules over suitable rings essentially due to Johnson and McCarthy in \cite{JMCahiers1, JMCahiers2}. Although this provides a classification of polynomial functors of all degrees, it is not satisfactory since the rings that appear are very complicated. So this complete classification does not seem to be manageable for functors of degree higher than 1. Therefore our aim is to describe polynomial functors by minimal data, which is achieved for quadratic functors in this paper. 

\subsection{Adjunction between reduced functors and $\Lambda$-modules}
In this section, we give an adjunction between reduced functors and $\Lambda$-modules which is the starting point of the equivalence between polynomial functors and module  categories given in the sequel.
We begin by the following straightforward lemmas.
\begin{lm}
Composition in $\C$ induces a ring structure on $\Lambda:=U(E)$ and a structure of right $\Lambda$-module on $U(X)$ for any $X \in \C$.
\end{lm}
\begin{lm} \label{F(E)}
For $F : \mathcal C \rightarrow Ab $ a reduced functor, $F(E)$ is a left $\Lambda$-module via 
$$\alpha . x :=  F(\alpha)(x)$$
 for $\alpha \in \mathcal C(E, E)$ and $x \in F(E)$. 
\end{lm}
So, we can give the following definition.
\begin{defi} \label{S}
The functor 
$$\mathbb{S} : Func_* (\mathcal C, Ab) \to \Lambda \ti Mod$$
is defined by $\mathbb{S}(F)=F(E)$ for $F \in Func_* (\mathcal C, Ab)$.
\end{defi}
A left adjoint of $\mathbb{S}$ is provided in the following definition.
\begin{defi} \label{T}
The functor 
$$\mathbb{T} : \Lambda \ti Mod \to Func_* (\mathcal C, Ab)  $$
is defined by $\mathbb{T}(M)(X)=U(X) \otimes_{\Lambda}M$ for $M \in \Lambda \ti Mod$.
\end{defi}
\begin{prop} \label{ST}
The functor $\mathbb{T} $ is a left adjoint of $\mathbb{S}$.

The \textbf{unit} of this adjunction is the canonical isomorphism
$$u_M: M \overset{\cong}{\rightarrow} \Lambda \otimes_ {\Lambda} M = {\mathbb S} {\mathbb T} (M) \mathrm{\qquad for\ } M \in {\Lambda} \ti Mod.$$

The \textbf{co-unit} is
$$(u'_F)_{X} : \mathbb T \mathbb S(F)(X) = U(X)\otimes_{\Lambda} F(E) \rightarrow F(X),$$
$\mathrm{ where\quad} (u'_F)_{X} (f \otimes x) =  F(f)(x)\; for \; f \in \mathcal C(E, X), \mathrm{\ and\ } x\in F(E).$
\end{prop}
We consider $u'_F$ as a first order approximation of $F$;  if $F$ is polynomial of degree $n$ then $u'_F$ may be reduced to a morphism
$$\overline{u'_F} : T_n (\mathbb T \mathbb S (F)) \rightarrow F.$$
This turns out to be an isomorphism for $n=1$ but is not for $n >1$. So our approach to polynomial functors consists of inductively improving the approximation $\overline{u'_F}$ in order to get an isomorphism again, by taking into account higher and higher cross-effects.

\subsection{Classification of linear functors}
Let $Lin(\C, Ab)$ denote the category of linear reduced functors from $\C$ to $Ab$. In this section we show that if $\mathcal{C}$ is a pointed theory then the category $Lin(\C, Ab)$ is equivalent to the category of modules over a suitable ring.

We begin by providing a number of equivalent characterizations of linear functors.

\begin{lm} \label{lem1}  
Let $F :  \mathcal C \rightarrow Ab $ be a reduced functor. Then the following conditions are equivalent
\begin{enumerate}
\item $F$ is linear;
\item $S^F_2 = 0$ where $S^F_2: cr_2(F)  \Delta_{\C} \to F$ is defined in Definition \ref{+};
\item For $X, \, Y \in \mathcal C$ one has $$1_{F(X \vee Y)} = i^2_{1*} r^2_{1*} + i^2_{2*} r^2_{2*};$$
\item For $X, \, Y \in \mathcal C$ and $\xi \in \mathcal C (X, \, Y \vee Y)$ one has
$$F(\nabla^2 \, \xi) = F(r^2_1 \, \xi) + F(r^2_2 \, \xi).$$
\end{enumerate}
\end{lm}

\begin{proof}
By definition, $F$ is linear if $cr_2(F)=0$. Since $S^F_2: cr_2(F) \Delta_{\C} \to F$ we have $S^F_2=0$. Conversely, if $S^F_2=0$, by Proposition \ref{Tn-prop}, $F \simeq T_1F$ so $F$ is linear and we proved that $(1) \Leftrightarrow (2)$.

By the short exact sequence (\ref{ses2}) we have $(1)  \Leftrightarrow (3)$.

By Proposition \ref{Tn-coker}, $F$ is linear if and only if $F \simeq T_1F$, so by Proposition \ref{explicit-lin} this is equivalent to, $\forall y \in F(Y \vee Y)$:
$F(\nabla^2)(y)=F(r_{1}^2)(y)+F(r_{2}^2)(y).$
Applying the last equality to $y=F(\xi)(x)\in F(Y \vee Y)$ where $\xi \in \C(X, Y \vee Y)$ and $x \in F(X)$, we obtain $(1)  \Leftrightarrow (4)$.
\end{proof}

\begin{prop} \label{Lambda-bar}
The abelian group $\overline{\Lambda}:=(T_1U)(E)$ is a ring and $T_1U(X)$ has a right $\overline{\Lambda}$-module structure such that $t_1$ is $\Lambda$-equivariant (i.e. for $\lambda \in \Lambda$ and $x \in U(X)$, $t_1(x \lambda)=t_1(x) t_1(\lambda)$.)
\end{prop}

\begin{proof}
For $f \in \C(E,E)$  the relation $(S^U_2)_E cr_2U(f,f)=U(f)(S^U_2)_E$ shows that $Im((S^U_2)_E)$ is a  left ideal of $\Lambda$,  and we deduce from  the  following commutative diagram that  $Ker(t_1)=Im((S^U_2)_E)$ is a  right ideal, too:
$$\xymatrix{
U(E) \ar[r]^{\alpha^*} \ar[d]_{t_1}& U(E)  \ar[d]^{t_1}\\
T_1U(E) \ar[r]_{T_1(\alpha^*)} & T_1U(E)
}$$
where $\alpha:E \to E$. Consequently $\overline{\Lambda}$ is a ring.

For $X \in \C$, $T_1U(X)$ is a right  $\overline{\Lambda}$-module via
$$t_1(f).t_1(\alpha)=t_1(f \alpha)$$
for $f \in \C(E,X)$ and $\alpha \in \C(E,E)$; this is welldefined again by naturality of $S^U_2$.
\end{proof}

\begin{lm} \label{F(E)-lin}
For $F : \mathcal C \rightarrow Ab $ a linear reduced functor, $F(E)$ is a left $\overline{\Lambda}$-module via 
$$\overline{\alpha} . x :=  F(\alpha)(x)$$
 for $\alpha \in \mathcal C(E, E)$ and $x \in F(E)$. 
\end{lm}

\begin{proof}
By Proposition \ref{explicit-lin} we have
$$\overline{\Lambda}=\Lambda/ \{ \nabla^2\xi-r_1^2\xi-r^2_2\xi\ |\  \xi \in \C (E, E \vee E) \} .$$
But
$$( \nabla^2\xi-r_1^2\xi-r^2_2\xi).x=\nabla^2\xi.x-r_1^2\xi.x-r^2_2\xi.x=F(\nabla^2\xi)(x)-F(r_1^2\xi)(x)-F(r^2_2\xi)(x)$$
$$=(F(\nabla^2\xi)-F(r_1^2\xi)-F(r^2_2\xi))(x)=0$$
by Lemma \ref{lem1} (4).

\end{proof}

This leads to the following definition.
\begin{defi} \label{S1}
The functor 
$$\mathbb{S}_1 : Lin(\mathcal C, Ab) \to \overline{\Lambda}-Mod$$
is defined by $\mathbb{S}_1(F)=F(E)$ for $F \in Lin(\mathcal C, Ab)$.
\end{defi}

\begin{defi} \label{T1}
The functor 
$$\mathbb{T}_1 : \overline{\Lambda}-Mod \to Lin(\mathcal C, Ab)  $$
is defined by $\mathbb{T}_1(M)(X)=T_1U(X) \otimes_{\overline{\Lambda}}M$ for $M \in \overline{\Lambda}-Mod$.
\end{defi}
The following proposition connects the functors $\mathbb{S}_1$ and $\mathbb{T}_1$.
\begin{prop} \label{S1T1}
The functor $\mathbb{T}_1 $ is the left adjoint of $\mathbb{S}_1$.

The \textbf{unit} of this adjunction is the canonical isomorphism
$$u_M: M \overset{\cong}{\rightarrow} \overline{\Lambda} \otimes_ {\overline{\Lambda}} M = {\mathbb S}_1 {\mathbb T}_1 (M) \mathrm{\qquad for\ } M \in {\overline{\Lambda}}-Mod.$$

The \textbf{co-unit} is
$$(u'_F)_{X} : \mathbb{T}_1\mathbb{S}_1(F)(X) = T_1U(X)\otimes_{\overline{\Lambda}} F(E) \rightarrow F(X),$$
$\mathrm{ where\quad} (u'_F)_{X} (t_1(f) \otimes x) =  F(f)(x)\; for \; f \in \mathcal C(E, X), \mathrm{\ and\ } x\in F(E).$
\end{prop}

A classification of linear functors is now given as follows:

\begin{thm} \label{lin}
If $\C=\langle E \rangle_{\C}$ then the functors ${\mathbb S}_1$ and  ${\mathbb T}_1$ form a pair of adjoint equivalences. 
\end{thm}

\begin{proof}
It remains to show that the counit is an isomorphism. For $F \in Lin(\mathcal{C}, Ab) $ the source and target functors of $u'_F$ are linear, so it suffices by Proposition \ref{car-poly} to prove that $(u'_F)_E: \overline{\Lambda} \otimes_{\overline{\Lambda} } F(E) \to F(E)$ is an isomorphism. We have $(u'_F)_E(t_1(f) \otimes x)=F(f)(x)=t_1(f).x$ so $(u'_F)_E$ coincides with the canonical isomorphism.
\end{proof}

\subsection{Classification of polynomial functors of arbitrary degree}

Theorem \ref{lin} can be generalized to polynomial functors of arbitrary degree; this can be deduced from a more general result of Johnson and McCarthy on polynomial functors with values in categories of chain complexes \cite{JMCahiers1, JMCahiers2}, by identifying the category $Ab$ with the category of chain complexes concentrated in a given degree. 

\begin{thm} \label{JMcCthm}
Suppose that $\C=\langle E \rangle_{\C}$. Then the group $T_nU_{E^{\vee n}}(E^{\vee n})$ has a ring structure induced by composition in $\C$, and there is an equivalence of categories
$$ {\mathbb S}_n : Func(\C,Ab)_{\le n} \to T_nU_{E^{\vee n}}(E^{\vee n})-Mod$$
given by ${\mathbb S}_n(F)=F(E)$.
\end{thm}

This result generalizes a similar one for additive categories $\C$ due to Pirashvili \cite{Pira}.
A simple direct proof of Theorem \ref{JMcCthm} is given in \cite{Manfred-multicalculus}. Further study of polynomial functors could thus be based on this theorem in analyzing the - very complicated  - structure of the rings
$T_nU(E)$ and their representations; instead, we follow the basic idea of the work of Baues and Pirashvili:
according to Proposition \ref{car-poly} a polynomial functor $F$ of degree $n$ is determined by the values of its cross-effects $cr_k(F)$ on $(E, \ldots, E)$, $1\le k\le n$; so we seek for minimal extra structure relating them so as to make the correspondance between 
polynomial functors and these enriched cross-effects into a functorial equivalence.

Inspired by the paper \cite{Pira-russe} we observe:

\begin{thm}\cite{Pira-russe} \label{thm-Pira}
For $\C=\langle E \rangle_{\C}$, the family of functors $\{ T_n U_{E^{\vee k}}, k \in \{1, \ldots, n \} \}$ is a family of small projective generators of the category $Func(\C, Ab)_{\leq n}$.
\end{thm}

This follows from the following proposition together with Proposition \ref{ce-prop}.

\begin{prop}[Yoneda lemma for polynomial functors] \label{Yoneda-poly}
Let $\C$ be a pointed category, 
$F: \C \to Ab$ a  polynomial functor of degree lower or equal to $n$. Then for $X \in Ob(\C)$ we have an isomorphism:
$$\Y: Hom_{Func(\C, Ab)_{\leq n}}(T_nU_X,F) \xrightarrow{\simeq} F(X)$$
natural in $F$ and $X$, 
defined by $\Y(\varphi) = \varphi_X(1_X)$.
\end{prop}

This is an immediate consequence of Proposition \ref{Tn-prop} and the additive Yoneda lemma.

\begin{rem} Pirashvili's Yoneda lemma  for polynomial functors in \cite{Pira-russe} treats the case where $\C$ is a category of modules over some ring $R$; instead of the term $T_nU_X$ in Proposition \ref{Yoneda-poly} it contains the term $P_n({\rm Hom}_R(X,-))$, cf.\ section 2.3 for the definition of Passi's functor $P_n$. In fact, a generalization of   results in section 8 below shows that if $X$ has a cogroup structure in $\C$, then the functors $T_nU_X$ and $P_n\circ \C(X,-)$ are isomorphic.

\end{rem}

\subsection{The case of bifunctors}
Since cross-effects of quadratic functors are bilinear we need an analogue of Theorem \ref{lin} for bifunctors which goes as follows. Let Bilin $(\mathcal C \times \mathcal C, Ab )$ denote the category of bilinear bireduced bifunctors from $\mathcal C \times \mathcal C$ to $Ab $.

We begin by the following lemma.

\begin{lm} \label{B(E,E)-lin}
For $B : \mathcal C \times \C \rightarrow Ab $ a bilinear bireduced bifunctor, $B(E,E)$ is a left $\overline{\Lambda} \otimes \overline{\Lambda} $-module via 
$$(\overline{\alpha} \otimes \overline{\beta}). x :=  B(\alpha, \beta)(x)$$
 for $\alpha , \beta \in \mathcal C(E, E)$ and $x \in B(E,E)$. 
\end{lm}

\begin{rem}
Note that $T_1U(X) \otimes T_1U(Y) $ is a right $\Lambda \otimes \Lambda$-module via $(t_1(x) \otimes t_1(y))(\alpha \otimes \beta)=t_1(x \alpha) \otimes t_1(y \beta)$; this induces a structure of right $\overline{\Lambda} \otimes \overline{\Lambda}$-module.
\end{rem}

\begin{defi} \label{S11}
The functor 
$$\mathbb{S}_{11} : Bilin(\mathcal C \times \mathcal C, Ab) \to \overline{\Lambda} \otimes \overline{\Lambda}-Mod$$
is defined by $\mathbb{S}_{11}(B)=B(E,E)$ for $B \in Bilin(\mathcal C \times \mathcal C, Ab) $.
\end{defi}

\begin{defi} \label{T11}
The functor 
$$\mathbb{T}_{11} :   \overline{\Lambda} \otimes \overline{\Lambda}-Mod \to Bilin(\mathcal C \times\mathcal C, Ab)  $$
is defined by $\mathbb{T}_{11}(M)(X,Y)=(T_1U(X) \otimes T_1U(Y) )\otimes_{\Lambda \otimes \Lambda}M$ for $M \in \overline{\Lambda} \otimes \overline{\Lambda}-Mod$.
\end{defi}
The following proposition connects  the functors $\mathbb{T}_{11}$ and $\mathbb{S}_{11}$.

\begin{prop} \label{S11T11}
The functor $\mathbb{T}_{11} $ is the left adjoint of $\mathbb{S}_{11}$.

The \textbf{unit} of this adjunction is  the canonical isomorphism
$$u_M: M \to (T_1U(E) \otimes T_1U(E)) \otimes_{\overline{\Lambda} \otimes\overline{\Lambda}}M$$
defined by $u_M(m)=(t_1(1) \otimes t_1(1)) \otimes m$ where $M \in \overline{\Lambda} \otimes \overline{\Lambda}-Mod$.

The \textbf{co-unit} is
$$(u'_B)_{X,Y} : \mathbb{T}_{11}\mathbb{S}_{11}(B)(X,Y) =(T_1U(X) \otimes T_1U(Y) )\otimes_{\overline{\Lambda} \otimes \overline{\Lambda}} B(E,E) \rightarrow B(X,Y),$$
$\mathrm{ where\quad} (u'_B)_{X,Y} (t_1(f) \otimes t_1(g)\otimes x) =  B(f,g)(x)\; for \; f \in \mathcal C(E, X),  g \in \mathcal C(E, Y)\mathrm{\ and\ } x\in B(E,E).$
\end{prop}

The arguments in the proof of Theorem \ref{lin} are easily adapted to obtain: 

\begin{thm} \label{bilin}
If $\C=\langle E \rangle_{\C}$ then the functors ${\mathbb S}_{11}$ and  ${\mathbb T}_{11}$ form a pair of adjoint equivalences. 
\end{thm}

By Proposition \ref{crissymbif} cross-effects are not only bifunctors but symmetric bifunctors. In the following we exploit this supplementary structure.
We begin by recalling the definition of a symmetric $R \otimes R$-module.

\begin{defi}
For a ring $R$, a symmetric $R \otimes R$-module is a left $R \otimes R$-module $M$ equipped with a $\Z$-linear involution $T$ (i.e. $T^2=Id$) such that for $r,s \in R$ and $m \in M$
$$T((r \otimes s)m)=(s \otimes r)T(m).$$
A morphism of symmetric $R \otimes R$-modules is a morphism of $R \otimes R$-modules compatible with the respective involutions.
\end{defi}

\begin{rem} A symmetric $R \otimes R$-module is the same as a module over the wreath product
$$(R \otimes R) \wr \mathfrak{S}_2 = (R \otimes R) \oplus (R \otimes R).t$$
whose multiplication is defined by
$$(r_1\otimes r_2+(s_1\otimes s_2).t)(r_1'\otimes r_2'+(s_1'\otimes s_2').t) $$
$$= (r_1r_1' \otimes r_2r_2' +
s_1s_2' \otimes s_2s_1' ) + (r_1s_1' \otimes r_2s_2' + s_1r_2' \otimes s_2r_1' ).t
$$
for $r_i,r_i',s_i,s_i'\in R$ and where $t$ denotes the generator of $ \mathfrak{S}_2$.
\end{rem}

Symmetric $R \otimes R$-modules naturally arise from evaluating symmetric bifunctors. Let $V:\C \times \C\to \C \times \C$ be the canonical interchange functor, $V(X,Y)=(Y,X)$.

\begin{prop}\label{symbifmod} Let $(B,T)$ be a symmetric bifunctor from $\C$ to $Ab$ such that $B$ is bireduced. Then for $E\in \C$, the group $B(E,E)$ is a symmetric $\Lambda\otimes \Lambda$-module where $(f\otimes g)x=B(f,g)(x)$ for $f,g\in \C(E,E)$ and $x\in B(E,E)$, and with involution $T_{E,E}$.\hfill$\Box$
\end{prop}

We thus obtain the following two examples of symmetric $R \otimes R$-modules for suitable $R$ which are important in section \ref{section-4}. 

\begin{cor} \label{ThochF}
For a reduced functor $F: \C \to Ab$, the group $F(E|E)$ admits the structure of  a symmetric $\Lambda \otimes \Lambda$-module such that for $f,g \in \C(E,E)$ and $\alpha \in F(E|E)$ we have
$$(f \otimes g). \alpha:= F(f | g)(\alpha) \in F(E |E)$$
and with involution:
$$T^F(\alpha)= (\iota^2_{(1,2)})^{-1} F(\tau_{E,E})  \iota^2_{(1,2)}(\alpha)  \in F(E |E).$$
\end{cor}

\begin{cor} \label{T11-sym}
For $F: \C \to Ab$ a reduced functor, the group $T_{11}(cr_2F)(E,E)$ admits the structure of  a symmetric $\bar{\Lambda} \otimes \bar{\Lambda}$-module, with involution denoted by $\bar{T}^F$, such that  the projection $t_{11}: F(E|E) \to T_{11}(cr_2F)(E,E)$ is a morphism of symmetric $\Lambda \otimes \Lambda$-modules.
\end{cor}

Conversely, symmetric $\Lambda \otimes \Lambda$-modules give rise to symmetric bifunctors, as follows.

\begin{prop}\label{symmodindfct}
Let $M$ be a symmetric $\Lambda \otimes \Lambda$-module with involution $T$. Then 
there are symmetric bifunctors $(\mathbb{T}(M),T^M)$ and $(\mathbb{T}_{11}(M),\bar{T}^M)$ where
$$T^M_{X,Y} : (U(X)\otimes U(Y))\otimes_{\Lambda \otimes \Lambda} M \to   (U(Y)\otimes U(X))\otimes_{\Lambda \otimes \Lambda} M$$
is given by
$ T^M_{X,Y} (f\otimes g\otimes x) = g\otimes f\otimes Tx$ for $f,g\in \C(E,E)$ and $x\in M$, and where $\bar{T}^M$ is given such that 
$t_1\otimes t_1\otimes id : \mathbb{T}(M) \to \mathbb{T}_{11}(M)$ is a morphism of symmetric bifunctors, in the obvious sense.
\hfill$\Box$
\end{prop}

\begin{rem} Similarly assigning  
$(\mathbb{T}_{11}(M),\bar{T}^M)$ to $(M,T)$ defines a functorial equivalence between symmetric 
$\bar{\Lambda} \otimes \bar{\Lambda}$-modules and symmetric bifunctors $(B,T')$ from $\C$ to $Ab$ such that $B$ is bilinear bireduced, but we do not need this here.
\end{rem}

\section{Equivalence between quadratic functors and modules over suitable ringoids}

In this section, we generalize an approach of Pirashvili in \cite{Pira-russe} to obtain an equivalence between the category $Quad(\C, Ab)$ and the category of modules over a particular pre-additive category (or ringoid following the terminology of Baues) with two objects. More explicitely we have:
\begin{defi}
Let $\mathcal{R}$ be the ringoid having two objects $R_e$ and $R_{ee}$ and morphisms:
$$Hom_{\mathcal{R}}(R_e, R_{ee})=T_2U_E(E \mid E)$$
$$Hom_{\mathcal{R}}(R_{ee}, R_{e})=\overline{\Lambda} \otimes \overline{\Lambda}$$
$$End_{\mathcal{R}}(R_e)=\overline{\overline{\Lambda}}:=(T_2U_E)(E)\quad\mbox{(as a ring)}$$
$$End_{\mathcal{R}}(R_{ee})=(\overline{\Lambda} \otimes \overline{\Lambda})\wr \mathfrak{S}_2 \quad\mbox{(as a ring);}$$
the remaining compositions in $\mathcal{R}$ are given as follows: for $a, b, c, d, \alpha, \beta, \gamma \in \C(E,E)$, $\xi \in \C(E,E \vee E)$:
$$\xymatrix{
R_e \ar@(ul,dl)_{t_2(\gamma)} \ar@(ur,ul)[rr]^{\tilde{\xi}=\rho^2_{(1,2)} t_2(\xi)} && R_{ee} \ar@(dl,dr)[ll]^{\overline{\alpha} \otimes \overline{\beta}} \ar@(ur,dr)^{\overline{a}\otimes \overline{b} +t (\overline{c} \otimes \overline{d})}
}$$
\begin{equation} \label{th1th2}
t_2(\gamma) \circ (\overline{\alpha} \otimes \overline{\beta})=\overline{ \gamma \alpha} \otimes \overline{\gamma \beta }
\end{equation}
\begin{equation} \label{th3th1}
 \rho^2_{(1,2)}t_2(\xi) \circ t_2(\gamma)= \rho^2_{(1,2)}t_2( \xi \gamma) 
\end{equation}
\begin{equation} \label{th3th2}
 \rho^2_{(1,2)}t_2(\xi) \circ (\overline{\alpha} \otimes \overline{\beta})=\overline{r_1 \xi \alpha} \otimes \overline{r_2 \xi \beta}+ (\overline{r_1 \xi \beta} \otimes \overline{r_2 \xi \alpha})t
\end{equation}
\begin{equation}  \label{th2th3}
(\overline{\alpha} \otimes \overline{\beta}) \circ  \rho^2_{(1,2)}t_2(\xi) =t_2(\nabla (\alpha \vee \beta) \xi-\alpha r_1 \xi - \beta r_2 \xi)
\end{equation}
\begin{equation} \label{th2th4}
(\overline{\alpha} \otimes \overline{\beta}) \circ (\overline{a}\otimes \overline{b} +(\overline{c} \otimes \overline{d})t)=\overline{\alpha a}\otimes \overline{\beta b} +\overline{\beta d} \otimes \overline{\alpha c}
\end{equation}
\begin{equation} \label{th4th3}
(\overline{a}\otimes \overline{b} +(\overline{c} \otimes \overline{d})t) \circ \tilde{\xi}= T_2U_E(a|b)(\tilde{\xi})+ T_2U_E(c|d)(T\tilde{\xi})
\end{equation}

where $T$ is the involution of the symmetric $\overline{\Lambda} \otimes \overline{\Lambda}$-module 
$ T_2U_E(E|E)$.
\end{defi}
It follows from Theorem \ref{eq-ringoid} below that $\mathcal{R}$ is a well defined ringoid.

\begin{thm} \label{thm-eq-quad}
If $\C=\langle E \rangle_{\C}$ we have an equivalence of categories:
$$\sigma: Quad(\C, Ab) \xrightarrow{\simeq} Func^{add}(\mathcal{R}, Ab):=\mathcal{R}\mbox{--\,}mod.$$
\end{thm}

The proof of this theorem requires many intermediate results and will only be achieved at the end of section 4.3 below.

\subsection{Projective generators of $Quad(\C, Ab)$}

Applying Theorem \ref{thm-Pira} to quadratic functors we see that the category $Quad(\C, Ab)$ admits $\{ T_2U_E, T_2U_{E \vee E} \}$ as a family of small projective generators. The following proposition gives a refinement of this result.
\begin{prop} \label{gene-quad}
For $\C=\langle E \rangle_{\C}$, the category $Quad(\C, Ab)$ admits $\{ T_2U_E, T_1U_E \otimes T_1U_E \}$ as a family of small projective generators.
\end{prop}

The proof of this proposition relies on the following lemma.
\begin{lm} \label{lm-decompo}
We have a natural decomposition:
$$U_{E \vee E}= U_E \oplus U_E \oplus U_E \otimes U_E.$$
\end{lm}
\begin{proof}
Let $X \in \C$ and $f \in \C(E \vee E, X)$, we define:
$$\sigma_X: U_{E \vee E}(X) \to U_E(X) \oplus U_E(X) \oplus U_E(X) \otimes U_E(X)$$
by: 
$$\sigma_X(h)=(h \circ i_1, h \circ i_2, h \circ i_1 \otimes h \circ i_2)$$
for $h \in \C(E \vee E, X)$ and 
$$\tau_X: U_E(X) \oplus U_E(X) \oplus U_E(X) \otimes U_E(X) \to U_{E \vee E}(X) $$
by: 
$$\tau_X(f, g, f' \otimes g')=f r_1+ g r_2+ \nabla^2(f' \vee g')-f' r_1- g' r_2$$
for $f, g, f', g' \in \C(E,X)$.
We easily verify that $\tau_X$ is the inverse of $\sigma_X$ and these maps are natural.
\end{proof}
\begin{cor} \label{T2UEE}
There is an isomorphism 
$$T_2U_{E \vee E} \simeq T_2U_E \oplus T_2U_E \oplus T_1U_E \otimes  T_1U_E $$
where the injection and retraction 
$\xymatrix{
T_1U_E \otimes  T_1U_E   \ar@<-.5ex>[r]_-{I}  & T_2U_{E \vee E}  \ar@<-.5ex>[l]_-{R}}$ are given by
$$I_X(\overline{f} \otimes \overline{g})=t_2( \nabla ( f \vee g) - f r_1 - g r_2)$$
$$R_X(t_2(h))= \overline{h i_1} \otimes \overline{h i_2}$$
for $X \in \C$, $f,g \in \C(E,X)$ and $h \in \C(E \vee E, X)$.
\end{cor}

\begin{proof}
We deduce from Lemma \ref{lm-decompo} that:
$$T_2U_{E \vee E}= T_2U_E \oplus T_2U_E \oplus T_2(U_E \otimes U_E).$$
But 
\begin{eqnarray*}
T_2(U_E \otimes U_E) &=& T_{11}(U_E \boxtimes U_E) \Delta \mathrm{\ by\ Lemma\ \ref{lm-bifunc}}\\
&=&(T_1U_E \boxtimes T_1U_E) \Delta \mathrm{\ by\ example\  \ref{TFoG}}\\
&=&T_1U_E \otimes T_1U_E.
\end{eqnarray*}
\end{proof}

  Proposition  \ref{gene-quad} now is a straightforward consequence of Theorem \ref{thm-Pira} and Corollary \ref{T2UEE}.

\subsection{Gabriel-Popescu Theorem}

Recall the following fundamental property  of abelian categories:

\begin{thm}[\cite{Popescu} Corollaire 6.4 p 103] \label{Freyd}
For any abelian category $\C$ the following assertions are equivalent.
\begin{enumerate}
\item The category $\C$ has arbitrary direct sums and  $\{ P_i\}_{i \in I}$ is a set of projective small generators of $\C$.
\item The category $\C$ is equivalent to the subcategory  $Func^{add}(\mathcal{P}^{op}, Ab)$ of $Func(\mathcal{P}^{op}, Ab)$ whose objects are additive functors  (i.e. functors satisfying $F(f+g)=F(f)+F(g)$ where $f$ and $g$ are morphisms of $Hom_{\mathcal{P}^{op}}(V,W)$) and $\mathcal{P}$ is the full subcategory of $\C$ whose set of objects is $\{ P_i \mid i \in I \}$.
\end{enumerate}
\end{thm}
Combining Proposition \ref{gene-quad} and Theorem \ref{Freyd} we obtain:
\begin{thm} \label{thm-alpha}
For $\C=\langle E \rangle_{\C}$ and $\mathcal{P}$ the full subcategory of $Quad(\C, Ab)$ having as objects $T_2U_E$ and $T_1U_E \otimes T_1U_E$, we have an equivalence of categories:
$$\alpha: Quad(\C, Ab) \xrightarrow{\simeq} Func^{add}(\mathcal{P}^{op}, Ab):= \mathcal{P}^{op}\mbox{--\,}  mod$$
assigning to a quadratic functor $F: \C \to Ab$ the restriction to $\mathcal{P}$ of the representable functor $Hom_{Quad(\C, Ab)}(-,F)$.
\end{thm}

\subsection{The category $\mathcal{P}^{op}$}
The aim of this section is to prove the following result:
\begin{thm} \label{eq-ringoid}
We have an isomorphism of ringoids:
$$\theta: \mathcal{P}^{op} \xrightarrow{\simeq} \mathcal{R}$$
given on the objects by: $\theta(T_2U_E)=R_e$ and $\theta(T_1U_E \otimes T_1U_E)=R_{ee}$.
\end{thm}
In order to prove this theorem  we need the following proposition.
\begin{prop} \label{T1T1-F}
For $F \in Quad(\C, Ab)$ there is an isomorphism $\tilde{\Y}$ fitting into the commutative diagram:
$$\xymatrix{
Hom_{Quad(\C, Ab)}(T_1U_E \otimes T_1U_E, F) \ar@{.>}[r]^-{\tilde{\Y}} \ar[d]_-{R^*}& F(E \mid E) \ar@{^(->}[d]^-{\iota^2_{(1,2)}}\\
Hom_{Quad(\C, Ab)}(T_2U_{E \vee E}, F) \ar[r]_-{\Y}^-{\simeq}& F(E \vee E).
}$$
\end{prop}
\begin{proof}
By Corollary \ref{T2UEE} we have a split exact sequence:
$$\xymatrix{
0 \ar[r] & T_2 U_E \oplus T_2 U_E \ar[r]^-{(r_1^*, r_2^*)}& T_2 U_{E \vee E} \ar@<.5ex>[r]^-{R} & T_1 U_E \otimes T_1 U_E \ar@<.5ex>[l]^-I    \ar[r] &0.
}$$

Homming into $F$ and using naturality  of $\Y$ in the first variable provides the following commutative diagram with exact rows which implies the assertion.
$$\xymatrix{
0 \ar[r] & Hom(T_1 U_E \otimes T_1 U_E,F) \ar[r]^-{R^*} \ar@{.>}[d]_{\simeq}^{\tilde{\Y} }& Hom(T_2 U_{E \vee E},F) \ar[r]^-{((r_1^*)^*, (r_2^*)^*)} \ar[d]_{\simeq}^{\Y} & Hom(T_2 U_E,F) \oplus Hom(T_2 U_E,F) \ar[d]_{\simeq}^{\Y \oplus \Y} \ar[r] &0\\
0 \ar[r] & F(E \mid E) \ar[r]_{\iota^2_{(1,2)}} & F(E \vee E) \ar[r]_{(F(r_1), F(r_2))} & F(E) \oplus F(E) \ar[r] & 0
}$$
\end{proof}

\begin{proof}[Proof of Theorem \ref{eq-ringoid}]
We first determine the morphisms groups of the ringoid $\mathcal{P}^{op}$. The Yoneda Lemma \ref{Yoneda-poly} provides isomorphisms of groups:
$$\theta_1=\Y: Hom_{\PP^{op}}(T_2U_E, T_2U_E) \xrightarrow{\simeq} T_2U_E(E)=\overline{\overline{\Lambda}}$$
$$\theta_2=\Y: Hom_{\PP^{op}}(T_1U_E \otimes T_1U_E, T_2U_E) \xrightarrow{\simeq} T_1U_E(E) \otimes T_1U_E(E)=\overline{\Lambda} \otimes \overline{\Lambda}.$$
Proposition \ref{T1T1-F} furnishes isomorphisms of groups:
$$\theta_3= \tilde{\Y}: Hom_{\PP^{op}}(T_2U_E, T_1U_E \otimes T_1U_E) \xrightarrow{\simeq} T_2U_E(E \mid E) $$
$$\tilde{\Y}: Hom_{\PP^{op}}(T_1U_E \otimes T_1U_E, T_1U_E \otimes T_1U_E) \xrightarrow{\simeq} (T_1U_E \otimes T_1U_E)(E \mid E).$$
As $T_1U_E$ is linear, we have an isomorphism $\tilde{\iota}: \overline{\Lambda} \otimes \overline{\Lambda} \oplus \overline{\Lambda} \otimes \overline{\Lambda} \to (T_1U_E \otimes T_1U_E) (E \mid E)$ given by: $\tilde{\iota}(\overline{a} \otimes \overline{b}, \overline{c} \otimes \overline{d})=\overline{i_1a} \otimes \overline{i_2b}+ \overline{i_2d} \otimes \overline{i_1c}$. Thus we get an isomorphism of groups:
$$\theta_4={\tilde{\iota}}^{-1} \tilde{\Y}: Hom_{\PP^{op}}(T_1U_E \otimes T_1U_E, T_1U_E \otimes T_1U_E) \xrightarrow{\simeq} \overline{\Lambda} \otimes \overline{\Lambda} \oplus \overline{\Lambda} \otimes \overline{\Lambda}. $$
Noting that $\theta_1^{-1}(t_2 \alpha)=\alpha^*$ we see that $\theta_1$ is a ring isomorphism. In order to compute the remaining composition laws in $\PP$, we need the following technical lemma.

\begin{lm} \label{Yoncomp}
Let $n \geq 1$, and for $k=1,2$ let $E_k \in \C$ and $S_k: \C \to Ab$ be a direct factor of $T_nU_{E_k}$ with injection and retraction $\xymatrix{
S_k \ar@<-.5ex>[r]_-{I_k}  & T_nU_{E_k}  \ar@<-.5ex>[l]_-{R_k}}$. Furthermore, let $F \in Func_*(\C, Ab)_{\leq n}$ and $S_1 \xrightarrow{\varphi_1} S_2 \xrightarrow{\varphi_2} F$ be natural transformations. Write $x_1=\Y R_1^*(\varphi_1) \in S_2(E_1)$ and $x_2= \Y R_2^*(\varphi_2) \in F(E_2)$ and let $(I_2)_{E_1}(x_1)=\sum_{j} n_j t_n(\alpha_j)$ be a decomposition in $T_nU_{E_2}(E_1)$ with $n_j\in \Z$ and $\alpha_j \in \C(E_2, E_1)$. Then
$$\Y R_1^*(\varphi_2 \varphi_1)=\sum_j n_j F(\alpha_j)(x_2) \in F(E_1).$$
\end{lm}
\begin{proof}
Let $\tilde{\varphi_1}=I_2 \varphi_1 R_1 \in Hom(T_n U_{E_1}, T_n U_{E_2})$. By naturality of $\Y$ in the second variable we have
$$\Y(\tilde{\varphi_1})=\Y {I_2}_*(\varphi_1 R_1)=(I_2)_{E_1} \Y R_1^*(\varphi_1)=\sum_j n_j t_n(\alpha_j)$$
whence
\begin{equation} \label{phi1tilde}
\tilde{\varphi_1}=\sum_j n_j t_n(\alpha_j)^*
\end{equation}
since $\Y(t_n(\alpha_j)^*)=t_n(\alpha_j)$. Now consider the following diagram:
$$\xymatrix{
Hom(S_2, F) \ar[r]^{\varphi_1^*}  \ar[d]_{R_2^*}& Hom(S_1,F) \ar[d]_{R_1^*}\\
Hom(T_nU_{E_2}, F) \ar[r]^{\tilde{\varphi_1}^*} \ar[d]_{\Y}^{\simeq}& Hom(T_nU_{E_1}, F)  \ar[d]_{\Y}^{\simeq}\\
F(E_2)  \ar[r]^{\sum_j n_j F(\alpha_j)}  & F(E_1).
}$$
The upper square commutes by definition of $\tilde{\varphi_1}$ while the lower square commutes by \ref{phi1tilde} and naturality of $\Y$ in the first variable. The desired formula follows applying commutativity of the exterior rectangle to $\varphi_2 \in Hom(S_2,F)$.
\end{proof}
In a first step we postcompose endomorphisms of $T_1U_E \otimes T_1U_E$ by other maps in $\PP$. By the isomorphisms $\theta_4$ and $\theta_3$ and by \ref{Imiota} it suffices to consider maps:
$$\xymatrix{
 & & T_1U_E \otimes T_1U_E\\
  T_1U_E \otimes T_1U_E \ar[r]^{\varphi_1} & T_1U_E \otimes T_1U_E \ar[ru]^{\varphi_2} \ar[rd]_{\varphi_2'}\\
  & & T_2U_E}
$$
such that:
$$x_k= \Y R^*(\varphi_k)=\tilde{\iota}(\overline{a_k} \otimes \overline{b_k}, \overline{c_k} \otimes \overline{d_k})= \overline{i_1 a_k} \otimes \overline{i_2 b_k}+ \overline{i_2 d_k} \otimes \overline{i_1 c_k}\in T_1U_E(E \vee E) \otimes T_1U_E(E \vee E)$$
for $k \in \{1,2\}$ and $a_k, b_k, c_k, d_k \in \C(E,E)$, and such that:
$$x'_2= \Y R^*(\varphi'_2)=\iota^2_{(1,2)} \rho^2_{(1,2)} t_2 (\xi)=t_2(\xi-i_1r_1 \xi-i_2r_2 \xi) \in T_2U_E(E \vee E)$$
for $\xi \in \C(E, E \vee E)$. Then:
$$I_{E \vee E}(x_1)=t_2(\nabla(i_1 a_1 \vee i_2 b_1)-i_1a_1r_1-i_2b_1r_2)+t_2(\nabla(i_2 d_1 \vee i_1 c_1)-i_2d_1r_1-i_1c_1r_2)$$
$$=t_2( a_1 \vee b_1-i_1a_1r_1-i_2b_1r_2)+t_2((c_1 \vee d_1) \tau -i_2d_1r_1-i_1c_1r_2).$$
Applying Lemma \ref{Yoncomp} for $E_1=E_2=E \vee E$, $S_1=S_2= F=T_1U_E \otimes T_1U_E $, $R_1=R_2=R$ and $I_1=I_2=I$ we get (omitting all terms being trivial since $r_1 i_2=0=r_2i_1)$):
\begin{eqnarray*}
\Y R^*(\varphi_2 \varphi_1)&=&(t_1 \otimes t_1)((a_1 \vee b_1)i_1a_2 \otimes (a_1 \vee b_1)i_2 b_2+(a_1 \vee b_1)i_2d_2 \otimes (a_1 \vee b_1)i_1 c_2\\
&&+(c_1 \vee d_1) \tau i_1a_2 \otimes (c_1 \vee d_1)\tau i_2 b_2+(c_1 \vee d_1)\tau i_2d_2 \otimes (c_1 \vee d_1) \tau i_1 c_2)\\
&=&(t_1 \otimes t_1)(i_1a_1a_2 \otimes i_2 b_1 b_2+i_2 b_1 d_2 \otimes i_1 a_1 c_2
+i_2 d_1a_2 \otimes i_1 c_1 b_2+ i_1 c_1 d_2 \otimes  i_2 d_1 c_2)\\
&=& \tilde{\iota}(t_1 \otimes t_1)(a_1a_2 \otimes b_1 b_2+c_1 d_2 \otimes d_1 c_2, 
 a_1c_2 \otimes b_1 d_2+  c_1 b_2 \otimes  d_1 a_2)\\
 &=&\tilde{\iota}((\overline{a_1} \otimes \overline{b_1}, \overline{c_1} \otimes \overline{d_1}) \star  (\overline{a_2} \otimes \overline{b_2}, \overline{c_2} \otimes \overline{d_2}))
 \end{eqnarray*}
 where $\star$ denotes the wreath product structure on $\overline{\Lambda} \otimes \overline{\Lambda} \oplus \overline{\Lambda} \otimes \overline{\Lambda}= \overline{\Lambda} \otimes \overline{\Lambda} \oplus (\overline{\Lambda} \otimes \overline{\Lambda})t$.
 This shows that $\theta_4$ is a ring isomorphism: $\theta_4: End_{\PP^{op}}(T_1U_E \otimes T_1U_E) \xrightarrow{\simeq} (\overline{\Lambda} \otimes \overline{\Lambda}) \wr  \mathfrak{S}_2. $ Next, again by Lemma \ref{Yoncomp}  for $E_1=E_2=E \vee E$, $S_1=S_2=T_1U_E \otimes T_1U_E $, $F=T_2U_E$, $R_1=R_2=R$ and $I_1=I_2=I$ and omitting all terms being trivial:
 \begin{eqnarray*}
\Y R^*(\varphi'_2 \varphi_1)&=& T_2U_E(a_1 \vee b_1) \iota^2_{(1,2)} \rho^2_{(1,2)} t_2(\xi) + t_2 (-i_1 a_1 r_1 \xi +i_1 a_1 r_1 \xi -i_2 b_1 r_2 \xi +i_2 b_1 r_2 \xi)\\
&&+T_2U_E(c_1 \vee d_1) T_2U_E(\tau)\iota^2_{(1,2)} \rho^2_{(1,2)} t_2(\xi)\\
&& + t_2 (-i_2 d_1 r_1 \xi +i_2 d_1 r_1 \xi -i_1 c_1 r_2 \xi +i_1 c_1 r_2 \xi)\\
&=&  \iota^2_{(1,2)} T_2U_E(a_1 \mid b_1) \rho^2_{(1,2)} t_2(\xi) +\iota^2_{(1,2)} T_2U_E(c_1 \mid d_1) T \rho^2_{(1,2)} t_2(\xi)
 \end{eqnarray*}
whence 
$$\theta_3(\varphi_1^{op} \circ {\varphi'_2}^{op})=\theta_4(\varphi_1^{op}) \circ \theta_3({\varphi'_2}^{op}) \mathrm{\quad by\ \ref{th4th3}}.$$
Next consider maps 
$$\xymatrix{
 & & T_2U_E \\
  T_1U_E \otimes T_1U_E \ar[r]^-{\varphi_1} & T_2U_E \ar[ru]^{\varphi_2} \ar[rd]_{\varphi_2'}\\
  & & T_1U_E \otimes T_1U_E }
$$
such that:
$$x_1= \Y R^*(\varphi_1)= \iota^2_{(1,2)} \rho^2_{(1,2)} t_2 (\xi)=t_2(\xi-i_1r_1 \xi-i_2r_2 \xi)$$ for $\xi \in \C(E, E \vee E)$
$$x_2=\Y(\varphi_2)=t_2(\alpha) \mathrm{\quad and\quad} x_2'=\Y(\varphi'_2)=\overline{\alpha} \otimes \overline{\beta}$$
for $\alpha, \beta \in \C(E,E)$; again this suffices by \ref{Imiota}. Applying Lemma \ref{Yoncomp} for $E_1=E \vee E$, $E_2=E$, $S_1=T_1 U_E \otimes T_1 U_E$, $S_2=F=T_2 U_E$, $R_1=R$, $I_1=I$ and $I_2=R_2=Id$ we get:
 \begin{eqnarray*}
\Y R^*(\varphi_2 \varphi_1)&=& T_2U_E(\xi) (t_2 \alpha)- T_2U_E(i_1 r_1 \xi) (t_2 \alpha)- T_2U_E( i_2 r_2 \xi) (t_2 \alpha)\\
&=& t_2(\xi \alpha- i_1 r_1 \xi \alpha - i_2 r_2 \xi \alpha)\\
&=& \iota^2_{(1,2)} \rho^2_{(1,2)} t_2(\xi \alpha)
\end{eqnarray*}
whence 
$$\theta_3(\varphi_1^{op} \circ {\varphi_2}^{op})=\theta_3(\varphi_1^{op}) \circ \theta_1({\varphi_2}^{op}) \mathrm{\quad by\ \ref{th3th1}}.$$
Next applying  Lemma \ref{Yoncomp} for $E_1=E \vee E$, $E_2=E$, $S_1=F=T_1 U_E \otimes T_1 U_E$, $S_2=T_2 U_E$, $R_1=R$, $I_1=I$ and $I_2=R_2=Id$ we get
 \begin{eqnarray*}
&&\Y R^*(\varphi'_2 \varphi_1)\\
&=& (T_1U_E(\xi) \otimes  T_1U_E(\xi)- T_1U_E(i_1 r_1 \xi) \otimes  T_1U_E(i_1 r_1 \xi)- T_1U_E(i_2 r_2 \xi) \otimes  T_1U_E(i_2 r_2 \xi)) (\overline{\alpha} \otimes \overline{\beta})\\
&=& (t_1 \otimes t_1)(\xi \alpha \otimes \xi \beta -i_1 r_1 \xi \alpha \otimes i_1 r_1 \xi \beta-i_2 r_2 \xi \alpha \otimes i_2 r_2 \xi \beta)\\
&=&(t_1 \otimes t_1)((i_1 r_1+ i_2 r_2)\xi \alpha \otimes (i_1 r_1+ i_2 r_2)\xi \beta -i_1 r_1 \xi \alpha \otimes i_1 r_1 \xi \beta-i_2 r_2 \xi \alpha \otimes i_2 r_2 \xi \beta) \mathrm{\quad by\ Lemma\ \ref{lem1}}\\
&=&(t_1 \otimes t_1)(i_1 r_1 \xi \alpha \otimes i_2 r_2 \xi \beta+i_2 r_2 \xi \alpha \otimes i_1 r_1 \xi \beta)\\
&=& \tilde{\iota}(\overline{r_1 \xi \alpha} \otimes \overline{r_2 \xi \beta} , \overline{r_1 \xi \beta} \otimes \overline{r_2 \xi \alpha} )
\end{eqnarray*}
whence
$$\theta_4(\varphi_1^{op} \circ {\varphi'_2}^{op})=\theta_3(\varphi_1^{op}) \circ \theta_2({\varphi'_2}^{op}) \mathrm{\quad by\ \ref{th3th2}}.$$
Now, consider diagrams
$$\xymatrix{
 & &  T_1U_E \otimes T_1U_E \\
 T_2U_E \ar[r]^-{\varphi_1} &  T_1U_E \otimes T_1U_E \ar[ru]^{\varphi_2} \ar[rd]_{\varphi_2'}\\
  & &  T_2U_E }
$$
such that:
$$x_1= \Y (\varphi_1)= \overline{\alpha} \otimes \overline{\beta} \in  \overline{\Lambda} \otimes \overline{\Lambda}  $$ 
$$x_2=\Y R^*(\varphi_2)= i_1 a \otimes i_2 b+ i_2 d \otimes i_1 c= \tilde{\iota}(a \otimes b, c \otimes d) \in T_1U(E \vee E) \otimes T_1 U(E \vee E)$$
$$ x_2'=\Y R^* (\varphi'_2)= \iota^2_{(1,2)} \rho^2_{(1,2)} t_2(\xi)$$
for $\alpha, \beta, a, b, c, d \in \C(E,E)$ and $\xi \in \C(E, E \vee E)$. We have:
$$I_E(x_1)=t_2(\nabla(\alpha \vee \beta)-\alpha r_1- \beta r_2).$$
Applying Lemma \ref{Yoncomp} for $E_1=E$, $E_2=E  \vee E$, $S_1=T_2 U_E$, $S_2=F=T_1 U_E \otimes T_1 U_E$, $I_1=R_1=Id$,$R_2=R$ and $I_2=I$ we get:
 \begin{eqnarray*}
\Y (\varphi_2 \varphi_1)&=& (T_1U_E(\nabla(\alpha \vee \beta)) \otimes T_1U_E(\nabla(\alpha \vee \beta)))(x_2)- (T_1U_E(\alpha r_1) \otimes T_1U_E(\alpha r_1))(x_2)\\
&&- (T_1U_E(\beta r_2) \otimes T_1U_E(\beta r_2))(x_2)\\
&=&(t_1 \otimes t_1)(\nabla(\alpha \vee \beta) i_1a \otimes \nabla(\alpha \vee \beta) i_2 b + \nabla(\alpha \vee \beta) i_2 d \otimes \nabla(\alpha \vee \beta) i_1 c\\
&=&(t_1 \otimes t_1)(\alpha a \otimes \beta b+ \beta d \otimes \alpha c).
\end{eqnarray*}
Hence
$$\theta_2(\varphi_1^{op} \circ {\varphi_2}^{op})=\theta_2(\varphi_1^{op}) \circ \theta_4({\varphi_2}^{op}) \mathrm{\quad by\ \ref{th2th4}}.$$
Applying Lemma \ref{Yoncomp} for $E_1=E$, $E_2=E  \vee E$, $S_1=F=T_2 U_E$, $S_2=T_1 U_E \otimes T_1 U_E$, $I_1=R_1=Id$,$R_2=R$ and $I_2=I$ we get:
 \begin{eqnarray*}
\Y (\varphi'_2 \varphi_1)&=&T_2U_E(\nabla(\alpha \vee \beta))( t_2 (\xi -i_1 r_1 \xi -i_2 r_2 \xi) )-T_2U_E(\alpha r_1) (t_2 (\xi -i_1 r_1 \xi - i_2 r_2 \xi))\\
&&-T_2U_E(\beta r_2) (t_2 (\xi -i_1 r_1 \xi - i_2 r_2 \xi))\\
&=& t_2(\nabla(\alpha \vee \beta) \xi- \alpha r_1 \xi - \beta r_2 \xi- \alpha r_1 \xi + \alpha r_1 \xi - \beta r_2 \xi+ \beta r_2 \xi).
\end{eqnarray*}
Thus 
$$\theta_1(\varphi_1^{op} \circ {\varphi'_2}^{op})=\theta_2(\varphi_1^{op}) \circ \theta_3({\varphi'_2}^{op}) \mathrm{\quad by\ \ref{th2th3}}.$$
Finally, let
$$\xymatrix{
 T_2U_E \ar[r]^-{\varphi_1} &  T_2U_E \ar[r]^-{\varphi_2} & T_1U_E \otimes T_1U_E
 }$$
such that:
$$x_1= \Y (\varphi_1)=  t_2(\gamma) \mathrm{\quad and \quad} x_2=\Y(\varphi_2)= \overline{\alpha}  \otimes \overline{\beta}$$
for $\alpha, \beta, \gamma \in \C(E,E)$. Then:
$$\Y((\varphi_2 \varphi_1)= (T_1U_E(\gamma) \otimes T_1U_E(\gamma))(\overline{\alpha}  \otimes \overline{\beta})=\overline{ \gamma \alpha}  \otimes \overline{ \gamma \beta}$$
whence
$$\theta_2(\varphi_1^{op} \circ {\varphi_2}^{op})=\theta_1(\varphi_1^{op}) \circ \theta_2({\varphi_2}^{op}) \mathrm{\quad by\ \ref{th1th2}}.$$
\end{proof}

\begin{proof}[Proof of Theorem \ref{thm-eq-quad}]
Combining  Theorem \ref{thm-alpha} and Theorem \ref{eq-ringoid} we obtain the result.
\end{proof}

\begin{rem}
Applying the ideas of this section to linear functors we find again Theorem \ref{lin} since, in this case, the full subcategory $\mathcal{P}$ of $Lin(\C, Ab)$ having the single object $T_1U_E$ is equivalent to the ring $(\overline{\Lambda})^{op}=End_{\mathcal{P}}(T_1U_E)$. 
\end{rem}

\subsection{The category $\R \ti mod$}
In this section, we give a reformulation of the structure of an $\R$-module that motivates the introduction of quadratic $\C \ti$modules in section \ref{section-4}.
\begin{lm} \label{Rmodreform}
An $\R \ti$module $M: \R \to Ab$ is equivalent with the following data:
\begin{enumerate}
\item  a left $\overline{\overline{\Lambda}} \ti$module $M_e$;
\item a symmetric $\overline{\Lambda} \otimes \overline{\Lambda} \ti$module $M_{ee}$ (with involution denoted by $T$);
\item a map of $\overline{\overline{\Lambda}} \ti$modules 
$$p: (\overline{\Lambda} \otimes \overline{\Lambda}) \underset{(\overline{\Lambda} \otimes \overline{\Lambda}) \wr \mathfrak{S}_2}{ \otimes } M_{ee} \to M_e$$
where the structure of right $(\overline{\Lambda} \otimes \overline{\Lambda}) \wr \mathfrak{S}_2 \ti $module on $\overline{\Lambda} \otimes \overline{\Lambda}$ is given by \ref{th2th4}, and the structure of $\overline{\overline{\Lambda}}$-module on $\overline{\Lambda} \otimes \overline{\Lambda}$ is given by the diagonal action;
\item a map of symmetric $\overline{\Lambda} \otimes \overline{\Lambda}$\ti modules 
$$h: T_2U_E(E \mid E) \underset{\overline{\overline{\Lambda}}}{ \otimes } M_e \to M_{ee}$$
\end{enumerate}
such that for $\alpha, \beta \in \C(E,E)$, $\xi \in \C(E, E \vee E)$, $a \in M_e$ and $m \in M_{ee}$ the following relations hold:
\begin{equation*} \label{RM1}
(RM1)\quad t_2(\nabla (\alpha \vee \beta) \xi- \alpha r_1 \xi- \beta r_2 \xi)a=p(\overline{\alpha} \otimes \overline{\beta} \otimes h(\rho^2_{(1,2)} t_2(\xi) \otimes a))
\end{equation*}
\begin{equation*} \label{RM2}
(RM2) \quad h(\rho^2_{(1,2)} t_2(\xi) \otimes p(\overline{\alpha} \otimes \overline{\beta} \otimes m))=(\overline{r_1 \xi \alpha} \otimes \overline{r_2 \xi \beta})m + (r_1 \xi \beta \otimes r_2 \xi \alpha)Tm.
\end{equation*}
\end{lm}

\begin{proof}
Given $F \in \R \ti mod$, let $M_e=F(R_e)$, $M_{ee}=F(R_{ee})$, $p(\ol{\alpha} \otimes \ol{\beta} \otimes m)=F(\ol{\alpha} \otimes \ol{\beta})(m)$ and $h(\tilde{\xi} \otimes a)=F(\tilde{\xi})(a)$ where $\tilde{\xi}= \rho^2_{(1,2)} t_2(\xi)$. Then the relation $(RM1)$ corresponds to the relation $F((\ol{\alpha} \otimes \ol{\beta}) \circ \tilde{\xi})= F(\ol{\alpha} \otimes \ol{\beta}) \circ  F(\tilde{\xi})$, in fact:
\begin{eqnarray*}
F((\ol{\alpha} \otimes \ol{\beta}) \tilde{\xi})(a) &=& F(\nabla(\alpha \vee \beta) \xi- \alpha r_1 \xi - \beta r_2 \xi)(a) \mathrm{\quad by\ \ref{th2th3}} \\
&=& t_2(\nabla(\alpha \vee \beta) \xi- \alpha r_1 \xi - \beta r_2 \xi)a
\end{eqnarray*}
while
\begin{eqnarray*}
F(\ol{\alpha} \otimes \ol{\beta}) \circ  F(\tilde{\xi})(a)&=& F(\ol{\alpha} \otimes \ol{\beta}) h(\tilde{\xi} \otimes a)\\
&=& p(\ol{\alpha} \otimes \ol{\beta} \otimes h(\tilde{\xi} \otimes a)).
\end{eqnarray*}
Similarly, the relation $(RM2)$ corresponds to the relation:
$F(\tilde{\xi} \circ (\ol{\alpha} \otimes \ol{\beta}))=  F(\tilde{\xi}) \circ F(\ol{\alpha} \otimes \ol{\beta})$. In fact:
\begin{eqnarray*}
F(\tilde{\xi} \circ (\ol{\alpha} \otimes \ol{\beta}))(m) &=& F( \ol{r_1 \xi \alpha} \otimes \ol{r_2 \xi \beta}+(\ol{r_1 \xi \beta} \otimes \ol{r_2 \xi \alpha})t)(m)\\
&=& ( \ol{r_1 \xi \alpha} \otimes \ol{r_2 \xi \beta})m+(\ol{r_1 \xi \beta} \otimes \ol{r_2 \xi \alpha})Tm
\end{eqnarray*}
while
\begin{eqnarray*}
F(\tilde{\xi}) \circ F(\ol{\alpha} \otimes \ol{\beta})(m)&=&F(\tilde{\xi}) p(\ol{\alpha} \otimes \ol{\beta} \otimes m)\\
&=&h(\tilde{\xi} \otimes p(\ol{\alpha} \otimes \ol{\beta} \otimes m)).
\end{eqnarray*}
\end{proof}

\section{Quadratic $\C$-modules} \label{section-4}
In this section we introduce quadratic $\C$-modules which generalize to any pointed category $\C$ the quadratic $\mathbb{Z}$-modules considered by Baues in \cite{Baues} for $\C=Ab$. We show that quadratic $\C$-modules constitute a minimal description of $\mathcal R$-modules, and thus of reduced quadratic functors from $\C$ to $Ab$ if $\C$ is a theory generated by $E$. We make the functor from quadratic functors to quadratic $\C$-modules explicit; a canonical  inverse  functor, however, is more difficult to exhibit, and is provided in section 6 by the construction of a quadratic tensor product.


\subsection{Quadratic $\mathcal{C}$-modules}

 \begin{defi} \label{proto-quad-C} {\bf (Proto-quadratic $\mathcal{C}$-module)}
A proto-quadratic $\mathcal{C}$-module relative to $E$ is a diagram of group homomorphisms:
$$M=(T_{11}(cr_2U)(E,E)\otimes_{\Lambda}M_e  \hspace{1mm}\xrightarrow{\hat{H}}  \hspace{1mm}M_{ee}  \hspace{1mm}\xrightarrow{T} \hspace{1mm} M_{ee}   \hspace{1mm}\xrightarrow{P}  \hspace{1mm}M_e)$$
where
\begin{itemize}
\item $M_{e}$ is a left $\Lambda$-module;
\item $M_{ee}$ is a symmetric $ \bar{\Lambda} \otimes \bar{\Lambda}$-module with involution $T$;
\item $P: M_{ee} \to M_e$ is a homomorphism of $\Lambda$-modules with respect to the diagonal action of $\Lambda$ on $M_{ee}$, i.e. for $\alpha \in \mathcal{C}(E,E)$ and $m \in M_{ee}$: 
$$P((\bar{\alpha} \otimes \bar{\alpha})m)= \alpha P(m),$$
and satisfies $PT=P$;
\item $\hat{H}$ is a homomorphism of symmetric $ \bar{\Lambda} \otimes \bar{\Lambda}$-modules such that for $\xi \in \C(E, E \vee E)$ and $a \in M_e$ the following relation holds:
$$(QM1) \quad (\nabla^2 \xi)a=(r_1^2 \xi)a+ (r_2^2 \xi)a+P( \hat{H}(\overline{\rho_{12}^2(\xi)} \otimes a)).$$
\end{itemize}
\end{defi}

\begin{rem} \label{rem-proto-quad}
By Proposition \ref{explicit-lin} condition $(QM1)$ implies that $coker(P)$ is a $\overline{\Lambda}$-module. 
\end{rem}

\begin{defi} \label{quad-C}{\bf (Quadratic $\mathcal{C}$-module)}
A quadratic $\mathcal{C}$-module (relative to $E$) is a proto-quadratic $\mathcal{C}$-module (relative to $E$) as above satisfying the additional property that for $\xi \in \mathcal{C}(E, E \vee E)$ and $m \in M_{ee}$
$$(QM2) \quad \hat{H} (\overline{\rho_{12}^2(\xi)} \otimes Pm)=(\overline{r_{1}^2\xi} \otimes  \overline{r_{2}^2\xi}) (m+Tm).$$
\end{defi}

The intermediate notion of proto-quadratic $\mathcal{C}$-module is justified by the fact that it suffices to give rise to quadratic functors via the quadratic tensor product; however, we show in Theorem \ref{cross-pt-quad} below that a proto-quadratic $\mathcal{C}$-module satisfies relation (QM2) iff the cross-effect of the associated quadratic tensor product  is $\mathbb T_{11}(M_{ee})$, see Theorem \ref{bilin}.

\begin{rem} Suppose that $E$ admits a comultiplication $\mu:E \to E\vee E$ in $\C$ admitting the zero map $E\to 0$ as a counit. Write 
$\mu'=(\iota_{(1,2)}^2)^{-1}(\mu -i_1^2-i_2^2)\in cr_2U(E,E)$ and 
$\alpha \bullet \beta = (\alpha,\beta)\mu$ for $\alpha,\beta\in \C(E,E)$. Then taking $\xi=(\alpha\vee \beta)\mu$ relation $(QM1)$ implies that
\begin{eqnarray*}
 (\alpha \bullet \beta)a &=&  \alpha a + \beta a + 
P\big(
\hat{H} \big( t_{11} (\iota_{(1,2)}^2)^{-1}\big((\alpha\vee \beta)\mu -(\alpha\vee \beta)i_1^2-(\alpha\vee \beta)i_2^2\big) \otimes a )\big) \\
&=&\alpha a + \beta a + 
P\big((\bar{\alpha}\otimes\bar{\beta})\hat{H} (  \overline{\mu'}  \otimes a )\big)
\end{eqnarray*} 
 by $\bar{\Lambda}\otimes \bar{\Lambda}$-linearity of $\hat{H}$. This shows that $(QM1)$ is a generalization of the distributivity law $(\alpha + \beta)a = \alpha a + \beta a + 
P\big(({\alpha}\otimes{\beta})H(a)\big)$ in the case where $\C$ is an additive category \cite{Baues}. In particular, $\hat{H}$ and $P$ are generalizations of the second Hopf invariant and the Whitehead product $[id,id]$, respectively.

Moreover, taking $\xi=\mu$ in relation $(QM2)$ shows that under the above assumption $T$ is determined by $\hat{H}$ and $P$, as 
\begin{eqnarray*}
T(m)&=&
\hat{H} (\overline{\rho_{12}^U(\mu)} \otimes Pm)-m \\
&=& \hat{H} (   \overline{\mu'} \otimes Pm)-m.
\end{eqnarray*} 
This generalizes the formula $T=HP-1$ in \cite{Baues}, \cite{Baues-Pira}, cf.\  the case where $E$ is a cogroup considered in section 7. In general, however,
$T$ is not determined by $\hat{H}$ and $P$, as is illustrated by the following ``extreme" example.

\end{rem}

\begin{exple}
Suppose that $\C$ and $E$ are such that $\C(E,E\vee E) = i^2_{1*}\C(E,E) \cup  i^2_{2*}\C(E,E)$. In particular, this holds when $\C$ is the category $\Gamma$ of finite pointed sets and $E=[1]=\{0,1\}$ is its canonical generator. Then $U(E|E)=0$ by
(\ref{cr2asquot}), whence $\Lambda=\bar{\Lambda}$ and the domain of $\hat{H}$ is trivial. Thus a quadratic $\C$-module relative to $E$ is a diagram
\[ M=(M_{ee} \xrightarrow{T} M_{ee} \xrightarrow{P} M_e)\]
satisfying  the properties  in Definition \ref{proto-quad-C} which do not involve $\hat{H}$; in fact,  the relations $(QM1)$ and $(QM2)$ are trivially satisfied. Together with Theorem \ref{thm} this reproduces a description of quadratic functors from $\Gamma$ to $Ab$ which is a particular case of results obtained in \cite{Pira-Dold}.
\end{exple}

\begin{rem}
 In view of the isomorphism in Proposition \ref{bilin-cross} the map $\hat{H}$ in the definition of a (proto-)quadratic $\mathcal{C}$-module can be replaced by a group homomorphism:
$$\tilde{H}: U(E \vee E) \otimes_{\Lambda} M_e \to M_{ee}$$
satisfying the following relations for $\alpha, \beta \in \mathcal{C}(E, E)$, $\xi \in \mathcal{C}(E, E \vee E)$, $\gamma \in \C(E, E \vee E \vee E)$, $a \in M_e$:
$$\begin{array}{lc}
 (H_1) & \quad \tilde{H}( (\alpha \vee \beta) \xi \otimes a)=(\bar{\alpha} \otimes \bar{\beta})\tilde{H}(  \xi \otimes a)\\
 (H_2) &\quad \tilde{H}( \tau \xi \otimes a)=T\tilde{H}(  \xi \otimes a)\\
 (H_3) &\quad \tilde{H}(i_1^2 \alpha \otimes a)= \tilde{H}(i_2^2 \alpha \otimes a)=0\\
 (H_4) &\quad \tilde{H}((\nabla^2 \vee Id -r_1^2 \vee Id- r_2^2 \vee Id) \gamma \otimes a)=0\\
 (H_5) &\quad \tilde{H}((Id \vee \nabla^2- Id \vee r^2_1-Id \vee r^2_2)  \gamma \otimes a)=0.
\end{array}$$
In fact, $(H_1)$ translates the fact that $\hat{H}$ is a morphism of  $ \bar{\Lambda} \otimes \bar{\Lambda}$-modules, $(H_2)$ corresponds to the fact that $\hat{H}$ is a morphism of symmetric modules, and $(H_3)$, $(H_4)$ and $(H_5)$ correspond to the fact that the source of $\hat{H}$ is $(T_{11}(cr_2(U))(E,E)\otimes_{\Lambda}M_e$, see Proposition \ref{bilin-cross}.
\end{rem}

In the following proposition we give a useful equivalent formulation of condition $(QM1)$.

\begin{prop} \label{QM1}
Relation $(QM1)$ means that the following diagram commutes
$$\xymatrix{
M_{ee} \ar[r]^P & M_e\\
U(E|E) \otimes_{\Lambda} M_e \ar[u]^{\hat{H}(t_{11} \otimes 1)} \ar[r]_{S_2^U \otimes 1} & U(E) \otimes_{\Lambda} M_e  \ar[u]^{\mu_e}_{\simeq}
}$$
 where $\mu_e$ is the canonical isomorphism and $S_2^U$ the map given in Definition \ref{+}. 
 \end{prop}
 \begin{proof}
 We have, for $\xi \in \C(E,E \vee E)$
 $$\begin{array}{ll}
\mu_e (S^U_2 \otimes 1)(\rho^2_{(1,2)}(\xi) \otimes a)&=  \mu_e (U(\nabla^2) \iota^2_{(1,2)} \rho^2_{(1,2)}(\xi) \otimes a)\mathrm{\ by\ definition\  of\ } S^U_2 \mathrm{\ in\ } \ref{+}\\
&=  \mu_e (U(\nabla^2)(Id-i^2_{1*} \circ r^2_{1*}- i^2_{2*} \circ r^2_{2*})(\xi) \otimes a) \mathrm{\ by\ \ref{ses1}} \\
&= \mu_e ((U(\nabla^2)(\xi)-r^2_{1*}(\xi)-r^2_{2*}(\xi)) \otimes a)\\
&=\mu_e((\nabla^2\xi-r^2_{1}\xi-r^2_{2}\xi) \otimes a)\\
&=(\nabla^2\xi)a -(r^2_{1}\xi)a-(r^2_{2}\xi)a.
\end{array}$$
 So the diagram commutes if and only if for all $\xi$ and $a$
 $$P( \hat{H}(\overline{\rho_{12}^2(\xi)} \otimes a))=(\nabla^2\xi)a -(r^2_{1}\xi)a-(r^2_{2}\xi)a.$$
 \end{proof}

\begin{defi}[\textbf{Morphisms of (proto)-quadratic $\C$-modules}]
A morphism $\phi: M \to M'$ of (proto)-quadratic $\C$-modules relative to $E$ is a pair $\phi=(\phi_e,\phi_{ee})$ where $\phi_e: M_e \to M'_e$ is a morphism of $\Lambda$-modules and $\phi_{ee}: M_{ee} \to M'_{ee}$ is a morphism of symmetric $\bar{\Lambda} \otimes \bar{\Lambda}$-modules which commute with the structure maps $\hat{H}$, $T$ and $P$.
\end{defi}

Composition of morphisms of (proto)-quadratic $\C$-modules is defined in the obvious way. This allows to give the following definition.
\begin{defi}
The category $PQMod^E_{\C}$ (resp.  $QMod^E_{\C}$)  is the category having as objects the proto-quadratic $\C$-modules (resp. quadratic $\C$-modules) and as maps the morphisms of proto-quadratic $\C$-modules (resp. quadratic $\C$-modules).
\end{defi}

\begin{rem} \label{rem-I1}
There is a fully-faithful functor $I_1: \overline{\Lambda}\text{-}Mod \to QMod^E_{\C}$ given by
$$I_1(M)=(T_{11}(cr_2U)(E,E)\otimes_{\Lambda}M \xrightarrow{\hat{H}} 0 \xrightarrow{T} 0  \xrightarrow{P} M).$$
In fact, $I_1(M)$ satisfies $(QM1)$  as $M$ is a $\overline{\Lambda}$-module, and $(QM2)$ is trivial.
\end{rem}

\subsection{Equivalence between $\mathcal{R}$-modules and quadratic $\C$-modules}
The aim of this section is to prove the equivalence between $\mathcal{R}$-modules and quadratic $\C$-modules. We begin by recalling the following notation.
\begin{nota} 
If $M$ is an abelian group equipped with an involution $t$, we denote by $M_{\mathfrak{S}_2}$ the coinvariants of the action of the symmetric group $\mathfrak{S}_2$ on $M$ given by $t$, i. e. 
$M_{\mathfrak{S}_2}=M/(1-t)M$. Furthermore, we denote by $\pi: M \to M_{\mathfrak{S}_2}$ the canonical projection.
\end{nota}

\begin{lm} \label{LoLoM}
Let $M$ be a symmetric $\overline{\Lambda} \otimes \overline{\Lambda}$-module with involution $T$. There is a natural isomorphism of groups:
$$\chi: (\overline{\Lambda} \otimes \overline{\Lambda}) \underset{(\overline{\Lambda} \otimes \overline{\Lambda}) \wr \mathfrak{S}_2}{ \otimes } M \to M_{\mathfrak{S}_2}$$
defined by:
$$\chi(\overline{\alpha} \otimes \overline{\beta} \otimes m)= \pi ((\overline{\alpha} \otimes \overline{\beta})m).$$
\end{lm}

\begin{proof}
$\chi$ is well defined since
\begin{eqnarray*}
\chi ((\overline{\alpha} \otimes \overline{\beta}) (\overline{a} \otimes \overline{b} + (\overline{c} \otimes \overline{d})t) \otimes m)&=& \chi((\overline{\alpha a} \otimes \overline{\beta b} + \overline{\beta d} \otimes \overline{\alpha c}) \otimes m) \mathrm{ \quad by\ \ref{th2th4}}\\
&=& \pi((\overline{\alpha a} \otimes \overline{\beta b})m) + \pi(( \overline{\beta d} \otimes \overline{\alpha c}) m) \\
&=&  \pi((\overline{\alpha a} \otimes \overline{\beta b})m) + \pi(T( \overline{\beta d} \otimes \overline{\alpha c}) m) \\
&=&  \pi((\overline{\alpha a} \otimes \overline{\beta b})m) + \pi(( \overline{\alpha c} \otimes \overline{\beta d}) T m) \\
&=&  \pi((\overline{\alpha a} \otimes \overline{\beta b}+( \overline{\alpha c} \otimes \overline{\beta d}) t) m) \\
&=& \pi((\overline{\alpha} \otimes \overline{\beta})(\overline{a} \otimes \overline{b}+ (\overline{c} \otimes \overline{d})t)m)\mathrm{ \quad by\ \ref{th2th4}}\\
&=& \chi (\overline{\alpha} \otimes \overline{\beta} \otimes  (\overline{a} \otimes \overline{b} + (\overline{c} \otimes \overline{d})t) m).
\end{eqnarray*}
Let $\chi': M \to (\overline{\Lambda} \otimes \overline{\Lambda}) \underset{(\overline{\Lambda} \otimes \overline{\Lambda}) \wr \mathfrak{S}_2}{ \otimes } M $ given by: $\chi'(m)= \overline{1} \otimes \overline{1} \otimes m$. Then:
\begin{eqnarray*}
\chi'(Tm)&=&  \overline{1} \otimes \overline{1} \otimes Tm\\
&=&\overline{1} \otimes \overline{1} \otimes ((\overline{1} \otimes \overline{1}) t m)\\
&=&(\overline{1} \otimes \overline{1}) (\overline{1} \otimes \overline{1}) t \otimes m\\
&=&  (\overline{1} \otimes \overline{1}) \otimes m  \mathrm{ \quad by\ \ref{th2th4}}\\
&=& \chi'(m)
\end{eqnarray*}
whence $\chi'$ factors through $M_{\mathfrak{S}_2}$ and provides an inverse map of $\chi$.
\end{proof}

\begin{thm} \label{Rmod=Cmodquad}
There is an equivalence of categories $\vartheta: \R \mbox{--\,}mod \to QMod^E_{\C}$ defined, using Lemma \ref{Rmodreform}, by
$$\vartheta(M)=( T_{11}(cr_2U)(E,E)\otimes_{\Lambda}M_{e} \hspace{1mm}\xrightarrow{\hat{H}}  \hspace{1mm}M_{ee}  \hspace{1mm}\xrightarrow{T}\hspace{1mm}M_{ee}  \hspace{1mm}
\xrightarrow{P}  \hspace{1mm}M_{e}
)$$
where $\hat{H}= h \circ (\overline{cr_2(t_2)} \otimes 1)$ and where $P=\overline{P} \pi$ with $\overline{P}: ({M_{ee}})_{\mathfrak{S}_2} \to M_{e}$ such that $\overline{P} \chi=p$.
\end{thm}
The proof of this theorem relies on the following lemma.
\begin{lm} \label{Lambdabarbarmod}
Let $M \in PQMod^E_{\C}$. Then the $\Lambda \ti$module structure on $M_e$ factors through $\overline{\overline{\Lambda}}$.
\end{lm}
\begin{proof}
Let $\varphi: \Lambda \to End(M_e)$, $\varphi(\lambda)(a)=\lambda a$ for $(\lambda,a) \in \Lambda \times M_e$. Then  for $\xi \in \C(E, E \vee E)$, we have:
\begin{eqnarray*}
cr_2(\varphi) \rho^2_{(1,2)}(\xi)(a)&=&(\nabla \xi)a-(r_1 \xi)a-(r_2 \xi)a \\
&=& P \hat{H}(t_{11} \rho^2_{(1,2)} (\xi) \otimes a) \mathrm{\quad by\ }(QM1)\\
&=& \tilde{\varphi} (t_{11}  \rho^2_{(1,2)} (\xi)) (a)
\end{eqnarray*}
where $\tilde{\varphi}: T_{11}cr_2U(E,E) \to End(M_e)$ is given by
$$\tilde{\varphi}(x)(a)=P \hat{H}(x \otimes a).$$
Thus, $cr_2(\varphi)=\tilde{\varphi}t_{11}$, whence $\varphi$ is a quadratic map (cf. section \ref{sec-quad-maps}). Now the assertion follows from Proposition \ref{univqumap}.
\end{proof}
\begin{proof}[Proof of Theorem \ref{Rmod=Cmodquad}]
Let $M_e$ be a $\Lambda \ti$module and $M_{ee}$ be a symmetric $\ol{\Lambda} \otimes \ol{\Lambda} \ti$module. First note that by Lemma \ref{Lambdabarbarmod}, $M_e$ actually is a $\ol{\ol{\Lambda}} \ti$module if $M_e$ and $M_{ee}$ are part of a quadratic $\C \ti$module. In view of Lemma \ref{LoLoM}, the data of a morphism of $\ol{\ol{\Lambda}} \ti$module $p$ is equivalent with the data of a $\Lambda \ti$equivariant morphism $P: M_{ee} \to M_e$ such that $PT=P$. Moreover, a morphism of symmetric $\ol{\Lambda} \otimes \ol{\Lambda} \ti$modules $h: T_2U(E \mid E) \otimes_{\Lambda}M_e \to M_{ee}$ is equivalent with a morphism of symmetric $\ol{\Lambda} \otimes \ol{\Lambda} \ti$modules $\hat{H}: T_{11}(cr_2U)(E,E) \otimes_{\Lambda} M_e \to M_{ee}$ in view of the isomorphism of symmetric bifunctors:
$$\overline{cr_2(t_2)}: T_{11}(cr_2 U_E) \to cr_2 (T_2U_E)$$
given in Theorem \ref{cr2t2}. Thus it remains to show that for $k=1,2$ the relations $(RMk)$ and $(QMk)$ are equivalent. This is based on the following relations for $\alpha, \beta \in \C(E,E)$, $\xi \in \C(E, E \vee E)$.
\begin{eqnarray}
\iota^2_{(1,2)}  \rho^2_{(1,2)}(( \alpha \vee \beta) \xi)&=& ( \alpha \vee \beta) \xi-i_1 r_1 ( \alpha \vee \beta) \xi- i_2 r_2 ( \alpha \vee \beta) \xi \\
&=& ( \alpha \vee \beta) \xi - i_1 \alpha r_1 \xi - i_2 \beta r_2 \xi  \label{(avb)xii}\\
&=& ( \alpha \vee \beta) \xi- ( \alpha \vee \beta)i_1 r_1  \xi- ( \alpha \vee \beta) i_2 r_2 \xi \\
&=& U_E(\alpha \vee \beta) \iota^2_{(1,2)}  \rho^2_{(1,2)}(\xi)\\
&=& \iota^2_{(1,2)}  U_E(\alpha | \beta)  \rho^2_{(1,2)}(\xi) \label{(avb)xi}.
\end{eqnarray}
Now consider $(RM1)$. On the one hand, we have
$$t_2 (\nabla(\alpha \vee \beta) \xi - \alpha r_1 \xi - \beta r_2 \xi)x=t_2(\nabla(\alpha \vee \beta) \xi - r_1 (\alpha \vee \beta) \xi - r_2(\alpha \vee \beta) \xi)x\,.$$
 On the other hand:
\begin{eqnarray*}
p(\ol{\alpha} \otimes \ol{\beta} \otimes h( \rho^2_{(1,2)}t_2(\xi) \otimes x))&=& P((\ol{\alpha} \otimes \ol{\beta}) h(t_2 \rho^2_{(1,2)}(\xi) \otimes x))\\
&=&Ph(T_2U_E( \alpha \mid \beta) t_2 \rho^2_{(1,2)}(\xi) \otimes x)
\end{eqnarray*}
since $h$ is $(\ol{\Lambda} \otimes \ol{\Lambda}) \wr \mathfrak{S}_2 \ti$equivariant, and by \ref{th4th3}; thus
\begin{eqnarray*}
p(\ol{\alpha} \otimes \ol{\beta} \otimes h( \rho^2_{(1,2)}t_2(\xi) \otimes x))&=&P \hat{H}(t_{11} \rho^2_{(1,2)}((\alpha \vee \beta) \xi) \otimes x) \mathrm{\quad by\ \ref{(avb)xi}}.
\end{eqnarray*}

Hence $(RM1)$ is equivalent with $(QM1)$. To see the equivalence between $(RM2)$ and $(QM2)$, just note that:
$$h( \rho^2_{(1,2)} t_2(\xi) \otimes p( \ol{\alpha}  \otimes \ol{\beta} \otimes m))=\hat{H}(t_{11}  \rho^2_{(1,2)}(\xi) \otimes P((\ol{\alpha}  \otimes \ol{\beta} )m))$$
while
$$(\ol{r_1 \xi \alpha} \otimes \ol{r_2 \xi \beta})m+ (\ol{r_1 \xi \beta} \otimes \ol{r_2 \xi \alpha}) Tm=(\ol{r_1 \xi} \otimes \ol{r_2 \xi })((\ol{\alpha} \otimes \ol{\beta})m)+(\ol{r_1 \xi} \otimes \ol{r_2 \xi }) T ((\ol{\alpha} \otimes \ol{\beta})m).$$
\end{proof}

\subsection{Quadratic $\C$-modules obtained from quadratic functors}

We here provide an explicit construction of the composite functor $$\xymatrix{\mathbb{S}_2: Quad(\C, Ab) \ar[r]^-{\sigma} & {\mathcal R}\mbox{--\,}mod \ar[r]^-{\theta} & QMod^E_{\C}\,.}$$


\begin{lm} \label{H^F}
For $F \in Quad(\C, Ab)$ there exists a natural transformation of functors $H^F:(T_{11}(cr_2U) \otimes_{\Lambda} F(E))\Delta_{\C} \to (cr_2F) \Delta_{\C}$  such that the following natural diagram is commutative  for  $X \in \C$:

$$\xymatrix{
T_{11}(cr_2U)(X,X) \otimes_{\Lambda} F(E) \ar[rr]^{(H^F)_X} && F(X |X)\\
& U(X|X) \otimes_{\Lambda} F(E) \ar@{->>}[ul] ^{(t_{11} \otimes 1)_X} \ar[ur]_{(cr_2(u'_F))_X}.&
}$$
\end{lm}

\begin{proof}
Recall that the cross-effect of a quadratic functor is a bilinear bifunctor. Hence the existence of $H^F$ follows from the universal property of $t_{11}$.

\end{proof}


\begin{prop} \label{MF}
For $F \in Quad(\C,Ab)$ we have 
$$\xymatrix{
\mathbb S_2(F) =(T_{11}(cr_2U)(E,E) \otimes_{\Lambda} F(E) \ar[r]^-{(H^F)_E} & F(E |E) \ar[r]^-{T^F}& F(E |E) \ar[r]^-{(S^F_2)_E} &F(E))\,.
}$$

\end{prop}

Checking this is straightforward going through the various definitions involved in the construction of $\sigma$ and $\theta$.






The following proposition formalizes the fact that $\mathbb{S}_2$ extends the functor $\mathbb{S}_1$.

\begin{prop} \label{4.15}
The following diagram is commutative
$$\xymatrix{
Quad(\C,Ab) \ar[r]^{\mathbb{S}_2} & QMod^E_{\C}\\
Lin(\C,Ab) \ar@{^{(}->}[u] \ar[r]_{\mathbb{S}_1} & \overline{\Lambda}-Mod \ar[u]_{I_1}
}$$

where $I_1$ is the functor defined in Remark \ref{rem-I1}.
\end{prop}


\section{Quadratic tensor product} \label{section5}
The left adjoint of the functor $\mathbb{S}_2: Quad(\C, Ab) \to QMod^E_{\C}$ is given by a construction which we call the quadratic tensor product. In fact, 
 a special case of a quadratic tensor product first appeared
 in \cite{Baues}, providing a left adjoint of a functor $Quad(Ab,Ab) \to QMod(\mathbb{Z})$ is defined explicitely by generators and relations; here $ QMod(\mathbb{Z})$ is the category of quadratic $\Z$-modules. Similarly, in \cite{Baues-Pira} a left adjoint of $Quad(Gr,Gr) \to Square$ is constructed by generators and relations; here $ Square$ is the category of square groups (see also section \ref{application}). In this paper, however, we give a more conceptual construction of the quadratic tensor product,  by means of a push-out diagram, in our general setting. We expect to generalize this construction  to polynomial functors of higher degree. A description of our quadratic tensor product in terms of generators and relations is  nevertheless provided   generalizing the constructions in \cite{Baues} and \cite{Baues-Pira}. We then  compute  the quadratic tensor product $E \otimes M$ for $M \in QMod_{\C}^E$ and  the cross-effect of $- \otimes M$ which are two essential tools in
  the proof of our main theorem in section \ref{section-thm}.

\subsection{Definition}
We start  with the following straightforward lemma.
\begin{lm} \label{invol}
If $N$ is a symmetric $\Lambda \otimes \Lambda$-module with involution $T$ and $X \in \C$, the group $M=((T_1U)(X) \otimes (T_1U)(X)) \otimes_{\Lambda \otimes \Lambda} N$ is equipped with an involution $t$ defined by
$$t(x \otimes y \otimes n)=y \otimes x \otimes T(n)$$
for $x,y \in (T_1U)(X)$ and $n \in N$.
\end{lm}
To define the quadratic tensor product, we need to consider the coinvariants by the action of the symmetric group $\mathfrak{S}_2$. 
\begin{defi} \label{quad-t-p}{\bf (Quadratic tensor product)}
Let $M$ be a proto-quadratic $\C$-module relative to $E$ and $X \in \C$. The quadratic tensor product $X \otimes M \in Ab$ is the push-out of the following diagram of abelian groups:
$$\xymatrix{
((U(X) \otimes U(X)) \otimes_{\Lambda \otimes \Lambda} U(E |E)\otimes_{\Lambda} M_e) \oplus (U(X) \otimes M_{ee}) \ar[rr]^-{\phi=(\phi_1, t_2 \otimes P)} \ar[d]_-{\psi=(\psi_1, \pi(\delta \otimes 1))} &&(T_2U)(X) \otimes_{\Lambda} M_e \ar[d]^{\hat{\psi}}\\
(((T_1U)(X) \otimes (T_1U)(X)) \otimes_{\Lambda \otimes \Lambda} M_{ee})_{\mathfrak{S}_2} \ar[rr]_-{\hat{\phi}}& &X \otimes M
}$$
where $\psi_1=\pi(t_1 \otimes t_1 \otimes \hat{H}(t_{11} \otimes 1))$, $\delta(f)=(t_1f) \otimes (t_1f)$ for $f \in \C(E,X)$ and $\phi_1(f \otimes g \otimes x \otimes a)=t_2 S^U_2U(f|g)(x) \otimes a$ for $f, g \in \C(E,X)$, $x \in U(E |E)$ and $a \in M_e$.
\end{defi}
In the following proposition we give a description of the quadratic tensor product $X \otimes M$ by generators and relations.

\begin{prop}
Let $M$ be a proto-quadratic $\C$-module relative to $E$ and $X \in \C$. The quadratic tensor product $X \otimes M \in Ab$ is the abelian group generated by the symbols
$$f \otimes a,\ f \in \C(E,X),\ a \in M_e$$
$$[f,g] \otimes m, \ f,g \in \C(E,X),\ m \in M_{ee}$$
subject to the following relations:
\begin{enumerate}
\item $(f \beta) \otimes a = f \otimes (\beta a) \ for\ \beta \in \C(E,E)$
\item $f \otimes (a+b)=f \otimes a+ f \otimes b$

\item $(1,1,1)\xi \otimes a -(1,1,0) \xi \otimes a - (1,0,1) \xi \otimes a -(0,1,1) \xi \otimes a + (1,0,0) \xi \otimes a+ (0,1,0) \xi \otimes a + (0,0,1) \xi \otimes a =0 \ for\   \xi \in \C(E, X \vee X \vee X)
$
\item $[f \alpha, g \beta] \otimes m=[f,g] \otimes (\overline{\alpha} \otimes \overline{\beta})m \ for\ \alpha, \beta \in \C(E,E)$
\item $[f,g] \otimes (m+n)=[f,g] \otimes m+ [f,g] \otimes n$
\item $[\nabla^2_*(\xi),g]\otimes m=[r_{1*}^2(\xi),g] \otimes m+[r_{2*}^2(\xi),g ] \otimes m  \ for\ \xi \in \C(E,X \vee X)$
\item $[f,g] \otimes m= [g,f] \otimes T(m)$
\item $[f,f] \otimes m=f \otimes P(m)$
\item $(f,g) \gamma \otimes a=fr^2_1 \gamma \otimes a + g r^2_2 \gamma \otimes a+[f,g] \otimes \hat{H}(t_{11}\rho_{(1,2)}^2(\gamma) \otimes a) \ for\ \gamma \in \C(E, E \vee E)$.
\end{enumerate}
\end{prop}
\begin{proof}
The symbol $f \otimes a$ corresponds to a generator of $U(X) \otimes_{\Lambda} M_e$ and $[f,g] \otimes m$ corresponds to a generator of $(U(X) \otimes_{\Lambda} U(X)) \otimes_{\Lambda \otimes \Lambda}M_{ee}$. 

For the elements $f \otimes a$, relation $(1)$ corresponds to the fact that the tensor product is taken over $\Lambda$, $(2)$ translates the linearity in $M_e$, $(3)$ corresponds to the fact that the element $f \otimes a$ is in $T_2U(X) \otimes_{\Lambda} M_e$ where we use Proposition \ref{quadratization} which describes $T_2F$ as a quotient of $F$.

For the elements $[f,g] \otimes m$, $(4)$ corresponds to the fact that the tensor product is taken over $\Lambda \otimes \Lambda$, $(5)$ translates the linearity in $M_{ee}$, $(6)$ (also using $(7)$) corresponds to the fact that the element $[f,g] \otimes m$ lies in $((T_1U)(X) \otimes (T_1U)(X)) \otimes_{\Lambda \otimes \Lambda} M_{ee}$, and $(7)$ translates the fact that we take the coinvariants by the action of $\mathfrak{S}_2$.

Finally, $(8)$ and $(9)$ correspond to the fact that the diagram in Definition \ref{quad-t-p}  is a pushout.
\end{proof}

We need the following technical lemma.

\begin{lm} \label{push-out-eq}
The quadratic tensor product $X \otimes M$ is equal to the pushout $\Pi$ of the following diagram of abelian groups

\noindent \makebox[14.7cm]{ \makebox[0mm]{
\begin{minipage}{17cm}
\small
  $$\xymatrix{
((T_1UX \otimes T_1UX) \underset{\Lambda \otimes \Lambda}{\otimes} T_{11}(cr_2U)(E ,E)
 \underset{\Lambda}{\otimes} M_e)_{\mathfrak{S}_2} \oplus (UX \otimes M_{ee})\ar[rrr]^-{\overline{\phi}=(\overline{\phi'_1 \otimes 1}, t_2 \otimes P)}  \ar[d]_-{\overline{\psi}=( \overline{\psi_1}, \pi(\delta \otimes 1))} &&& T_2UX  \underset{\Lambda}{ \otimes} M_e \ar[d]^{\hat{\psi}}\\
 ((T_1UX \otimes T_1UX) \otimes_{\Lambda \otimes \Lambda} M_{ee})_{\mathfrak{S}_2} \ar[rrr]_-{\hat{\phi}} &&& \Pi}$$
\end{minipage}
}\rule[-11mm]{0mm}{3mm}}

where  $UX:=U(X)$, $T_iUX:=T_iU(X)$ for $i \in \{ 1,2\}$,  $\overline{\psi_1}=\overline{1 \otimes 1 \otimes \hat{H}}$ and $\phi'_1$ is the following composite map:
$$\xymatrix{
(T_1U(X) \otimes T_1U(X)) \otimes_{\Lambda \otimes \Lambda} T_{11}(cr_2U)(E ,E)
 \ar@{-->}[r]^-{\phi'_1 } \ar[d]_-{1 \otimes 1 \otimes \overline{cr_2(t_2)}} &T_2U(X) \\
(T_1U(X) \otimes T_1U(X))  \underset{\Lambda \otimes \Lambda}{ \otimes} cr_2(T_2U)(E,E) \ar[r]_-{u'_{cr(T_2U)}}& cr_2(T_2U)(X,X) \ar[u]_{S_2^{T_2U}}.
}$$
\end{lm}

\begin{proof}
We have the following surjection 
$$\xymatrix{
(UX \otimes UX)  \underset{\Lambda \otimes \Lambda}{ \otimes}U(E |E) \underset{ \Lambda}{ \otimes}M_e \ar[rr]^-{\pi(t_1 \otimes t_1 \otimes t_{11} \otimes 1)} &&
(((T_1U)X \otimes (T_1U)X) \underset{\Lambda \otimes \Lambda}{ \otimes} T_{11}(cr_2U)(E ,E) \underset{\Lambda}{ \otimes} M_e)_{\mathfrak{S}_2}
}$$
 which verifies
 $$\psi_1=\pi(t_1 \otimes t_1 \otimes \hat{H}(t_{11} \otimes 1))=\pi (1 \otimes 1 \otimes \hat{H}) (t_1 \otimes t_1 \otimes t_{11} \otimes 1).$$
 We first check that $ \phi_1=(\phi'_1 \otimes 1) (t_1 \otimes t_1 \otimes t_{11} \otimes 1).$ For $f,g \in \C(E,X)$ and $x \in U(E|E)$ we have:
 \begin{eqnarray*}
\phi'_1  (t_1 \otimes t_1 \otimes t_{11} )(f \otimes g \otimes x) &= &S_2^{T_2U}u'_{cr_2(T_2U)}(1 \otimes 1 \otimes \overline{cr_2(t_2)})(t_1(f) \otimes t_1(g) \otimes t_{11}(x))\\
&=& S_2^{T_2U} cr_2(T_2U)(f,g) cr_2(t_2) (x) \mathrm{\qquad by\ definition\ of\ } u'_{cr_2(T_2U)}\\
&=& (T_2U)(\nabla^2) \iota^2_{(1,2)} cr_2(T_2U)(f,g) cr_2(t_2) (x) \mathrm{\qquad by\ definition\ of\ } S_2^{T_2U} \\
&=& (T_2U)(\nabla^2) (T_2U)(f \vee g) t_2 \iota^2_{(1,2)} (x)\\
&=& t_2\  U(\nabla^2) U(f \vee g) \iota^2_{(1,2)} (x)\\
&=& t_2 \ U(\nabla^2)  \iota^2_{(1,2)}\  U(f | g)  (x)\\
&=& t_2\  S_2^{U}  U(f | g) (x).
\end{eqnarray*}
It remains to check that $\phi'_1 $ factors through the coinvariants, that is $\phi'_1( t_1(f) \otimes t_1(g) \otimes t_{11} x)=\phi'_1( t(t_1(f) \otimes t_1(g) \otimes t_{11} x))$. We have:
 \begin{eqnarray*}
\phi'_1 ( t(t_1(f) \otimes t_1(g) \otimes t_{11} x))&=&\phi'_1 (t_1(g) \otimes t_1(f) \otimes t_{11} (\iota^2_{(1,2)})^{-1} U( \tau) \iota^2_{(1,2)} x)\\
&=& t_2\  S_2^{U}  U(g | f) (\iota^2_{(1,2)})^{-1} U( \tau) (\iota^2_{(1,2)} x) \mathrm{\ by\ the\ previous\ calculation}\\
&=& t_2\  U(\nabla^2)   \iota^2_{(1,2)} U(g | f) (\iota^2_{(1,2)})^{-1} U( \tau) (\iota^2_{(1,2)} x) \\
&=& t_2\  U(\nabla^2)    U(g  \vee f) \iota^2_{(1,2)} (\iota^2_{(1,2)})^{-1} U( \tau) (\iota^2_{(1,2)} x) \\
&=& t_2\  U(\nabla^2)    U(g  \vee f)U( \tau) (\iota^2_{(1,2)} x) \\
&=& t_2\  U(\nabla^2)  U( \tau)  U(f \vee g) (\iota^2_{(1,2)} x) \\
&=& t_2\  U(\nabla^2 \tau)  U(f \vee g) (\iota^2_{(1,2)} x) \\
&=& t_2\  U(\nabla^2)   U(f \vee g) (\iota^2_{(1,2)} x) \\
&=& t_2\  U(\nabla^2)   \iota^2_{(1,2)} U(f | g) ( x) \\
&=& t_2\  S_2^{U}  U(f | g) ( x) \\
&=& \phi'_1 ( t_1(f) \otimes t_1(g) \otimes t_{11} x).
\end{eqnarray*}

\end{proof}

\subsection{The quadratic tensor product defines a quadratic functor}

The aim of this section is to prove the following result.

\begin{prop} \label{qtp-is-quad}
For $M \in PQMod^E_{\C}$ the functor given by the quadratic tensor product: $- \otimes M: \C \to Ab$ is a quadratic functor.
\end{prop}

The proof of this proposition relies on the following lemma.

\begin{lm} \label{i}
For $M$ a proto-quadratic $\C$-module, the following natural transformation of functors from $\C$ to $Ab$  is surjective: 
$$cr_2(\hat{\phi}): cr_2((((T_1U(-) \otimes T_1U(-)) \otimes_{\Lambda \otimes \Lambda} M_{ee} )\Delta_{\C})_{\mathfrak{S}_2} ) \to cr_2(- \otimes M);$$
here $\hat{\phi}$ is the map in the push-out diagram in Definition \ref{quad-t-p}.
\end{lm}

\begin{rem}
In Theorem \ref{cross-pt-quad} below  we give an improved version of this lemma.
\end{rem}

\begin{proof}[Proof of Lemma \ref{i}]
To simplify notation we write $\mathbb{T}_{11}(M_{ee}) \Delta_{\C}$ instead of $(((T_1U)(-) \otimes (T_1U)(-)) \otimes_{\Lambda \otimes \Lambda} M_{ee}) \Delta_{\C}$. Recall that the functor $\mathbb{T}_{11}$ is the functor defined in \ref{T11}.

By the universal property of a push-out, we obtain the existence of a map $f: X \otimes M \to coker(\phi)$ making the following diagram of abelian groups commutative:
$$\xymatrix{
(UX \otimes UX) \otimes_{\Lambda \otimes \Lambda} U(E |E)\otimes_{\Lambda} M_e \oplus UX \otimes M_{ee} \ar[r]^-{\phi} \ar[d]_-{\psi} &(T_2U)(X) \otimes_{\Lambda} M_e \ar[d]^{\hat{\psi}} \ar@{->>}[ddr]\\
(((T_1U)(X) \otimes (T_1U)(X)) \otimes_{\Lambda \otimes \Lambda} M_{ee})_{\mathfrak{S}_2} \ar[r]_-{\hat{\phi}} \ar[drr]_{0} &X \otimes M \ar@{-->}[dr]^f\\
& & coker(\phi).
}$$
We deduce the existence of a natural exact sequence
\begin{equation} \label{ses}
(\mathbb{T}_{11}(M_{ee}) \Delta_{\C})_{\mathfrak{S}_2}(X) \xrightarrow{\hat{\phi}} X \otimes M \xrightarrow{f} coker(\phi) \to 0.
\end{equation}

In the sequel  we show that the functor $coker(\phi)$ is linear.
Recall that  $\phi_1= \phi'_1 (t_1 \otimes t_1 \otimes t_{11}) \otimes 1 $  where $(\phi'_1 (t_1 \otimes t_1 \otimes t_{11}))_X$ is given by the following composition: 

\xymatrix{
(U(X) \otimes U(X)) \otimes_{\Lambda \otimes \Lambda} U(E |E) \ar[r]^-{\alpha} &cr_2(T_2U)(X,X)\ar[r]^-{(S_2^{T_2U})_X}&(T_2U)(X) 
}
where $\alpha=(u'_{cr_2(T_2U)})_{X,X}(1 \otimes 1 \otimes \overline{cr_2(t_2)})(t_1 \otimes t_1 \otimes t_{11})$.

The functor $T_2U$ is quadratic by definition of $T_2$, so $cr_2(T_2U)$ is a bilinear functor, and so is $(T_1U(-) \otimes T_1U(-)) \otimes_{\Lambda \otimes \Lambda} cr_2(T_2U)(E,E)$. Since $(u'_{cr(T_2U)})_{E,E}$ is an isomorphism according to Theorem \ref{bilin}, we deduce by Proposition \ref{car-bipoly} that $(u'_{cr(T_2U)})$ is a natural equivalence and $(u'_{cr(T_2U)})_{X,X}$ is an isomorphism. Furthermore $1 \otimes 1 \otimes \overline{cr_2(t_2)}$ and $t_1 \otimes t_1 \otimes t_{11} $ are surjective by construction.
It follows that
$$coker((\phi_1)_X)=coker((S_2^{T_2U})_X)=T_1(T_2 U)(X)$$
by Proposition \ref{Tn-coker}. 

Since $Im(\phi_1) \subset Im(\phi)$ we see that $coker(\phi)$ is a quotient of $coker(\phi_1)=T_1(T_2 U) \otimes_{\Lambda} M_e$ which is a linear functor. Thus $coker(\phi)$ is a linear functor by Proposition \ref{stable}.
But  the cross-effect functor is exact by Proposition \ref{exact},  so applying it  to sequence (\ref{ses}) shows that $cr_2(\hat{\phi})$ is pointwise surjective.

\end{proof}

\begin{proof}[Proof of Proposition \ref{qtp-is-quad}]
The functor  $\mathbb{T}_{11}(M_{ee})$ is bilinear and bireduced, so by Lemma \ref{diago}, $\mathbb{T}_{11}(M_{ee}) \Delta_{\C}$ is quadratic. By Proposition \ref{bi-stable}, we deduce that $(\mathbb{T}_{11}(M_{ee}) \Delta_{\C})_{\mathfrak{S}_2}$ is quadratic. Consequently $cr_2((\mathbb{T}_{11}(M_{ee})\Delta_{\C})_{\mathfrak{S}_2})$ is bilinear. Since $cr_2(- \otimes M)$ is a quotient of $cr_2((\mathbb{T}_{11}(M_{ee})\Delta_{\C})_{\mathfrak{S}_2})$ by Lemma \ref{i}, the functor $cr_2(- \otimes M)$ is also bilinear, so $- \otimes M$ is quadratic.

\end{proof}

Proposition \ref{qtp-is-quad} leads to the following definition.

\begin{defi}
The functor
$$\mathbb{T}_2: QMod^E_{\C} \to Quad(\C,Ab)$$
is defined as follows: for $M \in QMod^E_{\C}$ let $\mathbb{T}_2(M)= - \otimes M$,  and for a morphism of quadratic $\C-$ modules $(\phi_e, \phi_{ee}): M \to N$ let $\mathbb{T}_2(\phi_e, \phi_{ee})=\varphi$ where $\varphi_X: X \otimes M \to X \otimes N$ is given by the universal property of a push-out.
\end{defi}

\begin{rem}
The functor $\mathbb{T}_2$ extends the tensor product functor $\mathbb{T}_1$ in the following sense: it improves the approximation of a quadratic functor $F$ by $\overline{u'}: T_2U(-) \otimes_{\Lambda} F(E) \to F$ by taking into account the cross-effect of $F$ in "amalgamating" $T_2U(X) \otimes_{\Lambda} F(E)$ with the image of $F(X|X)$ under $(S^F_2)_X$. We expect that this idea can be extended to polynomial functors of higher degree.
\end{rem}

\begin{rem}
Since the functor $- \otimes M$ is quadratic by Proposition \ref{qtp-is-quad}, the computation of $(- \otimes M)(E^{\vee n})$ for $n \geq 1$ reduces to computing $(- \otimes M)(E)$ and $cr_2(- \otimes M) (E,E)$, see Proposition \ref{ce-prop}.
\end{rem}

\subsection{Computation of the quadratic tensor product $E \otimes M$}

In section \ref{section-thm}, in order to obtain the desired  equivalence between quadratic functors and quadratic $\C$-modules, we need to compute $E \otimes M$ for $M \in QMod^E_{\C}$, as follows.

\begin{prop} \label{prop}
For $M \in PQMod^E_{\C}$,  there exists a natural isomorphism of abelian groups
$$E \otimes M \xrightarrow{\simeq} M_e.$$
\end{prop}

In order to define this isomorphism we need the following lemmas.

\begin{lm} \label{lm1}
For $M \in PQMod^E_{\C}$, there exists an isomorphism $\overline{\mu_e}: (T_2U)(E) \otimes_{\Lambda} M_e \to M_e$ making the following diagram commutative
$$\xymatrix{
U(E) \otimes_{\Lambda} M_e \ar[rr]^{\mu_e} \ar@{->>}[dr]_{t_2 \otimes 1}&&M_e\\
&(T_2U)(E) \otimes_{\Lambda} M_e \ar@{..>}[ru]_{\overline{\mu_e}}
}$$
where $\mu_e$ is the canonical isomorphism. In particular, the map  $t_2 \otimes 1$ is an isomorphism.
\end{lm}
\begin{proof}
To prove that $\overline{\mu_e}$ exists we have to prove that $\mu_e(Ker(t_2 \otimes 1))=0$. By Proposition \ref{T2byT11} we have $$T_2(F)=coker(ker(cr_2F(X |X) \xrightarrow{t_{11}} (T_{11}cr_2F)(X |X)) \xrightarrow{S_2^U} F(X) ).$$
So we have $ker(t_2)=S_2^U (ker(t_{11}))$ and
$$\mu_e(ker(t_2 \otimes 1))=\mu_e(S_2^U \otimes 1)(ker(t_{11}) \otimes M_e)=P \hat{H} (t_{11} \otimes Id)
(ker(t_{11}) \otimes M_e)=0$$
where the second equality is given by Proposition \ref{QM1}.

Since $t_2 \otimes 1$ is surjective and $\mu_e$ is an isomorphism we deduce that $\overline{\mu_e}$ is an isomorphism.

As a consequence we obtain that $t_2 \otimes 1$ is also an isomorphism.
\end{proof}

\begin{lm} \label{lm2}
For $M \in PQMod^E_{\C}$, there exists a morphism $\overline{P}: (M_{ee})_{\mathfrak{S}_2} \to M_e$ making the following diagram commutative
$$\xymatrix{
 M_{ee} \ar[rr]^{P} \ar@{->>}[dr]_{\pi}&&M_e\\
& (M_{ee})_{\mathfrak{S}_2}  \ar@{..>}[ru]_{\overline{P}}
}$$

\end{lm}

\begin{proof}
This is due to the relation $PT=P$ in Definition \ref{quad-C}.
\end{proof}

\begin{lm} \label{lm3}
There exists an isomorphism 
$$\overline{\mu_{ee}}: ((T_1U)(E) \otimes (T_1U)(E)) \otimes_{\Lambda \otimes \Lambda} M_{ee})_{\mathfrak{S}_2} \to (M_{ee})_{\mathfrak{S}_2} $$ making the following diagram commutative
$$\xymatrix{
 ((T_1U)(E) \otimes (T_1U)(E)) \otimes_{\Lambda \otimes \Lambda} M_{ee} \ar[rr]^-{\mu_{ee}} \ar@{->>}[d]_{\pi}&&M_{ee} \ar@{->>}[d]_{\pi}\\
(((T_1U)(E) \otimes (T_1U)(E)) \otimes_{\Lambda \otimes \Lambda} M_{ee})_{\mathfrak{S}_2}  \ar[rr]^-{\overline{\mu_{ee}}}&& (M_{ee})_{\mathfrak{S}_2} 
}$$
where 
$\xymatrix{
 ((T_1U)(E) \otimes (T_1U)(E)) \otimes_{\Lambda \otimes \Lambda} M_{ee} \ar[r]^-{\mu_{ee}}&M_{ee} }$ is the canonical isomorphism. 
\end{lm}

\begin{proof}
This is immediate from the fact that the canonical isomorphism
$$\xymatrix{
 ((T_1U)(E) \otimes (T_1U)(E)) \otimes_{\Lambda \otimes \Lambda} M_{ee} \ar[r]^-{\mu_{ee}}&M_{ee} }$$ 
 is compatible with the involutions. 
\end{proof}

Proposition \ref{prop} is a direct consequence of the following lemma.

\begin{lm} \label{lm33}
For $M \in PQMod^E_{\C}$, there exists an isomorphism
$$(\overline{P} \  \overline{\mu_{ee}}, \overline{\mu_e}): E \otimes M \to M_e$$
such that $(\overline{P} \  \overline{\mu_{ee}}, \overline{\mu_e}) \hat{\phi}= \overline{P} \  \overline{\mu_{ee}}$ and $(\overline{P} \  \overline{\mu_{ee}}, \overline{\mu_e}) \hat{\psi}= \overline{\mu_e}$ where the maps $\overline{P}$, $\overline{\mu_{ee}}$  and $\overline{\mu_e}$ are defined in Lemma \ref{lm2},  \ref{lm3}  and  \ref{lm1}  respectively and the maps $\hat{\phi}$ and $\hat{\psi}$ appear in the pushout diagram of Proposition \ref{quad-t-p}.

\end{lm}

\begin{proof}
To prove that the morphism $(\overline{P} \  \overline{\mu_{ee}}, \overline{\mu_e})$ exists, it is sufficient to prove that $\overline{P} \  \overline{\mu_{ee}} \psi=\overline{\mu_e} \phi$ by the universal property of the push-out.

For $f \in \C(E,E)$ and $m \in M_{ee}$ we have:
\begin{eqnarray*}
\bar{P} \,\overline{\mu_{ee}}\, \pi (\delta \otimes  1)(f\otimes m) &=& P \mu_{ee}
(\bar{f}\otimes \bar{f} \otimes m) \mathrm{\qquad by\ Lemmas\ \ref{lm2}\ and\ \ref{lm3}}\\
&=& P((\bar{f} \otimes \bar{f}). m) \mathrm{\qquad by\ definition\ of\ } \mu_{ee} \\
&=& fP(m)\mathrm{\qquad by\ Definition\ \ref{proto-quad-C}}\\
&=& \mu_e(f \otimes P(m)) \mathrm{\qquad by\ definition\ of\ } \mu_{e}\\
&=& \overline{\mu_e}\,(t_2 \otimes Id) (f \otimes P(m))  \mathrm{\qquad by\ Lemma\ \ref{lm1}}\\
&=& \overline{\mu_e}\,(t_2 \otimes P) (f \otimes m).  
\end{eqnarray*}

For $f, g \in \C(E,E)$, $x \in U(E|E)$, $a \in M_e$ we have:
$$\begin{array}{l}
\bar{P} \overline{\mu_{ee}} \pi (t_1 \otimes t_1 \otimes \hat{H} (t_{11} \otimes 1)) (f \otimes g \otimes x \otimes a)\\
= P \mu_{ee}(\bar{f}\otimes \bar{g} \otimes  \hat{H} (t_{11} (x)\otimes a)) \mathrm{\qquad by\ Lemmas\ \ref{lm2}\ and\ \ref{lm3}}\\
=  P((\bar{f} \otimes \bar{g}) \hat{H} (t_{11} (x)\otimes a)) \mathrm{\qquad by\ definition\ of\ } \mu_{ee} \\
=P \hat{H}((\bar{f} \otimes \bar{g}) (t_{11} (x)\otimes a)) \mathrm{\qquad since\ } \hat{H} \mathrm{\ is\ a\ morphism\ of\  } \bar{\Lambda} \otimes \bar{\Lambda} \mathrm{-modules\ by\ Definition\ \ref{proto-quad-C}}\\ 
 =P \hat{H}( t_{11} (U(f \mid g)(x)) \otimes a) \mathrm{\qquad by\ the\ structure\ of\ } \bar{\Lambda} \otimes \bar{\Lambda} \mathrm{-module\ of\ } T_{11}(cr_2U)(E,E)     \\
= P \hat{H}( t_{11} \otimes Id)( U(f \mid g)(x) \otimes a)\\ 
= \mu_e (S^U_2 \otimes Id) (U(f \mid g)(x) \otimes a) \mathrm{\qquad by\ Proposition\ \ref{QM1}}\\ 
= \overline{\mu_e}(t_2 \otimes Id)  (S^U_2 \otimes Id) (U(f \mid g)(x) \otimes a)  \mathrm{\qquad by\ Lemma\ \ref{lm1}}\ \\
= \overline{\mu_e}(t_2 S^U_2 U(f |g)(x) \otimes a).\\
\end{array}$$

Hence the morphism $(\overline{P} \  \overline{\mu_{ee}}, \overline{\mu_e}): E \otimes M \to M_e$ exists.

For $m \in M_{ee}$, we have:
$$ \overline{\mu_{ee}} \pi (\delta \otimes 1)(1 \otimes m)=\bar{m}.$$
Since $\overline{\mu_{ee}} $ is an isomorphism by Lemma \ref{lm3} we deduce that $ \pi (\delta \otimes 1)$ is surjective. Consequently $\psi$ is surjective and by general properties of push-out diagrams we obtain that $\hat{\psi}: (T_2U)(E) \otimes M_e \to E \otimes M$ is surjective. Now, since $(\overline{P} \  \overline{\mu_{ee}}, \overline{\mu_e})\hat{\psi}=\overline{\mu_e}$ is an isomorphism by Lemma \ref{lm1}, it follows that $(\overline{P} \  \overline{\mu_{ee}}, \overline{\mu_e})$ is an isomorphism.
\end{proof}

In particular, by Proposition \ref{MF}, for $F$ a quadratic functor, $\mathbb S_2(F)$ is a quadratic  $\C$-module, so we can apply the previous lemma to  $\mathbb S_2(F)$ to obtain:

\begin{lm} \label{F(E)-iso}
For $F \in Quad(\C, Ab)$ the morphism
$$((\overline{S^F_2})_E \  \overline{\mu_{ee}}, \overline{\mu_e}): E \otimes \mathbb S_2(F) \to F(E).$$
is an isomorphism.
\end{lm}

\subsection{Computation of the cross-effect of $- \otimes M$}

The aim of this section is to prove the following theorem which allows to compute the cross-effect of $-\otimes M$, in Corollary \ref{gamma}.

\begin{thm} \label{cross-pt-quad}
For $\C$ a pointed theory generated by $E$ and $M$ a proto-quadratic $\C$-module relative to $E$, the following natural transformation of functors is an equivalence if and only if $M$ is a quadratic $\C$-module.
$$cr_2(\hat{\phi}): cr_2((((T_1U(-) \otimes T_1U(-)) \otimes_{\Lambda \otimes \Lambda} M_{ee} )\Delta_{\C})_{\mathfrak{S}_2} ) \to cr_2(- \otimes M)$$
Here $\hat{\phi}$ is the map in the push-out diagram in Definition \ref{quad-t-p}.
\end{thm}

The proof of this theorem relies on the following lemmas.

\begin{lm} \label{bilin2} Let $(B,T)$ be a symmetric bifunctor from $\C$ to $Ab$ such that  $B$ is  bilinear bireduced, see Definition \ref{symbifdef}. Then
 the  map given by
$$(\iota^2_{(1,2)})^{-1} \pi B(i_1^2, i_2^2) \hspace{1mm} : \hspace{1mm} B(X,Y) \hspace{1mm} \xrightarrow{\phantom{aa}} \hspace{1mm} cr_2(B \Delta_{\C})_{\mathfrak{S}_2}(X ,Y)$$
for $X, Y \in \C$, is a natural equivalence.
\end{lm}

\begin{proof}
For $X,Y \in \C$, we have $B \Delta_{\C}(X \vee Y)=B(X,X) \oplus B(X,Y) \oplus B(Y,X) \oplus B(Y,Y)$ since $B$ is a bilinear functor. For $X_1=X$, $X_2=Y$ and $p,q=1,2$ we have $T_{X_1\vee X_2,X_1\vee X_2}B(i^2_p,i^2_q) =B(i^2_q,i^2_p)T_{X_p,X_q}$. Thus $(B \Delta_{\C})_{\mathfrak{S}_2}(X \vee Y)=B(X,X)_{\mathfrak{S}_2} \oplus B(Y,Y)_{\mathfrak{S}_2} \oplus (B(X,Y) \oplus B(Y,X))_{\mathfrak{S}_2}$ and $cr_2(B \Delta_{\C})_{\mathfrak{S}_2}(X,Y)= (B(X,Y) \oplus B(Y,X))_{\mathfrak{S}_2}$ where the action of $\mathfrak{S}_2$ on $B(X,Y) \oplus B(Y,X)$ is given by $t(x,y)=(T_{Y,X}(y),T_{X,Y}(x))$.
 We have $(\iota^2_{(1,2)})^{-1} \pi B(i^2_1, i^2_2)(x)=\overline{(x,0)}$. An inverse of this map is defined as follows: let $f:B(X,Y) \oplus B(Y,X) \to B(X,Y)$ be given by $f(x,y)= x+T_{Y,X}(y)$. Then $f(t(x,y))=
 T_{Y,X}(y) +  T_{Y,X} T_{X,Y}(x)=f(x,y)$.  So $f$ defines a map $\bar{f} :(B(X,Y) \oplus B(Y,X))_{\mathfrak{S}_2}\to B(X,Y)$, and one easily checks that $\bar{f}$ is the inverse of $(\iota^2_{(1,2)})^{-1} \pi B(i^2_1, i^2_2)$.
\end{proof}

\begin{lm} \label{DC-cross}
For a $\overline{\Lambda} \otimes \overline{\Lambda}$-module $A$, let $\mu_A: \mathbb{T}_{11}(A)(E,E) \to A$ denote the canonical isomorphism. Then the following diagram is commutative:
$$\xymatrix{
N \oplus U(E |E) \otimes M_{ee} \ar[rr]^-{k_1}_-{\simeq} \ar[d]^-{k_2}&& cr_2((\mathbb{T}_{11}(N) \Delta_{\C})_{\mathfrak{S}_2})(E,E) \oplus U(E |E) \otimes M_{ee} \ar[d]^-{cr_2(\overline{\psi})}\\
M_{ee} \ar[rr]_-{k_3}^-{\simeq}&& cr_2((\mathbb{T}_{11}(M_{ee}) \Delta_{\C})_{\mathfrak{S}_2})(E,E)
}$$
where: $N=T_{11}(cr_2(U))(E,E) \otimes_{\Lambda} M_e$,     $k_1=(\iota^2_{(1,2)})^{-1} \pi \mathbb{T}_{11}(N)(i^2_1, i^2_2)(\mu_N)^{-1} \oplus 1$, $k_2=(\hat{H}, k_2')$ with $k'_2 (\rho^2_{(1,2)}(\xi) \otimes m)=(r^2_1 \xi \otimes r^2_2 \xi) . (m+Tm)$ for $\xi \in \C(E, E \vee E)$ and $m \in M_{ee}$, and $k_3=(\iota^2_{(1,2)})^{-1} \pi \mathbb{T}_{11}(M_{ee})(i^2_1, i^2_2)(\mu_{M_{ee}})^{-1}$.
\end{lm}
\begin{proof}
The morphisms $k_1$ and $k_3$ are isomorphisms since $\mu_{M_{ee}}$ and $\mu_N$ are the canonical isomorphisms and $(\iota^2_{(1,2)})^{-1} \pi \mathbb{T}_{11}(N)(i^2_1, i^2_2)$ and $(\iota^2_{(1,2)})^{-1} \pi \mathbb{T}_{11}(M_{ee})(i^2_1, i^2_2)$ are isomorphisms by Lemma \ref{bilin2}. The fact that $k'_2$ is welldefined follows from the fact that $k_3$ is an isomorphism and from commutativity of the diagram which we prove now.

By Lemma \ref{push-out-eq}, $cr_2(\overline{\psi})=(cr_2(\overline{1 \otimes 1 \otimes \hat{H}}), cr_2(\pi ( \delta \otimes 1)))$. 

The commutativity of the diagram on $N$ follows from the naturality of $\mu$ and by the fact that $\mathbb{T}_{11}(-) (-,-)$ is a trifunctor.

To prove the commutativity of the diagram on $ U(E |E) \otimes M_{ee} $ we consider the injection $\iota^2_{(1,2)}:  cr_2((\mathbb{T}_{11}(M_{ee}) \Delta_{\C})_{\mathfrak{S}_2})(E,E) \to (\mathbb{T}_{11}(M_{ee})\Delta_{\C})_{\mathfrak{S}_2})(E \vee E)$ and we prove that $\iota^2_{(1,2)} \, cr_2( \bar{\psi}) \, k_1= \iota^2_{(1,2)} \, k_3 \, k_2$ on $ U(E |E) \otimes M_{ee} $.

$$\begin{array}{l}
\iota^2_{(1,2)} \, cr_2(\bar{\psi})\, k_1( \rho^2_{(1,2)}(\xi) \otimes m)\\
=  \iota^2_{(1,2)} \,    cr_2(\pi ( \delta \otimes 1)) \, ( \rho^2_{(1,2)}(\xi) \otimes m)\\
= \pi ( \delta \otimes 1) \iota^2_{(1,2)} \,  ( \rho^2_{(1,2)}(\xi) \otimes m)\\
=  \pi ( \delta \otimes 1) ((\xi-i_1^2r_1^2 \xi-i_2^2r_2^2 \xi) \otimes m) \mathrm{\ by\ \ref{ses1}}\\
=  \pi ((\overline{\xi} \otimes \overline{\xi}-\overline{i_1^2r_1^2 \xi} \otimes \overline{i_1^2r_1^2 \xi}-\overline{i_2^2r_2^2 \xi} \otimes \overline{i_2^2r_2^2 \xi} )\otimes m)\\
= \pi (((\overline{i_1^2r_1^2 \xi}+\overline{i_2^2r_2^2 \xi}) \otimes (\overline{i_1^2r_1^2 \xi}+\overline{i_2^2r_2^2 \xi})-\overline{i_1^2r_1^2 \xi} \otimes \overline{i_1^2r_1^2 \xi}-\overline{i_2^2r_2^2 \xi} \otimes \overline{i_2^2r_2^2 \xi} )\otimes m)\\
\mathrm{by\ Lemma\  \ref{lem1}\ (3)\ since\ } T_1U \rm{\ is \ linear}\\
= \pi (\overline{i_1^2r_1^2 \xi} \otimes \overline{i_2^2r_2^2 \xi} \otimes m + \overline{i_2^2r_2^2 \xi} \otimes \overline{i_1^2r_1^2 \xi} \otimes m )\\
= \pi (\overline{i_1^2r_1^2 \xi} \otimes \overline{i_2^2r_2^2 \xi} \otimes m + t( \overline{i_1^2r_1^2 \xi} \otimes \overline{i_2^2r_2^2 \xi} \otimes Tm ))  \mathrm{\mbox{\ cf.\  Propositions \ref{symmodindfct} and \ref{symbifmod}}}\\
= \pi (\overline{i_1^2r_1^2 \xi} \otimes \overline{i_2^2r_2^2 \xi} \otimes m +  \overline{i_1^2r_1^2 \xi} \otimes \overline{i_2^2r_2^2 \xi} \otimes Tm ) \mathrm{\ since\ } \pi(tx)=\pi(x)\\
= \pi (\overline{i_1^2r_1^2 \xi} \otimes \overline{i_2^2r_2^2 \xi} \otimes (m+Tm)) \\
=\pi \, \mathbb{T}_{11}(M_{ee})(i^2_1, i^2_2)(\overline{r_1^2 \xi} \otimes \overline{r_2^2 \xi} \otimes (m+Tm))\\
=\pi \, \mathbb{T}_{11}(M_{ee})(i^2_1, i^2_2)(\overline{1} \otimes \overline{1} \otimes ({r_1^2 \xi} \otimes {r_2^2 \xi} ).(m+Tm))\\
=\pi \, \mathbb{T}_{11}(M_{ee})(i^2_1, i^2_2) \mu_{M_{ee}}^{-1}(({r_1^2 \xi} \otimes {r_2^2 \xi} ).(m+Tm))\\
=\iota^2_{(1,2)} \, k_3 (({r_1^2 \xi} \otimes {r_2^2 \xi} ).(m+Tm))\\
=\iota^2_{(1,2)} \, k_3 \, k_2 ( \rho^2_{(1,2)}(\xi) \otimes m)
\end{array}$$
\end{proof}

\begin{proof}[Proof of Theorem \ref{cross-pt-quad} ]
By Lemma \ref{diago} and Proposition \ref{bi-stable} the source functor is bilinear and we deduce from Proposition \ref{qtp-is-quad} that the target functor is bilinear. So according to Proposition \ref{car-bipoly} it suffices to check that $cr_2(\hat{\phi})_{(E,E)}$ is an isomorphism if and only if $M$ is a quadratic $\C$-module. By Lemma \ref{i} we know that, for $M$ a proto-quadratic $\C$ module, $cr_2(\hat{\phi})_{(E,E)}$ is surjective. So it is sufficient to prove that $cr_2(\hat{\phi})_{(E,E)}$ is injective if and only if the condition (QM2) is satisfied.

As a pushout of abelian groups can be written as a right exact sequence in an obvious way and as the cross-effect functor is exact, we deduce from Lemma \ref{push-out-eq} that the following diagram is also a pushout.
$$\xymatrix{
cr_2((\mathbb{T}_{11}(N) \Delta_{\C})_{\mathfrak{S}_2})(E,E) \oplus cr_2(U)(E,E) \otimes M_{ee} \ar[r]^-{cr_2(\overline{\phi})_{E,E}} \ar[d]_-{cr_2(\overline{\psi})_{E,E}} &cr_2(T_2U)(E,E) \otimes_{\Lambda} M_e \ar[d]^{cr_2(\hat{\psi})_{E,E}}\\
cr_2((\mathbb{T}_{11}(M_{ee}) \Delta_{\C})_{\mathfrak{S}_2})(E,E) \ar[r]_-{cr_2(\hat{\phi})_{E,E}} &cr_2(- \otimes M)(E,E).
}$$

By the previous Lemma we obtain the following commutative diagram (where, for simplicity, we omit the subscript $E,E$):

$$\xymatrix{
N \oplus U(E |E) \otimes M_{ee} \ar[r]^-{k_1}_-{\simeq} \ar[d]^-{k_2}& cr_2((\mathbb{T}_{11}(N) \Delta_{\C})_{\mathfrak{S}_2})(E,E) \oplus cr_2(U)(E,E) \otimes M_{ee} \ar[r]^-{cr_2(\overline{\phi})} \ar[d]_-{cr_2(\overline{\psi})} &cr_2(T_2U)(E,E) \otimes_{\Lambda} M_e \ar[d]^{cr_2(\hat{\psi})}\\
M_{ee} \ar[r]_-{k_3}^-{\simeq}&cr_2((\mathbb{T}_{11}(M_{ee}) \Delta_{\C})_{\mathfrak{S}_2})(E,E) \ar[r]_-{cr_2(\hat{\phi})} &cr_2(- \otimes M)(E,E).
}$$

As $k_1$ and $k_3$ are isomorphisms, we deduce that the exterior diagram is a pushout too.

By a general property of pushouts in $Ab$ we deduce that:
\begin{equation} \label{eq.5.21}
ker(cr_2(\hat{\phi}))=k_3\, ker(cr_2(\hat{\phi}) k_3)=k_3\, k_2\, ker(cr_2({\overline{\phi}}) k_1).
\end{equation}

Recall that $cr_2(\overline{\phi})=(cr_2(\overline{\phi'_1 \otimes 1}), cr_2(t_2) \otimes P)$. So, we have $$cr_2(\overline{\phi}) k_1=(cr_2(\overline{\phi'_1 \otimes 1})(\iota^2_{(1,2)})^{-1} \pi \mathbb{T}_{11}(N)(i^2_1, i^2_2)(\mu_N)^{-1}, cr_2(t_2) \otimes P).$$
In the sequel, we compute $cr_2(\overline{\phi'_1 \otimes 1})(\iota^2_{(1,2)})^{-1} \pi \mathbb{T}_{11}(N)(i^2_1, i^2_2)(\mu_N)^{-1}$. For $x \in cr_2(U)(E,E)$ and $a \in M_e$ we have:
$$\begin{array}{l}
\iota^2_{(1,2)}cr_2(\overline{\phi'_1 \otimes 1})(\iota^2_{(1,2)})^{-1} \pi \mathbb{T}_{11}(N)(i^2_1, i^2_2)(\mu_N)^{-1}(t_{11}(x) \otimes a)\\
=(\overline{\phi'_1 \otimes 1} )\iota^2_{(1,2)}(\iota^2_{(1,2)})^{-1} \pi \mathbb{T}_{11}(N)(i^2_1, i^2_2)(\mu_N)^{-1}(t_{11}(x) \otimes a)\\
=(\overline{\phi'_1 \otimes 1} ) \pi \mathbb{T}_{11}(N)(i^2_1, i^2_2)(\overline{1} \otimes \overline{1} \otimes( t_{11}(x) \otimes a))\\
=(\phi'_1 \otimes 1)(\overline{i^2_1} \otimes \overline{i^2_2} \otimes( t_{11}(x) \otimes a))\\
=\phi'_1(t_1 \otimes t_1 \otimes t_{11}) (i^2_1 \otimes i^2_2 \otimes x) \otimes a\\
=t_2 S^U_2 U(i^2_1 \mid i^2_2)(x) \otimes a \mathrm{\qquad by \  the \ proof\ of \  Lemma\ } \ref{push-out-eq} \\
=t_2\, U(\nabla^2) U(i^2_1 \vee i^2_2) \iota^2_{(1,2)}(x) \otimes a\\
=t_2\, \iota^2_{(1,2)}(x) \otimes a\\
=\iota^2_{(1,2)} cr_2(t_2)(x) \otimes a\\
=\iota^2_{(1,2)} (\overline{cr_2(t_2)} \otimes 1)(t_{11}(x) \otimes a).
\end{array}$$
We deduce that:
$$cr_2(\overline{\phi}) k_1=(\overline{cr_2(t_2)} \otimes 1, cr_2(t_2) \otimes P).$$
By Theorem \ref{cr2t2}, $\overline{cr_2(t_2)} $ is an isomorphism, whence we have the following exact sequence:
$$U(E \mid E) \otimes_{\Lambda} M_{ee}  \xrightarrow{k_4} N \oplus (U(E \mid E) \otimes_{\Lambda} M_{ee}) \xrightarrow{cr_2(\overline{\phi}) k_1} T_2U(E \mid E)  \otimes_{\Lambda} M_{e} $$
where
$$k_4=((\overline{cr_2(t_2)} \otimes 1)^{-1} (cr_2(t_2) \otimes P), -1)^t=((\overline{cr_2(t_2)})^{-1} cr_2(t_2) \otimes P, -1)^t=(t_{11} \otimes P, -1)^t.$$ Since $k_3$ is an isomorphism, we deduce from \ref{eq.5.21} and the last exact sequence that:
$$ker(cr_2(\hat{\phi}) k_3)= k_2\, ker(cr_2({\overline{\phi}}) k_1)= k_2 Im(k_4).$$
For $\xi \in \C(E, E \vee E)$ and $m \in M_{ee}$ we have:
$$k_2\, k_4(\rho^2_{(1,2)}(\xi) \otimes m)=k_2(t_{11}\rho^2_{(1,2)}(\xi) \otimes Pm, - \rho^2_{(1,2)}(\xi) \otimes m)$$
$$=\hat{H}(t_{11}\rho^2_{(1,2)}(\xi) \otimes Pm)-(r^2_1 \xi \otimes r^2_2 \xi).(m+Tm).$$

It follows that $ker(cr_2(\hat{\phi}))=0$ if and only if the condition (QM2) holds.
\end{proof}

We now are ready to compute the cross-effect of $- \otimes M$.

\begin{cor} \label{gamma}
Let $M \in QMod^E_C$,  there is a natural isomorphism of bifunctors:
$$\gamma: (T_1U(X) \otimes T_1U(Y)) \otimes_{\Lambda \otimes \Lambda}M_{ee} \xrightarrow{\simeq} cr_2(- \otimes M)(X,Y)$$
 for all $X,Y \in \langle E \rangle_{\C}$ such that:
$$(\gamma)_{X,Y}( t_1(f) \otimes t_1(g) \otimes m)=(\iota^2_{(1,2)})^{-1} \hat{\phi} \pi (t_1(i_1 f) \otimes t_1(i_2 g) \otimes m)$$ where  $f  \in \C(E,X)$, $g \in \C(E,Y)$, $m \in M_{ee}$, $i_1: X \to X \vee Y$ and $i_2: Y \to X \vee Y$.
\end{cor}

When $E$ has suitable properties  this result extends to all objects $X,Y$ which are colimits of copies of $E$, see Theorem \ref{cr-of-otM-Cgen} below.

\begin{proof}
Recall that $\mathbb{T}_{11}(M_{ee})(X,Y)=  (T_1U(X) \otimes T_1U(Y)) \otimes_{\Lambda \otimes \Lambda}M_{ee} $ by Definition \ref{T11}.
We have the following natural factorization of $(\gamma)_{X,Y}$: 
$$\mathbb{T}_{11}(M_{ee})(X,Y) \xrightarrow{(\iota^2_{(1,2)})^{-1} \pi \mathbb{T}_{11}(M_{ee})(i_1, i_2)} cr_2((\mathbb{T}_{11}(M_{ee}) \Delta_{\C})_{\mathfrak{S}_2})(X,Y) \xrightarrow{cr_2(\hat{\phi})} cr_2(- \otimes M)(X,Y). $$
Since $\mathbb{T}_{11}(M_{ee})$ is bilinear we deduce from Proposition \ref{bilin2} that 
$(\iota^2_{(1,2)})^{-1} \pi \mathbb{T}_{11}(M_{ee})(i_1, i_2)$ is a natural equivalence, and we deduce from Theorem  \ref{cross-pt-quad} that $cr_2(\hat{\phi})$ is an isomorphism. Hence $\gamma$ is a natural equivalence, too.
\end{proof}

\subsection{Preservation of filtered colimits and suitable coequalizers}

In this section we show that the quadratic tensor product has strong enough preservation properties to be determined, under suitable assumptions, by its restriction to ``free objects of finite rank'' which was determined in the preceding sections. Under even stronger, but still very general assumptions which in particular cover all pointed algebraic varieties where a group law is part of the structure, we extend the computation of the cross-effect of $-\otimes M$ in the preceding section to all objects in $\C$.

More precisely, we suppose throughout this section that $\C$ has all (even infinite) sums (or at least all sums of copies of $E$), and that $E$ is a small regular projective generator of $\C$. Recall that small means that the functor $\C(E,-):\C \to Set_*$ preserves filtered colimits, regular projective means that $E$ is projective with respect to the class of all regular epimorphisms (i.e., quotients of coequalizers in $\C$), and generator means that all objects $X$ in $\C$ are colimits of copies of $E$, or equivalently, admit a regular epimorphism 
$\xymatrix{\bigvee_{i\in I} E \ar@{->>}[r] & X
}$. Note that these assumptions are satisfied for all algebraic varieties when we take $E$ to be the free object of rank $1$.

\begin{defi}\label{Epresentdefis} A presentation of an object $X$ in $\C$ is a coequalizer diagram
 $${\xymatrix{ {\mathbb P} \co X_1 \ar@<0.4ex>[r]^-{d_0} \ar@<-0.4ex>[r]_-{d_1}& X_0\ar@{->}[r]^q & X }}$$
 in $\C$. We say that $\mathbb P$ is $E$-free if $X_0$ and $X_1$ are $E$-free, i.e.\ sums of copies of $E$, and that $\mathbb P$ is $E$-saturated if $\forall f_0,f_1 \in \C(E,X_0), (qf_0=qf_1 \Rightarrow \exists f_{01}\in \C(E,X_1)$, $f_k=d_kf_{01}$, $k=0,1)$. Moreover, $\mathbb P$ is said to be reflexive if the pair $(d_0,d_1)$ is reflexive, i.e.\ admits a common section $s_0\in \C(X_0,X_1)$ of $d_0$ and $d_1$. 
 
 \end{defi}

\begin{lm}\label{prestypes} We have the following properties for $\C$ and $E$ as above.
\begin{enumerate}

\item Any object of $\C$ admits an $E$-saturated $E$-free presentation.

\item Any $E$-saturated $E$-free presentation is reflexive.

\item If $\C$ is Mal'cev and Barr exact, any reflexive presentation is $E$-saturated.
\end{enumerate}

\end{lm}

We point out that all pointed homological, in particular semi-abelian categories are Mal'cev and Barr exact, see \cite{BB}; so these hypothesis cover all pointed algebraic varieties where a group law is part of the structure, like groups, algebras over any reduced operad etc., as well as the  categories of compact Hausdorff-spaces, of crossed modules of groups, of $C^*$-algebras etc., see for example \cite{BB}.\medskip

\noindent\textit{Proof of Lemma \ref{prestypes}}. (1): Let $X$ be an object in $C$. As $E$ generates $\C$ there exists an $E$-free presentation
  ${\xymatrix{ {\mathbb P} \co E_1 \ar@<0.4ex>[r]^-{d_0} \ar@<-0.4ex>[r]_-{d_1}& E_0\ar@{->}[r]^q & X }}$. 
  Now let $P[q]=\{ (f_0,f_1)\in \C(E,E_0)\,|\,qf_0=qf_1\}$, $E_1' =\bigvee_{(f_0,f_1)\in P[q]}E$ and $e_{(f_0,f_1)}\co E \to E_1'$,  $(f_0,f_1)\in P[q]$, be the defining injections of the coproduct. For $k=0,1$, let $d_k'\co E_1' \to E_0$ such that $d_ke_{(f_0,f_1)} =f_k$. Then  
  ${\xymatrix{ {\mathbb P} \co X_1 \vee E_1' \ar@<0.4ex>[r]^-{(d_0,d_0')} \ar@<-0.4ex>[r]_-{(d_1,d_1')}& X_0\ar@{->}[r]^q & X }}$
is an $E$-saturated $E$-free presentation of $X$.
\medskip

\noindent(2): Let ${\xymatrix{ {\mathbb P} \co E_1 \ar@<0.4ex>[r]^-{d_0} \ar@<-0.4ex>[r]_-{d_1}& E_0\ar@{->}[r]^q & X }}$ be an $E$-saturated $E$-free presentation, with $E_0=\bigvee_{i\in I}E$ and defining injections $e_i\co E\to E_0$. Taking $f_0=f_1=e_i$ for $i\in I$ provides a map $f_{01}^i\in \C(E,E_1)$ such that $d_0f_{01}^i =d_1f_{01}^i = e_i$. Pasting these maps $f_{01}^i$ together furnishes the desired common section $s_0\co E_0\to E_1$ of $d_0$ and $d_1$.\medskip

\noindent(3): Consider the following diagram
\[\xymatrix @!0 @R=18mm @C=40mm{
X_1 \ar@<0.4ex>[r]^-{d_0} \ar@<-0.4ex>[r]_-{d_1} \ar@{.>}[rd]_{(d_0,d_1)}  & X_0\ar@{->}[r]^q  & X  \\
E \ar@{.>}[u]^{f_{01}} \ar[r]_{(f_0,f_1)}  & R[q] \ar[u]^{p_0} \ar[r]^{p_1} & X_0 \ar[u]^q
}\]
where the top line is a reflexive presentation and the right hand square is a pullback (i.e., $(p_0,p_1)$ is a kernel pair of $q$). Thus the map $(d_0,d_1)$ as indicated exists such that $p_k
(d_0,d_1)=d_k$, $k=0,1$. Moreover, we suppose that $f_0,f_1\in \C(E,X_0)$ are maps such that $qf_0=qf_1$; thus the map $(f_0,f_1)$ as indicated exists such that $p_k
(f_0,f_1)=f_k$, $k=0,1$. By hypothesis on $\C$ the map
$(d_0,d_1)$ is a regular epimorphism (cf.\ \cite{BB} or \cite{Carboni}), so by regular projectivity of $E$ there exists a lifting $f_{01}$ as indicated such that $(d_0,d_1)f_{01}=(f_0,f_1)$. It satisfies $d_kf_{01}=p_k(d_0,d_1)f_{01}=p_k(f_0,f_1)=f_k$, as desired.\hfill $\Box$\medskip

Now consider a functor $F\co \C \to \mathcal D$ to some category $\mathcal D$. Recall that one  says that $F$ preserves a certain type 
of coequalizers (or presentations) if it transforms coequalizers of this type in $\C$  into coequalizers in $\D$.

\begin{prop}\label{presofcoequ} For $F$ as above consider the following properties.

\begin{enumerate}

\item $F$ preserves $E$-saturated coequalizers.

\item $F$ preserves $E$-saturated $E$-free coequalizers.

\item $F$ preserves reflexive coequalizers.

\end{enumerate}

Then $(1) \Leftrightarrow (2) \Leftarrow (3)$, and
if $\C$ is Mal'cev and Barr exact then all three properties are equivalent.

\end{prop}

\begin{proof} It is obvious that (1) implies (2), and (3) implies (2) by Lemma \ref{prestypes} (2). Moreover, if $\C$ is Mal'cev and Barr exact (1) implies (3) by Lemma \ref{prestypes} (3), so it remains to prove that (2) implies (1). Let $\mathbb P$ be an $E$-saturated coequalizer as in Definition \ref{Epresentdefis}. Then we may choose a regular epimorphism $\xymatrix{q_1\co E_1 \ar@{->>}[r]& X_1}$ with an $E$-free object $E_1$, and an $E$-saturated $E$-free presentation of $X_0$ as in the diagram
\begin{equation} \xymatrix{ 
&E_1' \ar@<0.4ex>[d]^-{ d_1' } \ar@<-0.4ex>[d]_-{d_0'}   & \\
 E_1   \ar@<0.4ex>@{.>}[r]^-{  \tilde{d}_0 } \ar@<-0.4ex>@{.>}[r]_-{ \tilde{d}_1 } \ar[d]_{q_1} & E_0 \ar[ld]^s \ar[d]^{q_0} &  \\
X_1 \ar@<0.4ex>[r]^-{d_0} \ar@<-0.4ex>[r]_-{d_1}& X_0\ar@{->}[r]^q & X
  }\end{equation}
Here the diagonal map $s$ is a map such that $d_0s=d_1s=q_0$; it is constructed on a given summand $E$ of $E_0$ by taking $f_0$ and $f_1$ to be the restriction of $q_0$ to this summand, and applying the saturation property of $\mathbb P$, as in the proof of Lemma \ref{prestypes} (2). Moreover, $\tilde{d}_0$ and $\tilde{d}_1$ are liftings such that $q_0\tilde{d}_k=d_kq_1 $, $k=0,1$; they exist since $E_1$ is regular projective. Now observe that   ${\xymatrix{  E_1 \vee E_1' \ar@<0.4ex>[r]^-{(\tilde{d}_0,d_0')} \ar@<-0.4ex>[r]_-{(\tilde{d}_1,d_1')}& E_0\ar@{->}[r]^{qq_0} & X }}$
is an $E$-free presentation of $X$; by applying the procedure in the proof of Lemma \ref{prestypes}(1) we find an $E$-free object $E_1''$ and maps $\xymatrix{E_1''  \ar@<0.4ex>[r]^-{\epsilon_0} \ar@<-0.4ex>[r]_-{\epsilon_1}& E_0}$ in $\C$ such that
${\xymatrix{   E_1 \vee E_1' \vee E_1'' \ar@<0.4ex>[r]^-{(\tilde{d}_0,d_0',\epsilon_0)} \ar@<-0.4ex>[r]_-{(\tilde{d}_1,d_1',\epsilon_1)}& E_0\ar@{->}[r]^{qq_0} & X }}$
is an $E$-saturated $E$-free presentation of $X$.
Next we construct a map $s'\in \C(E_1'',X_1)$ such that $d_ks'=q_0\epsilon_k$, $k=0,1$: on a given summand $E$ of $E_1''$ it is obtained by taking $f_k$ to be the restriction of $q_0\epsilon_k$ to that summand, and applying the saturation property of $\mathbb P$. We thus get the following diagram where $e_1'$ is the injection.
\begin{equation}\label{E-resol-dia}\xymatrix{ 
&E_1' \ar@<0.4ex>[d]^-{ d_1' } \ar@<-0.4ex>[d]_-{d_0'} \ar[dl]_-{e_1'} & \\
 E_1 \vee E_1' \vee E_1'' \ar@<0.4ex>[r]^-{ (\tilde{d}_0,d_0',\epsilon_0) } \ar@<-0.4ex>[r]_-{(\tilde{d}_1,d_1',\epsilon_1)} \ar[d]_{(q_1,sd_0',s')} & E_0\ar@{->}[r]^{qq_0} \ar[d]^{q_0} & X  \ar@{=}[d] \\
X_1 \ar@<0.4ex>[r]^-{d_0} \ar@<-0.4ex>[r]_-{d_1}& X_0\ar@{->}[r]^q & X
  }\end{equation}
It is commutative, for the left hand square in the sense that the two squares  formed by taking the  horizontal arrows with the same index   $k=0,1$, commute.

 Now we apply a functor $F\co \C\to \D$  preserving $E$-saturated $E$-free coequalizers to this diagram, and deduce that  
 ${\xymatrix{F(X_1) \ar@<0.4ex>[r]^-{F(d_0)} \ar@<-0.4ex>[r]_-{F(d_1)}& F(X_0)\ar@{->}[r]^{F(q)} & F(X) }}$
 is a coequalizer. Let $A\in\D$ and $\alpha\in \D(F(X_0),A)$ such that $\alpha F(d_0) =\alpha F(d_1)$. Then 
 \begin{eqnarray*}
 \alpha F(q_0)
 F(\tilde{d}_0,d_0',\epsilon_0) &=& \alpha F(d_0)
F(q_1,sd_0',s')\\
&=& \alpha F(d_1)F(q_1,sd_0',s')\\
&=& \alpha F(q_0)F(\tilde{d}_1,d_1',\epsilon_1)
\end{eqnarray*}
As $F$ transforms the top line and the column in (\ref{E-resol-dia}) into a coequalizers it follows that there exists a unique map $\bar{\alpha} \in \D(F(X) ,A)$ such that $\bar{\alpha} F(qq_0)  = \alpha F(q_0)$, and that $F(q_0)$ is a regular epimorphism. Thus $\bar{\alpha}F(q) = \alpha$, which implies the assertion.
 
\end{proof}

Recall that if $\mathbb P$ as above is a reflexive coequalizer in $Ab$ then $X\cong {\rm Coker}(d_0-d_1\co X_1 \to X_0) \cong {\rm Coker}(d_1'\co {\rm Ker}(d_0) \to X_0)$ where $d_1'$ denotes the restriction of $d_1$. This implies that an additive functor between abelian categories preserves reflexive coequalizers iff it is right exact, so preservation of reflexive coequalizers generally is considered as the appropriate generalization of right exactness to non-linear functors. In fact, a functor preserving  reflexive coequalizers and  filtered colimits generically is determined
by its restriction to ``free objects of finite rank''.
In particular this reduction is used in the study of quadratic functors  \cite{Baues-Pira} and of functors between categories of algebras over operads \cite{Benoit}. In the same spirit, we now state a preservation theorem for our quadratic tensor product  which shows, however, that in the most general situation reflexive coequalizers must be replaced by the more specific $E$-saturated ones; on the other hand, Proposition \ref{presofcoequ} explains why this difference
did not yet appear in practice: most ``real life'' pointed categories are at least homological and hence Mal'cev and Barr exact.

\begin{thm}\label{otMpreserves} Suppose that $\C$ is a pointed category with sums and that $E$ is a small regular projective generator of $\C$. Moreover, let $M$ be a quadratic $\C$-module relative to $E$.
Then the functors $T_nU_E$ and $-\otimes M\co \C \to Ab$, $n=1,2$,  preserve filtered colimits and $E$-saturated coequalizers, and reflexive coequalizers if $\C$ is Mal'cev and Barr exact.

\end{thm}

\begin{proof} Let $Set_*, Set_*\mbox{-\,}\C(E,E)$ denote the categories of pointed sets and of pointed sets equipped with a right action of the monoid $\C(E,E)$, resp., the latter satisfying that $x_0a=x_0=x0$ where $x_0$ is the basepoint of $S\in Set_*\mbox{-\,}\C(E,E)$, $x\in S$ and $a\in \C(E,E)$. Consider the functors
\[\xymatrix{
Ab &  Set_* \ar[l]_-{\overline{\mathbb{Z}}[-]} & \C\ar[l]_-{\rho_E}   \ar[r]^-{\rho_E'} & Set_*\mbox{-\,}\C(E,E) \ar[r]^-{\overline{\mathbb{Z}}[-]} & Mod\mbox{\,-\,}\Lambda}\]
where $\rho_E,\rho_E'$ are both given by $\C(E,-)$, the right action of $\C(E,E)$ on $\C(E,X)$ being given by precomposition. Both 
$\rho_E$ and $\rho_E'$ preserve filtered colimits 
since $E$ is small, and preserve
$E$-saturated coequalizers since $E$ is  regular projective, and by definition of $E$-saturation. Moreover, both functors $\overline{\mathbb{Z}}[-]$ are left adjoint to the obvious forgetful functors, hence preserve colimits. Thus the composite functors $\xymatrix{
Ab & \C\ar[l]_-{U_E} \ar[r]^-{U_E'} & Mod\mbox{\,-\,}\Lambda}$,  both preserve filtered colimits and $E$-saturated coequalizers, and so do  the functors 
$\xymatrix{U_E(X^{\vee n}) & X \ar@{|->}[l] \ar@{|->}[r] & U_E'(X^{\vee n})}$ since colimits commute among each other. For the same reason, $T_1U_E = {\rm Coker}(S^F_2 \circ \rho^2_{(1,2)}\co U_E(X^{\vee 2}) \to U_E(X))$, $T_2U_E = {\rm Coker}(\nabla^3_*-(\nabla^2 r^3_{12})_*- (\nabla^2 r^3_{13})_*-(\nabla^2 r^3_{23})_*+r_{1*}^3+r_{2*}^3+r_{3*}^3\co U_E(X^{\vee 3}) \to U_E(X))$, and  $T_2U_E'$ preserve filtered colimits and $E$-saturated coequalizers,
see Propositions \ref{explicit-lin} and \ref{quadratization}. As the tensor product preserves filtered colimits and coequalizers so does the functor  $Mod\mbox{\,-\,}\overline{\Lambda} \to {\mathbb Z}[\mathfrak{S}_2]\mbox{\,-\,}Mod$ sending $A$ to $(A\otimes A)\otimes_{\overline{\Lambda}\otimes \overline{\Lambda}} N$, $N=M_{ee}$ or $N=T_{11}(cr_2U_E)(E,E)$, and also the functor ${\mathbb Z}\otimes_{{\mathbb Z}[\mathfrak{S}_2]}- \co {\mathbb Z}[\mathfrak{S}_2]\mbox{\,-\,}Mod \to Ab$. Thus $-\otimes M$ is a functorial pushout (hence colimit) of functors which preserve filtered colimits and $E$-saturated coequalizers, hence so does $-\otimes M$ as colimits commute among each other. The statement about reflexive coequalizers then follows from Proposition \ref{presofcoequ}.
\end{proof}

The preservation properties just proved allow for the desired reduction to the case of free objects of finite rank:

\begin{prop}\label{redtofinsums} Let $\D$ be any category, and let $\varphi\co F \to G$ be a morphism between functors $F,G\co \C \to \D$ which both preserve filtered colimits and $E$-saturated $E$-free coequalizers.
Then $\varphi$ is an isomorphism iff $\varphi_{E^{\vee n}}$ is an isomorphism for all $n\ge 1$. Similarly, let $\psi\co B \to C$ be a morphism between bifunctors $B,C\co \C \times \C \to \D$ which both preserve filtered colimits and $E$-saturated $E$-free coequalizers in each variable (i.e.\ the functors $B(X,-)$ and $B(-,X)$ do for all $X\in \C$, the same for $C$).
Then $\psi$ is an isomorphism iff $\psi_{(E^{\vee n},E^{\vee m})}$ is an isomorphism for all $n,m\ge 1$.

\end{prop}

\begin{proof} This is standard: let $I$ be a set and $P$ be the set of its finite subsets, and write $E_I=\bigvee _{i\in I} E$. Then $E_I=  {\rm colim}_{J\in P} E_J$ is a filtered colimit whence $\varphi_{E_I}$ is a colimit of isomorphisms, hence an isomorphism. It now follows that $\varphi_X$ is an isomorphism for all $X\in \C$ by using an $E$-saturated $E$-free presentation of $X$, which exists by Lemma \ref{prestypes}. As to bifunctors, the result for functors just proved successively implies that first $\psi_{(X,E^{\vee m})}$, then 
$\psi_{(X,Y)}$ is an isomorphism for all $X,Y\in \C$, $m\ge 1$.
\end{proof}


Now we wish to establish similar reduction properties for cross-effects. For this we must check that  suitable preservation properties of a functor are inherited by its cross-effects:

\begin{prop}\label{crpreserve} If a functor $F\co \C \to Ab$ preserves filtered colimits and reflexive coequalizers then so do the functors $cr_n(X_1,\ldots,X_n,-)\co \C\to Ab$ for $n\ge 1$ and $X_1,\ldots,X_n$ $\in \C$.
\end{prop}

\begin{proof} By induction it suffices to prove the case $n=1$. Let $X\in \C$.\medskip

\noindent \textit{Filtered colimits:} Let $\mathcal I$ be a filtered category and $D\co \mathcal I \to \C$ be a diagram  admitting a colimit in $\C$. Then $X\vee {\rm colim}\, D={\rm colim}\,  D_1$ where $D_1\co \mathcal I\to \C$ is given by $D_1(i)=X\vee D(i)$ for $i\in \mathcal I$ and
$D_1(f)=1\vee D(f)$ for a morphism $f$ in $\mathcal I$. Similarly we define diagrams $D_2,D_3\co \mathcal I \to\C$ such that $D_2(i)=F(X) \oplus FD(i)$ and $D_3(i)=F(X|D(i))$. Writing $r(i) =(r_{1*}^2, r^2_{2*})\co F(X\vee D(i)) \to F(X)\times FD(i)$ we have the following commutative diagram.
\[\xymatrix{
F(X\vee {\rm colim}\,D) \ar[r]^-{(r_{1*}^2, r^2_{2*})}   & F(X) \times F({\rm colim}\,D)  \\
F({\rm colim}\,D_1) \ar[u]^{\cong} & F(X) \oplus {\rm colim}\,(F\circ D) \ar[u]^{\cong}\\
{\rm colim}\,(F\circ D_1) \ar[u]^{\cong} \ar[r]^-{{\rm colim}\,r(i)} & {\rm colim}\,D_2\ar[u]^{\cong}
}\]
Hence 
\begin{eqnarray*}
F(X|{\rm colim}\,D) &\cong& {\rm Ker}({\rm colim}_{i\in \mathcal I}\,r(i))\\
&\cong& {\rm colim}_{i\in \mathcal I} \,{\rm Ker}(r(i))\quad\mbox{since $\mathcal I$ is filtered}\\
&\cong& {\rm colim}\, D_3,
\end{eqnarray*}
as asserted.\medskip

\noindent\textit{Reflexive coequalizers:} Let 
${\xymatrix{  Y_1 \ar@<0.4ex>[r]^-{d_0} \ar@<-0.4ex>[r]_-{d_1}& Y_0\ar@{->}[r]^q & Y }}$
be a reflexive coequalizer in $\C$ with commun splitting $s_0$ of $d_0$ and $d_1$. Then 
${\xymatrix{ X\vee  Y_1 \ar@<0.4ex>[r]^-{1 \vee d_0} \ar@<-0.4ex>[r]_-{1 \vee d_1}& X\vee Y_0\ar@{->}[r]^{ 1 \vee q} & X\vee Y }}$
also is a reflexive coequalizer in $\C$ with commun splitting $1 \vee s_0$ of $1 \vee d_0$ and $1 \vee d_1$. Consider the following commutative diagram in $Ab$.
\[\xymatrix @!0 @R=18mm @C=40mm{
F(X|Y_1) \ar[r]^{\alpha} \ar@{^{(}->}[d]^{\iota^2_{(1,2)}} & 
F(X|Y_0) \ar[r]^{F(1|q)} \ar@{^{(}->}[d]^{\iota^2_{(1,2)}} &
F(X|Y)\ar@{^{(}->}[d]^{\iota^2_{(1,2)}} & \\
F(X\vee Y_1) \ar[r]^{\beta} \ar@{->>}[d]^{(r^2_{1*}, r^2_{2*})} & 
F(X\vee Y_0) \ar[r]^{F(1\vee q)} \ar@{->>}[d]^{(r^2_{1*}, r^2_{2*})} &
F(X\vee Y)\ar@{->>}[d]^{(r^2_{1*}, r^2_{2*})} \ar[r] & 0\\
F(X)\times F(Y_1) \ar[r]^{\gamma} & 
F(X)\times  F(Y_0) \ar[r]^{1\times  F(q)}  &
F(X)\times  F(Y) \ar[r] & 0
}\]
\rule{0mm}{7mm} where $\alpha = F(1|d_0) - F(1|d_1)$, 
$\beta = F(1\vee d_0) - F(1\vee d_1)$,
$\gamma = (1\times F(d_0)) - (1\times F(d_1)) = 0\times (F(d_0) -F(d_1))$. The columns are exact by (\ref{ses2}), and the second and third row are exact since $F$ preserves reflexive coequalizers. Now replacing $F(X)\times F(Y_1)$ by its image under $\gamma$ the left half of the above diagram becomes
\[\xymatrix @!0 @R=18mm @C=52mm{
F(X|Y_1) \oplus F(X) \oplus {\rm Ker}(d_{0*} -d_{1*}) \ar[r]^-{\tilde{\alpha}} \ar@{^{(}->}[d]^{(\iota^2_{(1,2)},F(i_1),F(i_2)j)} & 
F(X|Y_0)  \ar@{^{(}->}[d]^{(\iota^2_{(1,2)}}  \\
F(X\vee Y_1) \ar[r]^{\beta} \ar@{->>}[d]^{\gamma(r^2_{1*}, r^2_{2*})} & 
F(X\vee Y_0)   \ar@{->>}[d]^{(r^2_{1*}, r^2_{2*})} \\
{\rm Im}(\gamma) \ar@{^{(}->}[r]^{} & 
F(X)\times  F(Y_0) 
}\]
\rule{0mm}{7mm}with $j\co {\rm Ker}(d_{0*} -d_{1*}) \hookrightarrow F(Y_1)$. 
But $\beta F(i_1)=F(i_1)-F(i_1)=0$, and $\beta 
F(i_2)j = (F(i_2d_0) - F(i_2d_1))j = F(i_2)(
F( d_0) - F( d_1))j = 0$, so ${\rm Im}(\tilde{\alpha}) = 
{\rm Im}(\alpha)$. Thus the snake lemma provides an exact sequence $F(X|Y_1) \xrightarrow{\alpha} F(X|Y_0) \xrightarrow{F(1|q)}
F(X|Y) \to 0$, as asserted.
\end{proof}

We now are ready to extend the computation of the cross-effect of the quadratic tensor product from $E$-free objects of finite rank to \textit{all} objects in $\C$, as follows.

\begin{thm}\label{cr-of-otM-Cgen} Suppose that $\C$ is a pointed Mal'cev and Barr exact category with sums, and that $E$ is a small regular projective generator of $\C$. Moreover, let $M$ be a quadratic $\C$-module relative to $E$. Then the natural isomorphism
$$\xymatrix{\gamma: (T_1U(X) \otimes T_1U(Y)) \otimes_{\Lambda \otimes \Lambda}M_{ee}  \hspace{2mm} \ar[r]^-{\simeq} & \hspace{2mm}cr_2(- \otimes M)(X,Y)}$$
in Corollary \ref{gamma} is valid for all objects $X,Y$ in $\C$.
\end{thm}

\begin{proof} The functor $-\otimes M$ preserves $E$-saturated   coequalizers by Theorem
\ref{otMpreserves}, hence reflexive coequalizers by Proposition \ref{presofcoequ} and by hypothesis on $\C$. Thus by Proposition \ref{crpreserve} the bifunctor $cr_2(-\otimes M)$ preserves reflexive coequalizers in both variables, in particular $E$-saturated $E$-free  coequalizers by Lemma \ref{prestypes} (2), and so does the bifunctor $(T_1U(-) \otimes T_1U(-)) \otimes_{\Lambda \otimes \Lambda}M_{ee}$, see the proof of Theorem \ref{otMpreserves}. So by Proposition \ref{redtofinsums} the natural map $\gamma$ is an isomorphism for all $X,Y\in \C$ since it is so when $X=E^{\vee n}$, $Y=E^{\vee m}$, $n,m\ge 1$, by Proposition \ref{gamma}.
\end{proof}

\section{Equivalence between quadratic functors and quadratic $\C$-modules} \label{section-thm}
The aim of this section is to prove the following theorem which is the main result of this paper.

Let $\mbox{\textit{QUAD}}(\C,Ab)$ denote the category of reduced quadratic functors from $\C$ to $Ab$ which preserve filtered colimits and $E$-saturated  coequalizers; recall from Proposition \ref{presofcoequ} that here $E$-saturated  coequalizers can be replaced by $E$-saturated $E$-free coequalizers, and by reflexive coequalizers if $\C$ is Mal'cev and Barr exact.

\begin{thm} \label{thm} Let $\C$ be a pointed category with finite sums.\medskip

\noindent(A) The functors 
$$\xymatrix{
Quad(\C,Ab) \ar@<-1ex>[r]_-{\mathbb{S}_2}  & QMod_{\C}^E \ar@<-1ex>[l]_-{\mathbb{T}_2}  }$$
where  $\mathbb{S}_2$ is defined in section 5.3 and $\mathbb{T}_2(M)=- \otimes M$, form a pair of adjoint functors extending $\mathbb{T}_1$ and $\mathbb{S}_1$ given in Proposition \ref{S1T1} (see also Proposition \ref{4.15}).\medskip

\noindent(B)  If $\C=\langle E \rangle_{\C}$, the functors $\mathbb{S}_2$ and $\mathbb{T}_2$ form a pair of adjoint equivalences.\medskip

\noindent(C) If $\C$ has sums and if $E$ is a small regular projective generator of $\C$, then the functors
$$\xymatrix{
\mbox{\textit{QUAD}}(\C,Ab) \ar@<-1ex>[r]_-{\mathbb{S}_2'}  & QMod_{\C}^E \ar@<-1ex>[l]_-{\mathbb{T}_2'}  }$$
form a pair of adjoint equivalences where $\mathbb{T}_2'$ is given by $\mathbb{T}_2$ which actually takes values in $\mbox{\textit{QUAD}}(\C,Ab)$ by Theorem \ref{otMpreserves}, and where ${\mathbb{S}_2'}$ is the restriction of ${\mathbb{S}_2}$.

\end{thm}

To prove this theorem we start by constructing the co-unit and prove that it is an isomorphism if $\C=\langle E \rangle_{\C}$, thanks to the computations of $E \otimes M$ given in Proposition \ref{prop} and of the cross-effect of $- \otimes M$ given in Corollary \ref{gamma}.
Next we construct  the unit of this adjunction and prove that it is an isomorphism. Assertion (C) then follows from Proposition \ref{redtofinsums}.

\subsection{The co-unit of the adjunction}
The co-unit of the adjunction of Theorem \ref{thm} is a natural map of quadratic functors $\epsilon: - \otimes \mathbb S_2(F) \to F$. To define this map we need a series of lemmas.

\begin{lm} \label{lm4}
For $F \in Quad(\C, Ab)$, there exists a natural transformation $\overline{u'_F}: (T_2U)(-) \otimes_{\Lambda} F(E) \to F(-)$ such that we have the following natural commutative diagram
$$\xymatrix{
U(X) \otimes_{\Lambda} F(E)  \ar[rr]^{(u'_F)_X} \ar@{->>}[dr]_{(t_2 \otimes 1)_X}&&F(X)\\
&(T_2U)(X) \otimes_{\Lambda} F(E) \ar@{..>}[ru]_{(\overline{u'_F})_X}
}$$
for $X \in \C$ and $u'_F$ the co-unit of the adjunction between $\mathbb{S}$ and $\mathbb{T}$ given in Proposition \ref{ST}.
\end{lm}

\begin{proof}
This is immediate from the universal property of $t_2$.
\end{proof}

\begin{lm} \label{lm5}
For $F \in Quad(\C, Ab)$, there exists a morphism $(\overline{S^F_2})_X: F(X|X)_{\mathfrak{S}_2} \to F(X)$ making the following diagram commutative
$$\xymatrix{
F(X |X) \ar[rr]^{(S^F_2)_X} \ar@{->>}[dr]_{\pi}&&F(X)\\
& F(X|X)_{\mathfrak{S}_2}  \ar@{..>}[ru]_{(\overline{S^F_2})_X}
}$$
for $X \in \C$.
\end{lm}

\begin{proof}
To prove that $\overline{S^F_2}$ exists we have to prove that $S^F_2T^F=S^F_2$ where $T^F$ is the involution of $F(-|-)$. Recall that $T^F$ is defined by the following commutative diagram
$$\xymatrix{
F(X |X) \ar@{^{(}->}[r]^{\iota^2_{(1,2)}} \ar[d]_{T^F}&F(X \vee X) \ar[d]^{F(\tau)}\\
F(X |X) \ar@{^{(}->}[r]^{\iota^2_{(1,2)}} & F(X \vee X) 
}$$
where $\tau$ is the canonical switch. Thus:
\begin{eqnarray*}
(S^F_2)_X T^F&=&F(\nabla^2) \iota^2_{(1,2)} T^F=F(\nabla^2 \tau)\iota^2_{(1,2)}=F(\nabla^2)\iota^2_{(1,2)}=(S^F_2)_X.
\end{eqnarray*}
\end{proof}

\begin{lm} \label{lm6}
For $F \in Quad(\C, Ab)$, there exists a natural transformation of functors:
$$\overline{u'_{cr(F)}}: (((T_1U)(-) \otimes (T_1U)(-)) \otimes_{\Lambda \otimes \Lambda} F(E |E)) \Delta_{\C})_{\mathfrak{S}_2} \to (cr_2F(-,-)  \Delta_{\C})_{\mathfrak{S}_2} $$ where $\Delta_{\C}: \C \times \C \to \C$ is the diagonal functor, making the following diagram naturally commutative.
$$\xymatrix{
 ((T_1U)(X) \otimes (T_1U)(X)) \otimes_{\Lambda \otimes \Lambda} F(E|E) \ar[rr]^-{(u'_{cr_2(F)})_{X,X}} \ar@{->>}[d]_{\pi}&&F(X |X)\ar@{->>}[d]_{\pi}\\
(((T_1U)(X) \otimes (T_1U)(X)) \otimes_{\Lambda \otimes \Lambda} F(E|E))_{\mathfrak{S}_2}  \ar[rr]^-{(\overline{u'_{cr_2(F)}})_{X,X}}&& (F(X|X))_{\mathfrak{S}_2} 
}$$
Here $X \in \C$, and
$\xymatrix{
 ((T_1U)(X) \otimes (T_1U)(X)) \otimes_{\Lambda \otimes \Lambda} F(E|E) \ar[r]^-{(u'_{cr_2(F)})_{X,X}}& F(X|X) }$ is the co-unit of the adjunction in Proposition \ref{S11T11}. 
\end{lm}

\begin{proof}
For $x, y \in U(X)$ and $m \in F(E \mid E)$ we have:
\begin{eqnarray*}
T^Fu'_{cr_2(F)}(t_1(f) \otimes t_1(g) \otimes m)&=&T^F(cr_2F(f,g))(m)\\
&=& (\iota^2_{(1,2)})^{-1} F(\tau) \iota^2_{(1,2)} cr_2F(f,g)(m)\\
&=& (\iota^2_{(1,2)})^{-1} F(\tau) F(f \vee g) \iota^2_{(1,2)} (m)\\
&=& (\iota^2_{(1,2)})^{-1} F(g \vee f)  F(\tau) \iota^2_{(1,2)} (m)\\
&=& cr_2F(g,f) (\iota^2_{(1,2)})^{-1}   F(\tau) \iota^2_{(1,2)} (m)\\
&=& cr_2F(g,f) (T^F (m))\\
&=& u'_{cr_2(F)}(t_1(g) \otimes t_1(f) \otimes T^F(m))\\
&=& u'_{cr_2(F)}(t(t_1(f) \otimes t_1(g) \otimes m))\\
\end{eqnarray*}
Thus $u'_{cr_2(F)}$ is $\mathfrak{S}_2$-equivariant and hence passes to coinvariants.
\end{proof}

We deduce the following proposition.

\begin{prop} \label{epsilon}
For $F \in Quad(\C, Ab)$, there exists a natural map $\epsilon: - \otimes \mathbb S_2(F) \to F$ given by: 
$$ \epsilon_X=((\overline{S^F_2})_X \ (\overline{u'_{cr_2(F)}})_X, (\overline{u'_F})_X): X \otimes \mathbb S_2(F) \to F(X)$$
for $X \in \C$ such that 
$\epsilon_X \hat{\phi}= (\overline{S^F_2})_X \ (\overline{u'_{cr_2(F)}})_X$
 and 
 $\epsilon_X\hat{\psi}= (\overline{u'_F})_X$; here the natural maps $\overline{u'_F}$,  $\overline{S^F_2}$ and  $\overline{u'_{cr_2(F)}}$ are defined in Lemmas \ref{lm4}, \ref{lm5} and \ref{lm6} respectively and the maps $\hat{\phi}$ and $\hat{\psi}$ appear in the pushout diagram of Proposition \ref{quad-t-p}.
\end{prop}
\begin{proof}
To prove the existence of $\epsilon_X$ it is sufficient to show that $(\overline{S^F_2})_X \ (\overline{u'_{cr_2(F)}})_X \psi=(\overline{u'_F})_X \phi$ where $\psi$ and $\phi$ are the maps in the pushout diagram of Proposition \ref{quad-t-p}.

For $f \in \C(E,X)$ and $ m \in F(E|E)$ we have:
$$\begin{array}{ll}
(\overline{S^F_2})_X \ (\overline{u'_{cr_2(F)}})_X\ \pi\ (\delta \otimes 1)(f \otimes m)&= (\overline{S^F_2})_X \  \pi ({u'_{cr_2(F)}})_X\ (\delta \otimes 1)(f \otimes m) \mathrm{\ by\ Lemma\  \ref{lm6}}\\
& =({S^F_2})_X \ ({u'_{cr_2(F)}})_X\ (\delta \otimes 1)(f \otimes m) \mathrm{\ by\ Lemma\  \ref{lm5}}\\
& =({S^F_2})_X \ ({u'_{cr_2(F)}})_X\ (t_1(f) \otimes t_1(f) \otimes m) \mathrm{\ by\ definition\  of\ } \delta\\
& =({S^F_2})_X \ cr_2F(f,f)(m) \mathrm{\ by\ definition\  of\ } u'_{cr_2(F)} \\
&= F(\nabla^2)\  \iota^2_{(1,2)}\ cr_2F(f,f)(m) \mathrm{\ by\ definition\  of\ } S^F_2\\
&=F(\nabla^2)\ F(f \vee f)\  \iota^2_{(1,2)}(m) \mathrm{\ by\ definition\  of\ } cr_2F(f,f)\\
&=F(\nabla^2 (f \vee f))\  \iota^2_{(1,2)}(m) \mathrm{\ by\ functoriality}\\
&=F(f\ \nabla^2) \iota^2_{(1,2)}( m) \mathrm{\ since\ } \nabla^2 (f \vee f)=f\ \nabla^2\\
&=F(f)\ F(\nabla^2) \iota^2_{(1,2)}( m)\mathrm{\ by\ functoriality} \\
&=({u'_F})_X(f \otimes F(\nabla^2) \iota^2_{(1,2)}( m)) \mathrm{\ by\ definition\ of}\  u'_F \\
&=(\overline{u'_F})_X(t_2 \otimes Id)(f \otimes F(\nabla^2) \iota^2_{(1,2)}( m)) \mathrm{\ by\ Lemma\ \ref{lm4}} \\
&= (\overline{u'_F})_X(t_2 \otimes Id)(f \otimes S^F_2( m)) \mathrm{\ by\ definition\  of\ } S^F_2\\
&=(\overline{u'_F})_X(t_2 \otimes S^F_2)(f \otimes m)\\
\end{array}$$

For $f,g \in \C(E,X)$, $\xi \in \C(E,E \vee E)$, $x=(\iota^2_{(1,2)})^{-1}(\xi -i^2_{1*} r^2_{1*}(\xi) -i^2_{2*} r^2_{2*}(\xi))$ ($x \in U(E \mid E)$ by \ref{ses1}) and $a \in F(E)$ we have:
$$\begin{array}{ll}
(\overline{S^F_2})_X \ (\overline{u'_{cr_2(F)}})_X\ \psi_1(f \otimes g \otimes x \otimes  a)\\
=(\overline{S^F_2})_X \ (\overline{u'_{cr_2(F)}})_X\  \pi (t_1 \otimes t_1 \otimes (H^F)_E (t_{11} \otimes 1)) (f \otimes g \otimes x \otimes a)\\
=({S^F_2})_X  ({u'_{cr_2(F)}})_X\  (\bar{f} \otimes \bar{g} \otimes (H^F)_E (t_{11}  \otimes 1)(x\otimes a)) \mathrm{\quad by\ Lemmas\ \ref{lm5}\ and\ \ref{lm6}}\\
=({S^F_2})_X ({u'_{cr_2(F)}})_X\  (\bar{f} \otimes \bar{g} \otimes (cr_2(u'_F))_E (x\otimes a)) \mathrm{\quad by\ Lemma\  \ref{H^F}}\\
=({S^F_2})_X cr_2F(f,g) (cr_2(u'_F))_E (x\otimes a)\mathrm{\quad by\ definition\ of\ } {u'_{cr_2(F)}}\\
= F(\nabla^2) \iota^2_{(1,2)} cr_2F(f,g) (cr_2(u'_F))_E (x\otimes a)\\
= F(\nabla^2)  F(f \vee g) (u'_F)_E (\iota^2_{(1,2)}x\otimes a)\\
= F(\nabla^2)  F(f \vee g)  F(\iota^2_{(1,2)}x)( a) \mathrm{\quad by\ definition\ of\ } u'_F\\
= F(\nabla^2 (f \vee g)\iota^2_{(1,2)}x)( a) \mathrm{\quad by\ functoriality}\\
= (u'_F)_X (\nabla^2 (f \vee g)\iota^2_{(1,2)}x \otimes a) \mathrm{\quad by\ definition\ of\ } u'_F\\
= (u'_F)_X (U(\nabla^2) U(f \vee g)\iota^2_{(1,2)}x \otimes a)\\
= (u'_F)_X (U(\nabla^2) \iota^2_{(1,2)} cr_2U(f, g)x \otimes a)\\
= (\overline{u'_F})_X(t_2 \otimes 1) (U(\nabla^2) \iota^2_{(1,2)} cr_2U(f, g)x \otimes a) \mathrm{\quad by\ Lemma\ \ref{lm4}}\\
= (\overline{u'_F})_X(t_2 S^U_2 cr_2U(f, g)(x )\otimes a) \\
= (\overline{u'_F})_X \phi_1 (f \otimes g \otimes x \otimes a).
\end{array}$$
\end{proof}

In the following proposition we consider the cross-effect of  $\epsilon$.

\begin{prop} \label{ce-epsilon}
Let $F: \C \to Ab$ be a quadratic functor. Then the natural map:
$$cr_2(\epsilon)_{X,Y}: cr_2(- \otimes \mathbb S_2(F)) (X,Y) \to cr_2F(X,Y)$$
is an isomorphism for all $X,Y \in \langle E \rangle _{\C}$.
\end{prop}

\begin{proof}
For $X, Y$ arbitrary objects of $\C$ we have the following diagram whose right hand square commutes by definition of $cr_2(\epsilon)_{X,Y}$:

$$\xymatrix{
cr_2(\mathbb{T}_{11} (F(E|E)) \Delta_{\C})_{\mathfrak{S}_2}(X,Y) \ar[r]^-{cr_2(\hat{\phi})} & cr_2(- \otimes \mathbb S_2(F))(X,Y) \ar@{^{(}->}[r]^{\iota^2_{(1,2)}}   \ar[d]^{cr_2(\epsilon)_{X,Y}}& (- \otimes \mathbb S_2(F))(X \vee Y) \ar[d]^{(\epsilon)_{X \vee Y}}\\
\mathbb{T}_{11} (F(E|E))(X,Y)   \ar[u]^{(\iota^2_{(1,2)})^{-1} \pi \mathbb{T}_{11}(F(E|E))(i_1, i_2)}_{\simeq} \ar[r]_{u'_{cr_2F}}& cr_2F(X,Y)  \ar@{^{(}->}[r]^{\iota^2_{(1,2)}} & F(X \vee Y).
}$$
In the following, we prove that the left hand square is commutative. 
Since $\iota^2_{(1,2)}$ is injective, it is sufficient to prove that 
$$\iota^2_{(1,2)} cr_2(\epsilon)_{X,Y} cr_2(\hat{\phi}) (\iota^2_{(1,2)})^{-1} \pi \mathbb{T}_{11}(F(E|E))(i_1, i_2)=\iota^2_{(1,2)}  u'_{cr_2F}.$$
For $f \in \C(E,X)$, $g \in \C(E,Y)$ and $m \in F(E|E)$ we have:

$$\begin{array}{l}
\iota^2_{(1,2)} cr_2(\epsilon)_{X,Y} cr_2(\hat{\phi}) (\iota^2_{(1,2)})^{-1} \pi \mathbb{T}_{11}(F(E|E))(i_1, i_2) (\bar f \otimes \bar g \otimes m) \\
=(\epsilon)_{X \vee Y} \iota^2_{(1,2)} cr_2(\hat{\phi}) (\iota^2_{(1,2)})^{-1} \pi (\overline{i_1f} \otimes \overline{i_2 g} \otimes m)\\
=(\epsilon)_{X \vee Y} \hat{\phi} \pi (\overline{i_1f} \otimes \overline{i_2 g} \otimes m)\\
=(\overline{S^F_2})_{X \vee Y} (\overline{u'_{cr(F)}})_{X \vee Y}  \pi (\overline{i_1f} \otimes \overline{i_2 g} \otimes m) \mathrm{ \quad by\ definition\ of\ } \epsilon \mathrm{\ given\ in\ Proposition\  \ref{epsilon}}\\
=({S^F_2})_{X \vee Y} ({u'_{cr(F)}})_{X \vee Y}  (\overline{i_1f} \otimes \overline{i_2 g} \otimes m) \\
=({S^F_2})_{X \vee Y} cr_2F(i_1f,i_2g)(m) \mathrm{ \quad by\ definition\ of\ } {u'_{cr(F)}} \mathrm{\ given\ in\ Proposition\  \ref{S11T11}}\\
=F(\nabla^2) \iota^2_{(1,2)} cr_2F(i_1f,i_2g)(m) \mathrm{ \quad by\ definition\ of\ } S^F_2\\
=F(\nabla^2)  F(i_1f \vee i_2g)\iota^2_{(1,2)}(m) \\
=F(f \vee g)\iota^2_{(1,2)}(m)  \mathrm{\quad since\ } \nabla^2 (i_1f \vee i_2g)=f \vee g\\
=\iota^2_{(1,2)} cr_2F(f, g)(m) \\
=\iota^2_{(1,2)} {u'_{cr_2(F)}} (\overline{f} \otimes \overline{g} \otimes m) \mathrm{ \quad by\ definition\ of\ } {u'_{cr_2(F)}} .
\end{array}$$

Now let $X,Y \in \langle E \rangle _{\C}$.
Since $F$ is quadratic, $cr_2(F)$ is bilinear, so by Theorem \ref{bilin}, $(u'_{cr_2(F)})_{X,Y}$ is an isomorphism. By commutativity of the left hand square this implies that $cr_2(\hat{\phi})$ is injective and by Lemma \ref{i} we know that $cr_2(\hat{\phi})$ is surjective. Thus $cr_2(\hat{\phi})$ is an isomorphism, hence so is $cr(\epsilon)_{X,Y}$.

\end{proof}





\begin{prop} \label{epsilon-eq}
If $\C=\langle E \rangle_{\C}$, for $F \in Quad(\C, Ab)$ the natural map $\epsilon: - \otimes M^F \to F$ is a natural equivalence.
\end{prop}

\begin{proof}
Since $F$ is quadratic by hypothesis and $- \otimes M^F$ is quadratic by Proposition \ref{qtp-is-quad}, it is sufficient to prove that $(\epsilon)_E$ and $cr_2(\epsilon)_{(E,E)}$ are isomorphisms according to Proposition \ref{car-poly}.

By Lemma \ref{F(E)-iso} we know that $(\epsilon)_E$  is an isomorphism since 
$$((\overline{S^F_2})_E \ (\overline{u'_{cr(F)}})_E, (\overline{u'_F})_E)=((\overline{S^F_2})_E \ \overline{\mu_{ee}}, \overline{\mu_e}).$$

By Proposition \ref{ce-epsilon} we know that $cr_2(\epsilon)_{(E,E)}$ is an isomorphism.
\end{proof}

\begin{prop} \label{epsilon-counit}
The natural map $\epsilon: - \otimes M^F \to F$ is the counit of the adjunction given in Theorem \ref{thm}.
\end{prop}

\begin{proof}
For $M \in QMod^E_{\C}$ and $\alpha: - \otimes M \to F$ we want to prove that there exists a unique $\beta: M \to M^F$ such that $\epsilon (- \otimes M)(\beta)=\alpha$. Since $\epsilon$ is a natural equivalence the latter relation implies that: $(- \otimes M)(\beta)=\epsilon^{-1} \alpha$. One can check that $\beta$ such that $\beta_e=\alpha_e (\overline{P} \overline{\mu_{ee}}, \overline{\mu_e})^{-1}$ and $\beta_{ee}=cr_2(\alpha)_{E,E} \gamma_{E,E} (\mu_{ee})^{-1}$ is the unique solution.
\end{proof}

\subsection{The unit of the adjunction}

\begin{defi}
Let $\eta: Id_{QMod^E_{\C}} \to \mathbb{S}_2 \mathbb{T}_2$ be the natural map such that, for $M \in QMod^E_{\C}$, $\eta_M: M \to \mathbb{S}_2 \mathbb{T}_2(M)$ is given by $(\eta_M)_e=( \overline{P} \overline{\mu_{ee}}, \overline{\mu_e})^{-1}$ and $(\eta_M)_{ee}=(\gamma)_{E,E} \mu_{ee}^{-1}$ where $\mu_{ee}$ is the canonical isomorphism and $\gamma$ is the natural equivalence of Corollary \ref{gamma}.
\end{defi}

\begin{prop} \label{unit-1}
The natural map $\eta: Id_{QMod^E_{\C}} \to \mathbb{S}_2 \mathbb{T}_2$ is a natural equivalence.
\end{prop}
\begin{proof}
By Lemma \ref{lm33} and Corollary \ref{gamma}, $(\eta_M)_{e}$ and $(\eta_M)_{ee}$ are isomorphisms.
\end{proof}

\begin{prop} \label{unit-2}
The natural map $\eta: Id_{QMod^E_{\C}} \to \mathbb{S}_2 \mathbb{T}_2$ is the unit of the adjunction in Theorem \ref{thm}.
\end{prop}

\begin{proof}
For $F \in Quad(\C, Ab)$ and $\alpha: M \to M^F$ a morphism in $QMod^E_{\C}$, we want to prove that there exists a unique natural map $\beta: - \otimes M \to F$ such that $\mathbb{S}_2(\beta) \eta_M= \alpha$. Let $\beta$ be the composition $\epsilon (- \otimes \alpha)$. We check that $\beta$ is the unique solution of $\mathbb{S}_2(\beta) \eta_M= \alpha$.
\end{proof}

Together with Propositions \ref{unit-1}, \ref{epsilon-counit} and \ref{epsilon-eq} this achieves the proof of the assertions (A) and (B) of Theorem \ref{thm}; together with Proposition \ref{redtofinsums} these then imply assertion (C).
\section{Application to theories of cogroups}
In this section, we apply Theorem \ref{thm} to the particular case where $E$ has a cogroup structure in $\C$.
%
This is the case when $\C$ is a theory of cogroups $\mathcal{T}$ with generator $E$, or when $\C$ is an algebraic variety where a group law is part of the structure, and $E$ is the free object of rank $1$. The first case (which is generic in view of Theorem \ref{thm}) is considered in
\cite{Baues-Pira} where the authors  
define, in Proposition $3.6$, a functor:
$$u: Quad(\mathcal{T}, Gr) \to Square$$
to the category $Square$ of so-called square groups,  sending $F \in Quad(\mathcal{T}, Gr) $ to the square group $u(F)(E) = (F(E) \xrightarrow{H} F(E \mid E) \xrightarrow{P} F(E))$, see Definition \ref{square-g} below. They prove that $u$ is an equivalence of categories when $\mathcal{T}$ is the category $\langle \Z \rangle_{Gr}$ of free groups of finite rank. This is no longer true for more general theories of cogroups, even when restricting to quadratic functors with values in $Ab$. In fact, we here show how for general $\mathcal{T}$ the structure of square group has to be enriched in order to obtain an equivalence with quadratic functors from $\mathcal{T}$ to $Ab$. In particular, a second map relating $F(E)$ to $F(E\mid E)$ has to be added to the picture; in the case where $\mathcal{T}=\langle \Z \rangle_{Gr}$ this turns out to provide a new interpretation of the map $\Delta$ associated with a square group in \cite{Baues-Pira}.
Recall that a theory of cogroups is a pointed category $\mathcal{T}$ with finite sums such that each object $X$ has the structure of a cogroup given by the maps $\mu: X \to X \vee X$ and $\tau: X \to X$. Then, for $X,Y \ \in \mathcal{T}$, $\T(X,Y)$ is a group with $f\bullet g=(f,g) \mu$ and $f^{-1}=f \tau$. The identity element of $\bullet$ is the null map $0$.
Moreover, for $Z\in \T$ and $h\in \T(Y,Z)$ the map $h_*: \T(X,Y) \to \T(X,Z)$ is a homomorphism of groups.\medskip

\noindent{\bf Notation.} For brevity,
 we write in this section $i_1$, $i_2$, $r_1$ and $r_2$ instead of  $i_1^2$, $i_2^2$, $r^2_1$ and $r^2_2$, respectively.\medskip

For a theory of cogroups $\T$ we have a simpler description of a  quadratic $\T$-module relative to $E$ given in the following theorem. In particular, note that the functor $U$ disappears from the picture.

 \begin{thm} \label{proto-quad-cogroup}
Let $\T$ be a theory of cogroups.
The category of quadratic $\T$-modules relative to $E$ is equivalent to the category of diagrams of group homomorphisms:
$$M= \left(
\vcenter{
\xymatrix{
M_e \ar[dr]^{H_1} & \\
 & M_{ee}  \ar[r]^-{P} & M_e\\
T_{11}cr_2(\T(E, -))(E,E)\otimes_{\Lambda}coker(P) \ar[ur]_{H_2}
} } \right)$$
where
\begin{itemize}
\item $M_{e}$ is a left $\Lambda$-module;
\item $M_{ee}$ is a left $ \bar{\Lambda} \otimes \bar{\Lambda}$-module;
\item $P: M_{ee} \to M_e$ is a homomorphism of $\Lambda$-modules with respect to the diagonal action of $\Lambda$ on $M_{ee}$;
\item
$H_1$ is a homomorphism of abelian groups;

\item $H_2$ is a homomorphism of   $ \bar{\Lambda} \otimes \bar{\Lambda}$-modules,

\end{itemize}

 satisfying the following relations for $a \in M_{e}$, $m\in M_{ee}$, $\gamma\in cr_2(\T(E, -))(E,E)$,
$\alpha, \beta \in \mathcal{T}(E,E)$:
$$\begin{array}{lc}
(T1) & PH_1P = 2P\\
(T2) & H_1P((\overline{\alpha} \otimes \overline{\beta})m) - (\overline{\beta} \otimes \overline{\alpha})H_1P(m) = (\overline{\alpha} \otimes \overline{\beta})m - (\overline{\beta} \otimes \overline{\alpha})m\\
(T3) & H_1PH_1(a) = 2H_1(a)+ H_2(t_{11}([i_{2}, i_{1}]) \otimes \bar{a})\\
(T4) &  H_1PH_2(t_{11}\gamma \otimes \bar{a}) = H_2(t_{11}(\gamma\bullet \tau\gamma) \otimes \bar{a}) \\
(T5) &  (\nabla^2 \gamma \bullet \alpha \bullet \beta)a =\alpha a + \beta a + P((\bar{\alpha} \otimes \bar{\beta})H_1(a)) + PH_2(t_{11}\gamma \otimes \bar{a})\\
 (T6) &  H_1(\alpha a)=H_2(t_{11}h(\alpha)\otimes \bar{a})+(\overline{\alpha} \otimes \overline{\alpha}) H_1(a)
\end{array}$$
where $[i_{2}, i_{1}]=i_2\bullet i_1\bullet (i_2)^{-1}\bullet (i_1)^{-1}$, $\tau$ is  the canonical switch of $E\vee E$ and 
$h: \T(E,E) \to \T(E, E \mid E)$ is given by 
$$h(\alpha)=(\iota^2_{(1,2)})^{-1} (((i_1 \bullet i_2) \alpha ) \bullet (i_2 \alpha)^{-1} \bullet (i_1 \alpha)^{-1}).$$
\end{thm}

Note that the map $h$
describes the deviation of endomorphisms of $E$ from being morphisms of cogroups.

The plan of the proof of Theorem \ref{proto-quad-cogroup} is as follows: we first compute, in several steps,
the symmetric bifunctor $T_{11}(cr_2U)$. The result is then used to split up the map $\hat{H}$ in a given quadratic $\T$-module into the two maps $H_1$ and $H_2$ occuring in Theorem \ref{proto-quad-cogroup}, and to translate the properties of the map 
$\hat{H}$ in terms of $H_1$ and $H_2$, which leads to the relations $(T1)$-$(T6)$.

As a main tool  we need some elementary facts about augmentation ideals of group rings.

\subsection{Augmentation ideals: recollections and action of the linearization functor}
 Recall that the augmentation ideal $IG$ of a group $G$ is the kernel of the augmentation map $\epsilon: \mathbb{Z}[G] \to \mathbb{Z}$  of the group ring $\mathbb{Z}[G]$. $IG$ is a free $\mathbb{Z}$-module generated by the elements $g-1$ for $g \in G\setminus \{1\}$. For a subset $H$ of $G$ we denote by $IH$ the subgroup of $IG$ generated by the elements of the form $h-1$ for $h \in H$.
 The key ideas of this section are to use the natural isomorphism $\Xi: I(\T(E,X)) \to U(X)$ such that $\Xi(f-1) =f$ and the trivial but useful formula $a.b-1=(a-1)+(b-1)+(a-1).(b-1)$ in $IG$.
Moreover, we need the following elementary result which is wellknown.

\begin{prop} \label{prop-cog-2}
For $G \in Gr$ we have a natural isomorphism of abelian groups:
$$\theta: IG / (IG) ^2\xrightarrow{\simeq} G/[G,G]$$
such that for $g \in G$, $\theta(\overline{g-1})=\overline{g}$.
\end{prop}

\subsection{Computation of $cr_2U(X,Y)$ }
\label{cogroups-1}
In a first step we compute the cross-effect of the functor $U$, which by the isomorphism $\Xi$ above comes down to computing the cross-effect of the composite functor $I\T(E,-)$ where $I: Gr \to Ab$ is given by taking the augmentation ideal. More generally, we have the following result.

\begin{prop} \label{cr2IF}
Let $F: \C \to Gr$ be a reduced functor. There is 
 an isomorphism of bifunctors $\T\times \T \to Ab$:
$$\Theta:  I(F(X \mid Y)) \oplus (IF(X) \otimes  IF(Y) )\oplus ( I(F(X \mid Y)) \otimes I(F(X) \times F(Y) ))   \xrightarrow{\simeq} cr_2(IF)(X, Y)$$
given by
$$\Theta(x-1,0,0)=x-1$$
$$\Theta(0,(y-1) \otimes (z-1),0)=(i_{1*}y-1). (i_{2*}z-1)$$
$$\Theta(0,0,(u-1) \otimes (v-1))=(u-1).(sv-1)$$
where $s: F(X) \times F(Y)  \to  F(X \vee Y)$ is the map defined by 
$s(x,y)=i_{1*}x . i_{2*}y$, \mbox{ the point 
  $.$} denoting the group structure on $F(X \vee Y)$.
 \end{prop}

The proof requires the following elementary fact.

\begin{lm} \label{prop-cog-1}
Let $G$ be a group and $H, K$ be two subsets of $G$ containing $1$ such that  each $g \in G$ admits a unique decomposition $g=h_gk_g$ with $h_g \in H$ and $k_g \in K$. Then the map:
$$\phi: IH \oplus IK \oplus (IH \otimes IK) \to IG$$
given by
$\phi(h-1,0,0)=h-1$, $\phi(0, k-1, 0)=k-1$ and $\phi(0,0,(h'-1) \otimes (k'-1))=(h'-1).(k'-1)$
is an isomorphism of $\mathbb{Z}$-modules.
\end{lm}

\begin{proof}
We define 
$$\psi: IG \to IH \oplus IK \oplus (IH \otimes IK) $$
by $$\psi(g-1)=(h_g-1,k_g-1,(h_g-1) \otimes (k_g-1)).$$ One readily checks that $\psi$ is the inverse of $\phi$, using the fact that the unique decomposition of $h \in H \subset G$ (resp. $k \in K \subset G$) in $HK$ is $h.1$ (resp. $1.k$).
\end{proof}

\noindent{\it Proof of Proposition \ref{cr2IF}.}
For $X, Y \in \C$ we have a short exact sequence:

$$1 \to F(X \mid Y) \to F(X \vee Y) \xrightarrow{(r_{1*}, r_{2*})^t} F(X) \times F(Y) \to 1.$$

Since $i_{1*}$ and $i_{2*}$ are group morphisms we have $s(1,1)=1$ and since $F$ is reduced, $s$ is a set-theoretic section of $(r_{1*}, r_{2*})^t$ natural in $X$ and $Y$.
Hence the subsets $H=F(X \mid Y)$ and $K=s(F(X) \times F(Y))$ of the group $F(X \vee Y)$ satisfy the conditions of Lemma \ref{prop-cog-1} and we have an isomorphism of bifunctors:
$$ I(F(X \vee Y)) \simeq I(F(X \mid Y)) \oplus I(s(F(X) \times F(Y))) \oplus ( I(F(X \mid Y)) \otimes I(s(F(X) \times F(Y)))).$$

Since the map $s$ is a bijection of $F(X) \times F(Y)$ onto the set $K$, applying Proposition \ref{prop-cog-1} to the product group  $F(X) \times F(Y)$ provides an isomorphism of bifunctors:
\begin{eqnarray*}
I(F(X \vee Y)) &\simeq& I(F(X \mid Y)) \oplus (IF(X) \oplus IF(Y) \oplus IF(X) \otimes  IF(Y) )\\
&&\oplus ( I(F(X \mid Y)) \otimes I(F(X) \times F(Y) ))
\end{eqnarray*}
Now the assertion follows from  (\ref{cr2asquot}).\hfill$\Box$\medskip

The following application of Proposition \ref{cr2IF} illustrates the power of the linearization functor $T_1$.

\begin{prop} \label{ideal-lin}
Let $F:\C \to Gr$ be a reduced functor. Then we have a natural isomorphism:
$$\theta^F: T_1(IF) \to T_1F$$ 
defined by $\theta^F(t_1(x-1))=t_1(x)$ for $X \in \C$ and $x \in F(X)$.
\end{prop}

\begin{proof}
By the isomorphism $\Theta$ in Proposition \ref{cr2IF} we have
 $$T_1(IF)(X)=coker(cr_2(IF)(X,X) \xrightarrow{S_2^{IF(-)}} IF(X))$$
 $$=coker( I(F(X \mid X)) \oplus (IF(X) \otimes  IF(X) )\oplus ( I(F(X \mid X)) \otimes I(F(X) \times F(X) ) \xrightarrow{S_2^{IF(-)} \Theta} IF(X))$$
 $$S_2^{IF(-)}\Theta(x-1,0,0)=S_2^{IF(-)}(x-1)=I(S_2^F)(x-1)$$
$$S_2^{IF(-)}\Theta(0,(y-1) \otimes (z-1),0)=S_2^{IF(-)}((i_{1*}y-1). (i_{2*}z-1))=IF(\nabla^2)((i_{1*}y-1). (i_{2*}z-1))$$
$$=(\nabla^2_*i_{1*}y-1). (\nabla^2_*i_{2*}z-1)
=(y-1).(z-1) \in (IF(X))^2$$
$$S_2^{IF(-)}\Theta(0,0,(u-1) \otimes (v-1))=IF(\nabla^2)((u-1).(sv-1))=(\nabla^2 u-1).(\nabla^2sv-1) \in (IF(X))^2.$$
By Proposition \ref{prop-cog-2} we obtain:
$$T_1(IF)(X)=IF(X)/(Im(I(S^F_2))+(IF(X))^2) \simeq F(X)^{ab}/ ab(Im(S^F_2))$$
$$\simeq (F(X)/Im(S^F_2))^{ab}\simeq (T_1F(X))^{ab} \simeq T_1F(X) \mbox{\ by Proposition \ref{lin-ab}}$$
where $-^{ab}: Gr \to Ab$ is the abelianization functor and $ab: F(X) \to F(X)^{ab}$. 
\end{proof}

\subsection{Computation of $T_{11}(cr_2U)(X,Y)$ }
In order to compute $T_{11}(cr_2U)(E,E)$ as a right $\Lambda$-module we apply the results of  section \ref{cogroups-1} to the reduced functor $\T(E,-): \C \to Gr$.
\begin{nota}
In the sequel we use the following abbreviations:
$$\T(E,X \mid Y):= cr_2(\T(E,-)) (X,Y)\:;$$
$$T_1\T(E,X):=T_1( \T(E,-))(X)\:;\quad T_1I\T(E,X):=T_1(I \T(E,-))(X);$$
$$T_{11}  \T(E,X \mid Y):=T_{11}( cr_2(\T(E,-)))(X,Y)\:; \quad T_{11}I \T(E,X \mid Y):=T_{11}(I cr_2(\T(E,-)))(X,Y).$$
\end{nota}

\begin{prop} \label{Ypsilonprop} There is 
a binatural isomorphism:
\begin{equation}  \label{Upsilon}
 \Upsilon : T_{11} \T(E,X\mid Y) \oplus (  T_1\T(E,X)  \otimes T_1\T(E,Y) ) \xrightarrow{\simeq}T_{11}cr_2(I\T(E,-))(X, Y)
\end{equation}
such that
$$\Upsilon(t_{11} \xi, t_1f \otimes t_1g)=t_{11} (\iota^2_{(1,2)})^{-1}((\xi-1)+(i_1f-1) \bullet (i_2g-1)).$$ 
Moreover, the right action of $\Lambda$ on
$T_{11}cr_2(I\T(E,-))(X, Y)$ induced by precomposition in $\T$ is given on the above components by
\begin{equation}\label{1stpiece} \Upsilon(t_{11} \xi,0).\alpha=\Upsilon(t_{11}(\xi \alpha),0)
\end{equation}
\begin{equation} \label{2ndpiece}
\Upsilon(0,t_1f \otimes t_1g) . \alpha=\Upsilon(t_{11}((f\vee g)h(\alpha)), t_1(f\alpha) \otimes t_1(g \alpha))
\end{equation}
for $\alpha \in \T(E,E)$.
\end{prop}

\begin{proof} Taking $F=\T(E,-)$ in Proposition \ref{cr2IF} gives an isomorphism of bifunctors
$$I(\T(E,X\mid Y)) \oplus(  I(\T(E,X))  \otimes  I(\T(E,Y)) ) \oplus (I(\T(E,X \mid Y)) \otimes I(\T(E,X) \times \T(E,Y)))$$
$$  \xrightarrow{\Theta} cr_2(I\T(E,-))(X, Y).$$
We have the following binatural isomorphisms 
\begin{eqnarray*}
I(\T(E,X \mid Y)) \otimes I (\T(E,X) \times \T(E,Y)) &\simeq&  (I(\T(E,X \mid Y)) \otimes I( \T(E,X)))\\
&&\oplus  (I(\T(E,X \mid Y)) \otimes I( \T(E,Y))) \\
&& \oplus  (I(\T(E,X \mid Y)) \otimes I( \T(E,X)) \otimes I( \T(E,Y))).
\end{eqnarray*}
Thus
$I(\T(E,X \mid Y)) \otimes I (\T(E,X) \times \T(E,Y))$ is a sum of  bifunctors which are diagonalizable as functors in $X$ or in $Y$. So
Proposition \ref{T1BD=0} implies that $T_{11}(I(\T(E,X \mid Y)) \otimes I(\T(E,X) \times \T(E,Y)))=0$.
  Using Proposition \ref{ideal-lin} and Example \ref{TFoG} the isomorphism $T_{11}(\Theta)$ thus becomes the desired isomorphism $\Upsilon$.

The structure of right $\Lambda$-module on $T_{11} I\T(E,X\mid Y)$ is induced by the inclusion $\T(E,X\mid Y) \to \T(E, X \vee Y) $; this implies relation (\ref{1stpiece}). To prove relation (\ref{2ndpiece})
let  $f \in   \T(E,X) $ and $g \in \T(E,Y) $. Then:
\begin{eqnarray*}
\iota^2_{(1,2)} \Theta (0, (f-1) \otimes (g-1), 0).\alpha&=&((i_{1}f-1) \bullet (i_{2}g-1)). \alpha\\
&=&((i_{1}f\bullet i_{2}g-1)-(i_{1}f-1)-(i_{2}g-1)). \alpha\\
&=&((i_{1}f\bullet i_{2}g)\circ \alpha-1)-(i_{1}f\circ \alpha-1)-(i_{2}g\circ \alpha-1)\\
&=&(\omega \bullet (i_{1}f  \alpha) \bullet (i_{2}g \alpha)-1)-(i_{1}f \alpha-1)-(i_{2}g \alpha-1)
\end{eqnarray*}
where $\omega=((i_{1}f\bullet i_{2}g)\circ \alpha) \bullet (i_{2}g \circ \alpha)^{-1} \bullet (i_{1}f  \circ \alpha)^{-1}.$ Hence:

$$\begin{array}{lll}
\iota^2_{(1,2)} \Theta (0, (f-1) \otimes (g-1), 0). \alpha\\
= (\omega-1)+((i_{1}f  \alpha) \bullet (i_{2}g  \alpha)-1)+(\omega-1) \bullet ((i_{1}f  \alpha) \bullet (i_{2}g \alpha)-1)
-(i_{1}f \alpha-1)\\ \phantom{=} -(i_{2}g \alpha-1)\\
= (\omega-1)+(i_{1}f  \alpha-1) \bullet (i_{2}g  \alpha-1)
+(\omega-1) \bullet ((i_{1}f  \alpha) \bullet (i_{2}g \alpha)-1).
\end{array}$$
Now observe that  
\begin{equation}\label{vorschaufeln}
(i_1f) \bullet (i_2g) = (f \vee g)_*(i_1  \bullet  i_2).
\end{equation}
Thus $\omega=(f\vee g)_*((i_{1} \bullet i_{2} )\alpha) \bullet 
(f\vee g)_*(i_{2}\alpha)^{-1}  \bullet  (f\vee g)_*(i_{1}\alpha)^{-1}
=(f\vee g)_*(h(\alpha))$.
It follows that
$\omega=\T(E,f\mid g)(h(\alpha))\in \T(E,X\mid Y)$. Hence 
$$(\omega-1) \bullet ((i_{1}f  \alpha) \bullet (i_{2}g \alpha)-1) \in \Theta\big( \T(E,X \mid Y) \otimes I(\T(E,X) \times \T(E,Y))\big),$$ and we deduce that
\begin{eqnarray*}
\Upsilon(0,t_1f \otimes t_1g) . \alpha&=&t_{11}(\Theta(0,(f-1) \otimes (g-1), 0). \alpha)\\
&=&\Upsilon(t_{11}((f\vee g)h(\alpha)), t_1(f\alpha) \otimes t_1(g \alpha)),
\end{eqnarray*}
as desired.
\end{proof}

Using the natural  isomorphism  $\Xi: I(\T(E,X)) \to U(X)$ we finally obtain:

\begin{cor} \label{T11crXi}
We have a natural isomorphism of right $\Lambda$-modules:
$$T_{11}cr_2(\Xi) \Upsilon: T_{11} \T(E,X\mid Y) \oplus (  T_1\T(E,X) \otimes T_1\T(E,Y))  \rightarrow T_{11}(cr_2U)(X,Y)$$
such that:
$$T_{11}cr_2(\Xi) \Upsilon(t_{11} \xi, t_1f \otimes t_1g)=t_{11}( \xi + \rho^2_{(1,2)}(i_1f \bullet i_2g))$$
and where the action of $\Lambda$ on the domain of $T_{11}cr_2(\Xi) \Upsilon$ is given by:
$$(t_{11} \xi, t_1f \otimes t_1g). \alpha=(t_{11}((\xi \alpha) \bullet (f\vee g)h(\alpha)), t_1(f \alpha) \otimes t_1(g \alpha)).$$
\end{cor}

\begin{proof}
It suffices to combine Proposition \ref{Ypsilonprop} with the isomorphism $T_{11}cr_2(\Xi)$. In fact,
\begin{eqnarray*}
T_{11}cr_2(\Xi) \Upsilon(t_{11} \xi, t_1f \otimes t_1g)&=&T_{11}cr_2(\Xi)t_{11} (\iota^2_{(1,2)})^{-1}((\xi-1)+(i_1f-1) \bullet (i_2g-1))\\
&=&t_{11} (\iota^2_{(1,2)})^{-1} \Xi((\xi-1)+(i_1f-1) \bullet (i_2g-1))\\
&=&t_{11} (\iota^2_{(1,2)})^{-1} \Xi((\xi-1)+(i_1f \bullet i_2g-1)-(i_1f-1)-(i_2g-1))\\
&=&t_{11} (\iota^2_{(1,2)})^{-1} (\xi+i_1f \bullet i_2g-i_1f-i_2g)\\
&=&t_{11}(\xi+\rho^2_{12}(i_1f \bullet i_2g)) \mathrm{\ by\ (\ref{ses1})}
\end{eqnarray*}

\end{proof}


\subsection{The action of the involution on $T_{11}(cr_2U)(X,Y)$} \label{involution}
Since $\hat{H}$ is a morphism of symmetric $\overline{\Lambda} \otimes \overline{\Lambda}$-modules we have to understand the action of the involution on the two components of $T_{11}(cr_2U)(X,Y)$ according to Corollary \ref{T11crXi}.

For $\xi\in \T(E,X\mid Y)$ we have the relation
$$\label{TUxi} \bar{T}^U t_{11}(\xi) = t_{11}T^{\T(E,-)} (\xi)$$
in $T_{11}(cr_2U)(X,Y)$ by Proposition \ref{crissymbif}.

For $f,g\in \T(E,X)$ we have
\begin{eqnarray*}
I(\tau) \iota_{(1,2)}^2\Theta (0,(f-1) \otimes (g-1),0)&=&I(\tau) ((i_1f-1) \bullet (i_2g-1))\\
&=&(i_2f-1) \bullet (i_1g-1)\\
&=&(i_2f \bullet i_1g-1)-(i_2f-1)-(i_1g-1)\\
&=&([i_2f, i_1g] \bullet i_1g \bullet i_2f-1)-(i_2f-1)-(i_1g-1)\\
&=& ([i_2f, i_1g]-1) + ( i_1g \bullet i_2f-1)+\\
&&([i_2f, i_1g]-1) \bullet ( i_1g \bullet i_2f-1)-(i_2f-1)-(i_1g-1)\\
&=&([i_2f, i_1g]-1) +  (i_1g-1) \bullet( i_2f-1)+\\
& &([i_2f, i_1g]-1) \bullet ( i_1g \bullet i_2f-1)
\end{eqnarray*}

But $[i_2f, i_1g]\in \T(E,X\mid Y)$, hence we obtain:
\begin{eqnarray*}
\bar{T}^U\Upsilon(0,t_1f\otimes t_1g) &=& t_{11}
(\iota_{(1,2)}^2)^{-1} I(\tau) \iota_{(1,2)}^2\Theta (0,(f-1) \otimes (g-1),0)\\
&=& \Upsilon(t_{11}[i_2f, i_1g],t_1g\otimes t_1f)
\end{eqnarray*}
Combining this relation with (\ref{TUxi}) and Corollary \ref{T11crXi} we obtain
\begin{equation}
\bar{T}^U t_{11} \big( \xi + \rho_{(1,2)}^2((i_1f)\bullet (i_2g)) \big) =  t_{11} \left(T^{\T(E,-)}(\xi) + [i_2f,i_1g] + \rho_{(1,2)}^2((i_1g)\bullet (i_2f)) \right)
\end{equation}


\subsection{The homomorphism $\hat{H}$}
By definition of a quadratic $\C$-module the map 
$${\hat{H}}: T_{11}(cr_2(U))(E,E)\otimes_{\Lambda}M_e \rightarrow M_{ee} $$ is a morphism of symmetric $\overline{\Lambda} \otimes \overline{\Lambda}$-modules. In the case of cogroups we have the following equivalent description of $\hat{H}$:

\begin{prop}\label{Hsplits}
For a theory of cogroups $\mathcal{T}$, the morphism of symmetric $\overline{\Lambda} \otimes \overline{\Lambda}$-modules $\hat{H}$ is equivalent to the following data:
\begin{enumerate}
\item a morphism of abelian groups:
$$H_1: M_e \to M_{ee}$$
\item a morphism of symmetric  $\overline{\Lambda} \otimes \overline{\Lambda}$-modules
$$H'_2: T_{11}\T(E,E\mid E) \otimes_{\Lambda} M_e \to M_{ee}$$
\end{enumerate}
satisfying the following relations for $\alpha \in \T(E,E)$ and $a \in M_e$:
\begin{equation}\label{H1alphaa}
H_1(\alpha a)=H'_2(t_{11}h(\alpha) \otimes a) + (\overline{\alpha} \otimes \overline{\alpha}) H_1(a)
\end{equation}
\begin{equation}\label{TH1}
TH_1(a)= H_1(a)+H'_2(t_{11}([i_{2}, i_{1}]) \otimes a).
\end{equation}
\end{prop}

\begin{proof} 
Using Corollary \ref{T11crXi} and relation (\ref{vorschaufeln}) one easily verifies that when dropping the symmetry conditions, the map $\hat{H}$ is equivalent with the maps $H_1$ and $H'_2$ as in (1) and (2) satisfying relation (\ref{H1alphaa}). In fact, given $\hat{H}$ we may define
\begin{equation}\label{H1def}
H_1(a) = \hat{H}(t_{11}\rho_{(1,2)}^2(i_1\bullet i_2) \otimes a)
\end{equation}
\begin{equation}\label{H2def}
H'_2(t_{11}\xi \otimes a) = \hat{H}(t_{11}\xi \otimes a).
\end{equation}
Conversely, given $H_1$ and $H'_2$ the associated map $\hat{H}$ is determined by the relations (\ref{H2def}) and 
$$\hat{H}(t_{11}\rho_{(1,2)}((i_1\alpha)\bullet (i_2\beta)) \otimes a) = (\bar{\alpha} \otimes \bar{\beta})H_1(a)$$
for $\alpha, \beta \in \T(E,E)$.
It remains to check that $\hat{H}$ commutes with the respective involutions if and only if relation (\ref{TH1}) holds. In fact,
$$\begin{array}{lll}
\hat{H}\bar{T}^U (t_{11} ( \xi + \rho_{(1,2)}^2((i_1\alpha)\bullet (i_2\beta)) ) \otimes a )\\
=  \hat{H} \left(t_{11} \left(T^{\T(E,-)}(\xi) + [i_2\alpha,i_1\beta] + \rho_{(1,2)}^2((i_1\beta)\bullet (i_2\alpha)) \right)\otimes a\right)  \hspace{1mm}\mbox{by (\ref{vorschaufeln})}\\
=  \hat{H} \left(t_{11} \left(T^{\T(E,-)}(\xi) + [i_2\alpha,i_1\beta] + cr_2U(\beta,\alpha)\rho_{(1,2)}^2(i_1\bullet i_2) \right)\otimes a\right)\\
= H'_2\left(t_{11} \left(T^{\T(E,-)}(\xi) + [i_2\alpha,i_1\beta]\right)\otimes a\right) +(\bar{\beta} \otimes\bar{\alpha})H_1(a) \\
= H'_2\left(t_{11} T^{\T(E,-)}(\xi) \otimes a\right) +  (\bar{\beta} \otimes\bar{\alpha})H'_2 ([i_2 ,i_1]\otimes a) +
(\bar{\beta} \otimes\bar{\alpha})H_1(a) 
\end{array}$$
since $[i_2\alpha,i_1\beta]=(\beta \vee \alpha)_*[i_2 ,i_1]=
\T(E, \beta\mid\alpha)([i_2 ,i_1])$ and $H'_2$ is $\bar{\Lambda} \otimes \bar{\Lambda}$-linear.
On the other hand, 
\begin{eqnarray*}
T\hat{H}(t_{11} ( \xi + \rho_{(1,2)}^2((i_1\alpha)\bullet (i_2\beta)) ) \otimes a ) 
&=& T\hat{H}(t_{11} ( \xi + cr_2U(\alpha,\beta)\rho_{(1,2)}^2( i_1  \bullet  i_2) ) \otimes a ) \\
&=&TH'_2 (t_{11}    \xi \otimes a ) +   
T((\bar{\alpha} \otimes\bar{\beta})H_1(a) )\\
&=& H'_2\left(t_{11} T^{\T(E,-)}(\xi)\otimes a\right) + (\bar{\beta} \otimes\bar{\alpha})TH_1(a) .
\end{eqnarray*}
Thus if relation (\ref{TH1}) holds, $\hat{H}$ commutes with the respective involutions; the converse is also true as we may take $\alpha=\beta=1_E$.
\end{proof}

\subsection{The conditions $(QM1)$ and $(QM2)$}
In this section we translate the conditions $(QM1)$ and $(QM2)$ to the case of cogroups. We begin by the remark that for each $\xi \in  \T(E,E \vee E)$, there exists $\gamma \in \T(E,E \mid E)$ such that $\xi=\gamma \bullet  i_1r_1 \xi \bullet i_2 r_2 \xi $. In fact, $(r_{1*},r_{2*})^t(\xi)=(r_1 \xi, r_2 \xi)=(r_{1*},r_{2*})^t(i_1r_1 \xi \bullet i_2 r_2 \xi)$ hence $\xi \bullet (i_1r_1 \xi \bullet i_2 r_2 \xi)^{-1} \in \T(E,E \mid E)$. Thus any element of $\T(E,E \vee E)$ can be written in the form $\gamma \bullet i_1 \alpha \bullet i_2 \beta$ for $\alpha, \beta \in \T(E,E)$ and $\gamma \in \T(E,E \mid E)$.

\begin{prop} \label{QM1-cogroup}
For a theory of cogroups $\mathcal{T}$, the condition $(QM1)$ is equivalent to the following equation:
$$(\nabla^2  \gamma \bullet \alpha \bullet \beta ).a=\alpha a+\beta a+P((\bar{\alpha} \otimes \bar{\beta}) H_1(a))+PH'_2(t_{11}\gamma\otimes a)$$
for $\alpha, \beta \in \T(E,E)$, $a \in M_e$ and $\gamma \in \T(E,E \mid E)$. 
\end{prop}

\begin{proof}
For $\xi \in \T(E, E \vee E)$ such that $\xi=\gamma \bullet  i_1 \alpha \bullet i_2 \beta $ and $a \in M_e$, we have the following equations in $U(E \vee E)$:
\begin{eqnarray*}
\iota^2_{(1,2)}\circ \rho^2_{(1,2)}(\xi)&=&\iota^2_{(1,2)}\circ \rho^2_{(1,2)}(\gamma \bullet i_1\alpha \bullet i_2 \beta)\\
&=&\gamma \bullet  i_1\alpha \bullet i_2 \beta -i_1r_1(\gamma \bullet i_1\alpha \bullet i_2 \beta)-i_2r_2(\gamma \bullet i_1\alpha \bullet i_2 \beta) \mathrm{\ by\ \ref{ses1}}\\
&=&\gamma \bullet i_1\alpha \bullet i_2 \beta -0 \bullet i_1\alpha \bullet 0-0 \bullet 0 \bullet i_2 \beta\\
&=&\gamma \bullet i_1\alpha \bullet i_2 \beta-i_1\alpha - i_2 \beta.
\end{eqnarray*}
Hence:
$$\begin{array}{ll}
(\Xi)^{-1}(\iota^2_{(1,2)}\circ \rho^2_{(1,2)}(\xi))\\
=(\gamma \bullet i_1\alpha  \bullet i_2 \beta-1)-(i_1\alpha-1) -(i_2 \beta -1) \\
= ( \gamma-1)+(i_1\alpha \bullet i_2 \beta-1)+(\gamma -1) \bullet (i_1\alpha \bullet i_2 \beta-1) -(i_1\alpha-1) -(i_2 \beta -1) \\ 
=(\gamma-1)+(i_1\alpha -1) \bullet (i_2 \beta-1)+(\gamma -1) \bullet (i_1\alpha \bullet i_2 \beta-1)
\end{array}$$
and:$$(\Upsilon)^{-1}  (T_{11}cr_2(\Xi))^{-1}(t_{11}{\rho^2_{12}}(\xi))=(t_{11}{\gamma}, t_1{\alpha} \otimes t_1{\beta}).$$
We deduce that:
\begin{eqnarray}
\hat{H}(t_{11}{\rho^2_{12}}(\xi) \otimes a) &=& \hat{H}(T_{11}cr_2(\Xi) \Upsilon(t_{11}{\gamma}, t_1{\alpha} \otimes t_1{\beta}) \otimes a)\nonumber\\
&=&\hat{H}(t_{11}(\gamma+\rho^2_{(1,2)}(i_1 \alpha \bullet i_2 \beta)) \otimes a)    \mathrm{\ by \ Corollary\ \ref{T11crXi}}\nonumber\\
&=& H'_2(t_{11}{\gamma} \otimes a)+ ( t_1{\alpha} \otimes t_1{\beta}) H_1(a).\label{PhatH}
\end{eqnarray}
On the other hand 
\begin{eqnarray*}
(\nabla^2 \xi)a-(r^2_1 \xi)a -(r^2_2 \xi) a&=&(\nabla^2(\gamma \bullet i_1 \alpha \bullet i_2 \beta))a-(r^2_1(\gamma \bullet i_1 \alpha \bullet i_2 \beta))a -(r^2_2(\gamma \bullet i_1 \alpha \bullet i_2 \beta)) a\\
&=& (\nabla^2 \gamma \bullet \alpha \bullet \beta  )a-(0\bullet \alpha \bullet  0)a-(0\bullet 0 \bullet \beta)a\\
&=&(\nabla^2 \gamma  \bullet \alpha \bullet \beta )a-\alpha a- \beta a.
\end{eqnarray*}
\end{proof}

\begin{prop}\label{QM12}
For a theory of cogroups $\mathcal{T}$, the condition $(QM2)$ is equivalent to the following equations for $m \in M_{ee}$ and $\gamma \in \T(E,E \mid E)$:
$$(QM2-1)\quad H'_2(t_{11}{\gamma} \otimes Pm)=0$$
$$(QM2-2) \quad H_1(Pm)=m+Tm.$$
. 
\end{prop}

\begin{proof}
For $m \in M_{ee}$ and $\xi \in \T(E, E \vee E)$ such that $\xi= \gamma \bullet  i_1 \alpha \bullet i_2 \beta$ as above, we have:
\begin{eqnarray*}
\hat{H}(t_{11}\rho^2_{(1,2)}(\xi) \otimes Pm)&=&H'_2(t_{11}{\gamma} \otimes Pm)+ (t_1{\alpha} \otimes t_1{\beta}) H_1(Pm)
\end{eqnarray*}
by (\ref{PhatH}). On the other hand:
$$(t_1{r_1(\xi)} \otimes t_1{r_2(\xi)})(m+Tm)=(t_1{\alpha} \otimes t_1{\beta})(m+Tm).$$
Hence the condition $(QM2)$ is equivalent to the condition:
$$H'_2(t_{11}{\gamma} \otimes Pm)+ (t_1{\alpha} \otimes t_1{\beta}) H_1(Pm)
=(t_1{\alpha} \otimes t_1{\beta})(m+Tm).$$
Since $\alpha$, $\beta$ and $\gamma$ are independent, taking $\alpha=0$ we deduce $(QM2-1)$ and taking $\gamma=0$ and $\alpha=\beta=1_E$ we deduce $(QM2-2)$. The converse is clear.
\end{proof}

\subsection{Proof of Theorem  \ref{proto-quad-cogroup}}
This now is an easy combination of Corollary \ref{T11crXi} and the Propositions \ref{Hsplits} and \ref{QM12}: let $M$ be a quadratic $\C$-module relative to $E$. First use Proposition  \ref{Hsplits} to replace the map $\hat{H}$ by $H_1$ and $H_2'$, then $H_2'$ by $H_2$ using relation (QM2-1). Now the main feature is that relation (QM2-2) implies that $T$ is determined by $H_1$ and $P$, as
$$ T=H_1P - 1.$$
Using this, the relation
$PT=P$ becomes (T1), which, in the converse proof, implies that $T$ is an involution. Next (T2) translates the relation $T((\alpha \otimes \beta)m)= (\beta \otimes \alpha)T(m)$. Relation (\ref{TH1}) becomes (T3), while (T4) translates the fact that $H_2$ is compatible with the respective involutions.
Finally, (T5) is relation (QM1) and (T6) is (\ref{H1alphaa}).

\subsection{Application: quadratic functors from free groups of finite rank to $Ab$} \label{application}
In this section we apply Theorems \ref{thm} and  \ref{proto-quad-cogroup} to the category $\C=Gr$, or equivalently, to the theory $\T=\langle \mathbb{Z} \rangle_{Gr}$ of free groups of finite rank. Baues and Pirashvili described   quadratic functors from $\T$ to $Gr$ and $Ab$ in terms of simpler data in \cite{Baues-Pira}. We start by recalling the simplified version of this description given in \cite{BJP}.

\begin{defi}[\cite{Baues-Pira}] \label{square-g}
A square group is a diagram 
$$M=(M_{e} \xrightarrow{H} M_{ee} \xrightarrow{P} M_e)$$
where $M_e$ is a group and $M_{ee}$ is an abelian group. Both groups are written additively. Moreover $P$ is a homomorphism and $H$ is a quadratic map, that is, the cross effect
$$(a \mid b)_H=H(a+b)-H(a)-H(b)$$
is linear in $a,b \in M_e$. In addition the following identities are satisfied for all $x, y \in M_e$ and $a,b \in M_{ee}$:
$$(Pa\mid y)_H=0=(x\mid Pb)_H;$$
$$P(x \mid y)_H=-x-y+x+y;$$
$$PHP(a)=2P(a).$$
\end{defi}

\begin{thm}[\cite{Baues-Pira}]
The category of quadratic functors from $\langle \mathbb{Z} \rangle_{Gr}$ to $Ab$ (or of quadratic endofunctors of $Gr$ preserving filtered colimits and reflexive coequalizers) is equivalent to the category of square groups with $H$ linear.
\end{thm}

We reprove this theorem by specializing our general results to the case $\T=\langle \mathbb{Z} \rangle_{Gr}$, as follows.

\begin{proof} Since $\langle \mathbb{Z} \rangle_{Gr}$ is a theory of cogroups we can apply Theorem \ref{proto-quad-cogroup}.
Let $M$ be a diagram as in this theorem. Condition $(T1)$ shows that 
$$Square(M)=(M_{e} \xrightarrow{H_1} M_{ee} \xrightarrow{P} M_e)$$
is a square group such that $H_1$ is linear. So we have to prove that the remaining structure of $M$ is determined and well defined by $Square(M)$ when one requires  the conditions $(T2)$ - $(T6)$ to hold.

First of all,
\begin{eqnarray*}
\overline{\Lambda}&=& (T_1U)(\Z) \mathrm{\ by\ Proposition\ \ref{Lambda-bar}}\\
&\simeq &(T_1I(Gr(\Z,-)))(\Z)\\
&\simeq& T_1(Gr(\Z,-))(\Z) \mathrm{\ by\ Proposition\ \ref{ideal-lin}}\\
&\simeq&T_1(Id_{Gr})(\Z)\\
&\simeq& \Z \mathrm{\ by\ Proposition\ \ref{T1Id}}.
\end{eqnarray*}
We deduce that $\overline{\Lambda} \otimes \overline{\Lambda} \simeq \Z$ hence condition $(T_2)$ is trivially satisfied. 

The isomorphism of endofunctors of $Gr$:
$$\nu: Gr(\Z,-) \to Id_{Gr},\quad \nu_G(f) =f(1)$$
induces an isomorphism of bifunctors
$$T_{11}(cr_2\nu) : T_{11}cr_2(Gr(\Z,-)) \to T_{11}cr_2(Id_{Gr}) $$
such that 
\begin{equation}\label{T11crnu}
T_{11}(cr_2\nu)_{G,H}(t_{11}(\xi)) = t_{11}(\xi(1))
\end{equation}
 for $\xi \in cr_2Gr(\Z,-)(G,H)$.
So by Proposition \ref{T11crId},
\begin{equation}\label{TcrGrZZ} T_{11}cr_2(Gr(\Z,-))(\Z, \Z) \simeq \Z;
\end{equation}
and $T_{11}cr_2(Gr(\Z,-))(\Z, \Z)$ is generated by the commutator $t_{11}[ i_2, i_1]$.  We deduce that  condition $(T3)$ means that $H_2$ is determined by $Square(M)$ and satisfies
$$ PH_2=0 \quad\mbox{and}\quad H_2(t_{11}\gamma \otimes \bar{a}) = 0$$
by $(T1)$. However, starting out with the square group  
$Square(M)$, condition $(T3)$ a priori only gives rise to a  
map $$\tilde{H_2} : T_{11}cr_2(Gr(\Z,-))(\Z, \Z) \otimes_{\Z} coker(P) \to M_{ee};$$
in order to check that it factors through the tensor product over $\Lambda$ we must first consider the action of $\Lambda$ on $M_e$.

Let $n: \Z \to \Z$ be the homomorphism such that $1 \mapsto n$. Consider condition $(T5)$. By (\ref{TcrGrZZ})
we have $\gamma =k[ i_2, i_1]$, $k\in\Z$, whence $\nabla^2\gamma =k[Id,Id]=0$. As we know that $PH_2=0$, condition $(T5)$ is equivalent to the relation
\begin{equation} \label{eq-gr-1}
([n] \bullet [m]) a=[n] a+[m] a+P((\overline{[n]} \otimes \overline{[m]})H_1(a)).
\end{equation}
which by induction is equivalent to
\begin{equation} \label{[n]a}
[n]a = na + {n \choose 2}PH_1(a).
\end{equation}
So $(T5)$ means that the action of $\Lambda$ on $M_e$ is determined by $Square(M)$, via (\ref{[n]a}); the property to be an action is a formal consequence of the identity
$${nm\choose 2} = m{n \choose 2}+n{m \choose 2}+2 {n \choose 2}{m \choose 2}$$
for $n,m\in\Z$. 

By (\ref{[n]a}), $[n] \overline{a}=n \overline{a}$ in $coker(P)$. On the other hand, for $\gamma \in cr_2(Gr(\Z, \Z \mid \Z))$ we have:
$$\gamma([n +m])=\gamma([n] \bullet [m])=\gamma([n]) \bullet \gamma([m]).$$
Since $\gamma([1])=\gamma$ we deduce that $\gamma([n])=(\gamma([1]))^{\bullet n}=\gamma^{\bullet n}$. Hence the right action of $\Lambda$ on $T_{11}cr_2(Gr(\Z,-))(\Z, \Z)$  is given by 
$$t_{11}(\gamma). [n]=t_{11}(\gamma [n])=t_{11}(\gamma^{\bullet n})=n  t_{11}(\gamma).$$
We deduce that:
$$T_{11}cr_2(Gr(\Z,-))(\Z, \Z) \otimes_{\Lambda} coker(P) = T_{11}cr_2(Gr(\Z,-))(\Z, \Z) \otimes_{\Z} coker(P)\simeq coker(P).$$ 
So finally we obtain that $H_2=\tilde{H}_2$ is welldefined.

The fact that $P$ is $\Lambda$-equivariant is equivalent to condition $(T1)$ by (\ref{[n]a}).

Condition $(T4)$ is trivially satisfied
since $PH_2=0$ and $\gamma \bullet \tau\gamma = (k[i_2,i_1]) \bullet (k[i_1,i_2])=0$.

So it remains to show that condition $(T6)$ is a consequence of the others. For this we need the following lemma:

\begin{lm} \label{H-P}
Let $G$ and $H$ be two groups. Then for $g\in G$, $h\in H$ and $n\in \Z$, the following identity holds in $T_{11}cr_2(Id_{Gr})(G,H)$:
$$t_{11}\Big(((i_1g)(i_2h))^n (i_2h)^{-n}(i_1g)^{-n}\Big)
= t_{11}([i_2h,i_1g]^{{n \choose 2}}).$$
\end{lm}

\begin{proof}
Consider the following diagram where $\gamma_3(G)=[[G,G], G]$, the isomorphism $\Gamma_{11}'$ is defined in the proof of Proposition \ref{T11crId}, $p$ is the canonical projection, and  the map $c$ is given by
$c(\bar{g}\otimes \bar{h}) = [i_1g,i_2h]$.

$$\xymatrix{
cr_2(Id_{Gr})(G,H) \ar@{->>}[d]^-{t_{11}} \ar@{^{(}->}[rr]_-{inc} & & G \vee H  \ar@{->>}[d]^-{p}\\
T_{11} cr_2(Id_{Gr})(G,H) \ar[r]_-{\Gamma'_{11}} & G^{ab} \otimes H^{ab} \ar[r]_-{c} & (G \vee H)/\gamma_3(G \vee H)
}$$

The diagram commutes as is easily checked on the canonical generators
$t_{11}[i_1g,i_2h]$ of  $T_{11}cr_2(Id_{Gr})(G,H)$, see Proposition \ref{T11crId}. Thus
\begin{eqnarray*} 
(c\Gamma_{11}')t_{11}\Big((( i_1g)(i_2h))^n (i_2h)^{-n}(i_1g)^{-n}\Big) &=& p\Big((( i_1g) (i_2h))^n (i_2h)^{-n}(i_1g)^{-n}\Big)\\
&=& p([i_2h,i_1g]^{{n \choose 2}})\ \mbox{by the Hall-Petrescu formula}\\
&=&(c\Gamma_{11}')t_{11}([i_2h,i_1g]^{{n \choose 2}}).
\end{eqnarray*}
But the map $c$ is injective, see \cite[Proposition $1.2$]{JP}.
\end{proof}

This implies that 
$$ t_{11}(h([n])) = {n \choose 2}t_{11}[i_2,i_1];$$
in fact,
\begin{eqnarray*}
T_{11}(cr_2\nu)(t_{11}(h([n])) &=& 
t_{11} ( h([n])(1))\ \mbox{by (\ref{T11crnu})} \\
&=& t_{11}\Big( ((i_1\bullet i_2)[n]) \bullet (i_2[n])^{-1}  \bullet (i_1[n])^{-1})(1) \Big)\\
&=& t_{11}\Big( ((i_1\bullet i_2)(n))    i_2(n)^{-1}    i_1(n)^{-1}  \Big) \\
&=&  t_{11}\Big( ((i_1\bullet i_2)(1))^n    i_2(1)^{-n}    i_1(1)^{-n}  \Big)\\
&=& t_{11}( ((i_1(1) i_2(1))^n    i_2(1)^{-n}    i_1(1)^{-n}) \\
&=& t_{11}([i_2(1),i_1(1)]^{{n \choose 2}}) \mathrm{\ by\ Lemma\ \ref{H-P}} \\
&=& t_{11}([i_2,i_1]^{\bullet{n \choose 2}}(1)) \\
&=& T_{11}(cr_2\nu)(t_{11}([i_2,i_1]^{\bullet{n \choose 2}})) \\
&=& T_{11}(cr_2\nu)({n \choose 2}t_{11}([i_2,i_1]  )). 
\end{eqnarray*}
Thus
\begin{eqnarray*}
H_2(t_{11}h([n]) \otimes \overline{a}) + n^2 H_1(a)  
&=& {n \choose 2} H_2(t_{11} [i_2,i_1] \otimes \overline{a}) + n^2 H_1(a)\\
&=&  {n \choose 2}(H_1PH_1(a)-2H_1(a))+ n^2 H_1(a) \mathrm{\ by\ condition}\  (T3)\\
&=&  {n \choose 2}H_1PH_1(a)+nH_1(a)\\
&=& H_1([n]a) \mathrm{\ by\ (\ref{[n]a})},
\end{eqnarray*}
as desired.

\end{proof}

\begin{rem}
In the definition of square group given in  \cite{Baues-Pira} the authors consider the map:
$\Delta(a)=HPH(a)+H(2a)-4H(a)$. When $H$ is linear we have $\Delta =HPH -2H$. So the map $\Delta$ corresponds to our map $H_2$ according to condition $(T3)$.
\end{rem}

\begin{rem}
The case of the theory of free groups of finite rank is very simple compared to  a general theory of cogroups since a quadratic $Gr$-module relative to $\Z$ must satisfy only the single condition $(T1)$ instead of the six conditions $(T1)$ - $(T6)$ in the general situation.
\end{rem}

\small{\centerline{\sc\bf Acknowledgement:}
We are indebted to Teimuraz Pirashvili for suggesting the alternative approach of quadratic functors on theories, based on the Yoneda lemma, which became chapter 4 of the present version of the paper.

En outre, le premier auteur tient \`a remercier vivement son ami Jean Wielchuda pour son accueil chaleureux et son excellente cuisine lors de ses trois s\'ejours \`a Strasbourg.}


\bibliographystyle{amsplain}
\bibliography{biblio-poly}

\def\cprime{$'$}
\providecommand{\bysame}{\leavevmode\hbox to3em{\hrulefill}\thinspace}
\providecommand{\MR}{\relax\ifhmode\unskip\space\fi MR }
\providecommand{\MRhref}[2]{%
  \href{http://www.ams.org/mathscinet-getitem?mr=#1}{#2}
}
\providecommand{\href}[2]{#2}
\begin{thebibliography}{10}

\bibitem{Baues}
Hans-Joachim Baues, \emph{Quadratic functors and metastable homotopy}, J. Pure
  Appl. Algebra \textbf{91} (1994), no.~1-3, 49--107. \MR{MR1255923
  (94j:55022)}

\bibitem{BDFP}
Hans-Joachim Baues, Winfried Dreckmann, Vincent Franjou, and Teimuraz
  Pirashvili, \emph{Foncteurs polynomiaux et foncteurs de {M}ackey non
  lin\'eaires}, Bull. Soc. Math. France \textbf{129} (2001), no.~2, 237--257.
  \MR{MR1871297 (2002j:18004)}

\bibitem{BJP}
Hans-Joachim Baues, M.~Jibladze, and T.~Pirashvili, \emph{Quadratic algebra of
  square groups}, Adv. Math. \textbf{217} (2008), no.~3, 1236--1300.
  \MR{MR2383899}

\bibitem{Baues-Pira}
Hans-Joachim Baues and Teimuraz Pirashvili, \emph{Quadratic endofunctors of the
  category of groups}, Adv. Math. \textbf{141} (1999), no.~1, 167--206.
  \MR{MR1667150 (2000b:20073)}

\bibitem{Bor2}
Francis Borceux, \emph{Handbook of categorical algebra. 2}, Encyclopedia of
  Mathematics and its Applications, vol.~51, Cambridge University Press,
  Cambridge, 1994, Categories and structures. \MR{MR1313497 (96g:18001b)}

\bibitem{BB}
Francis Borceux and Dominique Bourn, \emph{Mal'cev, protomodular, homological
  and semi-abelian categories}, Mathematics and its Applications, vol. 566,
  Kluwer Academic Publishers, Dordrecht, 2004. \MR{MR2044291 (2005e:18001)}

\bibitem{Carboni}
A.~Carboni, G.~M. Kelly, and M.~C. Pedicchio, \emph{Some remarks on
  {M}al\cprime tsev and {G}oursat categories}, Appl. Categ. Structures
  \textbf{1} (1993), no.~4, 385--421. \MR{MR1268510 (95c:18003)}

\bibitem{EML}
Samuel Eilenberg and Saunders Mac~Lane, \emph{On the groups {$H(\Pi,n)$}. {II}.
  {M}ethods of computation}, Ann. of Math. (2) \textbf{60} (1954), 49--139.
  \MR{MR0065162 (16,391a)}

\bibitem{Benoit}
Benoit Fresse, \emph{Modules over operads and functors}, Lecture Notes in
  Mathematics, vol. 1967, Springer-Verlag, Berlin, 2009. \MR{MR2494775}

\bibitem{Manfred-multicalculus}
Manfred Hartl, \emph{Calculus of functors in several variables}, (in
  preparation).

\bibitem{Q3}
\bysame, \emph{Some successive quotients of group ring filtrations induced by
  {$N$}-series}, Comm. Algebra \textbf{23} (1995), no.~10, 3831--3853.
  \MR{MR1348267 (96j:20009)}

\bibitem{QMaps}
\bysame, \emph{Quadratic maps between groups}, Georgian Math. J. \textbf{16}
  (2009), no.~1, 55--74. \MR{MR2527615}

\bibitem{JP}
Mamuka Jibladze and Teimuraz Pirashvili, \emph{Quadratic envelope of the
  category of class two nilpotent groups}, Georgian Math. J. \textbf{13}
  (2006), no.~4, 693--722. \MR{MR2309253 (2007k:20114)}

\bibitem{JMCahiers1}
B.~Johnson and R.~McCarthy, \emph{A classification of degree {$n$} functors.
  {I}}, Cah. Topol. G\'eom. Diff\'er. Cat\'eg. \textbf{44} (2003), no.~1,
  2--38. \MR{MR1961524 (2004b:18016)}

\bibitem{JMCahiers2}
\bysame, \emph{A classification of degree {$n$} functors. {II}}, Cah. Topol.
  G\'eom. Diff\'er. Cat\'eg. \textbf{44} (2003), no.~3, 163--216. \MR{MR2003579
  (2004e:18015)}

\bibitem{MKS}
Wilhelm Magnus, Abraham Karrass, and Donald Solitar, \emph{Combinatorial group
  theory}, second ed., Dover Publications Inc., Mineola, NY, 2004,
  Presentations of groups in terms of generators and relations. \MR{MR2109550
  (2005h:20052)}

\bibitem{Pa}
I.~B.~S. Passi, \emph{Polynomial maps on groups}, J. Algebra \textbf{9} (1968),
  121--151. \MR{MR0231915 (38 \#241)}

\bibitem{Pira-russe}
T.~I. Pirashvili, \emph{Polynomial functors}, Trudy Tbiliss. Mat. Inst.
  Razmadze Akad. Nauk Gruzin. SSR \textbf{91} (1988), 55--66. \MR{MR1029007
  (91f:18004)}

\bibitem{Pira}
Teimuraz Pirashvili, \emph{Polynomial approximation of {${\rm Ext}$} and {${\rm
  Tor}$} groups in functor categories}, Comm. Algebra \textbf{21} (1993),
  no.~5, 1705--1719. \MR{MR1213983 (94d:18020)}

\bibitem{Pira-Dold}
\bysame, \emph{Dold-{K}an type theorem for {$\Gamma$}-groups}, Math. Ann.
  \textbf{318} (2000), no.~2, 277--298. \MR{MR1795563 (2001i:20112)}

\bibitem{Popescu}
N.~Popescu, \emph{Abelian categories with applications to rings and modules},
  Academic Press, London, 1973, London Mathematical Society Monographs, No. 3.
  \MR{MR0340375 (49 \#5130)}

\end{thebibliography}

\end{document}